\documentclass[10pt,a4paper,oneside]{article}

\usepackage{subfigure}
\usepackage[titletoc, title]{appendix}
\usepackage{mathrsfs}
\usepackage{paralist}
\usepackage{graphics}
\usepackage{epsfig}
\usepackage{multirow}
\usepackage{graphicx}
\usepackage{epstopdf}
\usepackage{bm,amsmath,amsbsy,amsthm,amssymb,amsfonts}
\usepackage{etoolbox}
\usepackage{diagbox}
\usepackage[colorlinks,linkcolor=blue,citecolor=blue]{hyperref}
\usepackage[capitalise, nosort]{cleveref}
\usepackage{syntonly,indentfirst,mathrsfs,latexsym,float}

\usepackage{url}
\usepackage{color}

\usepackage{graphicx}
\usepackage{amsmath,amssymb}
\allowdisplaybreaks
\usepackage{float}
\usepackage{graphicx}
\usepackage{appendix}
\usepackage{listings}
\usepackage{subfigure}
\usepackage{cite}
\usepackage{bm}
\usepackage{makecell}
\usepackage{multirow}
\usepackage{algorithm}
\usepackage{algorithmicx}
\usepackage{algpseudocode}
\usepackage[T1]{fontenc}
\usepackage[utf8]{inputenc}
\usepackage{authblk}
\usepackage{array}
\usepackage{color}
\usepackage{multirow}
\usepackage{mathrsfs}
\usepackage{type1cm}
\usepackage{wasysym}

\usepackage{booktabs,makecell,multirow}
\usepackage[capitalise,nosort]{cleveref}
\usepackage{cases,color}

\numberwithin{equation}{section}
\newtheorem{myDef}{Definition}[section]
\newtheorem{myTheo}{Theorem}[section]
\newtheorem{myLem}{Lemma}[section]

\usepackage{indentfirst}
\usepackage{stmaryrd}
\usepackage{subfigure}
\usepackage[utf8]{inputenc}
\usepackage{pgf,tikz}
\usepackage{threeparttable}

\usetikzlibrary{arrows}
\newcommand{\tabcaption}{\def\@captype{table}\caption}

\newtheorem{rem}{Remark}[section]

\newtheorem{exam}{Example}[section]

\newcommand{\snmii}[1]
{
  \left\vert\kern-0.25ex
  \left\vert\kern-0.25ex
  \left\vert
  #1
  \right\vert\kern-0.25ex
  \right\vert\kern-0.25ex
  \right\vert
}

\numberwithin{equation}{section}

\date{}
\title{{\Large \bf }}

\begin{document}
\title
{
  \Large\bf Robust globally divergence-free Weak Galerkin finite element method for   incompressible Magnetohydrodynamics flow
  \thanks
  {
    This work was supported in part by National Natural Science Foundation
    of China  (12171340) and National Key R\&D Program of China (2020YFA0714000).
  }
}
\author
{
  Min Zhang \thanks{School of Mathematics, Sichuan University, Chengdu 610064, China. Email: 894858786@qq.com }, \
  Tong Zhang  \thanks{School of Mathematics, YanTai University, Yantai 264005, Chine. Email: tzhang@ytu.edu.cn }, \
    Xiaoping Xie \thanks{School of Mathematics, Sichuan University, Chengdu 610064, China.
  Email: xpxie@scu.edu.cn},

  }
\date{}
\maketitle
\begin{abstract}
This paper develops a weak Galerkin (WG) finite element method of arbitrary order for the steady incompressible Magnetohydrodynamics
equations. The WG scheme uses   piecewise polynomials of degrees $k(k\geq 1),k,k-1$ and $k-1$ respectively for the approximations of  the velocity,  the magnetic field, the  pressure, and the  magnetic pseudo-pressure in the interior of elements, and uses piecewise polynomials of degree $k $  for their numerical traces on the interfaces of elements. The method is shown to yield globally divergence-free approximations of the velocity and magnetic fields.   We give existence and uniqueness results  for  the discrete scheme  and derive  optimal a priori error estimates. We also present a convergent  linearized  iterative algorithm.  Numerical experiments are provided  to verify the  obtained theoretical results.


\end{abstract}

\section{Introduction}

Magnetohydrodynamics (MHD) equations describe the  basic physics laws  of electrically conducting fluid flow interacting with magnetic fields,
and are widely used in engineering areas \cite{Davidson2001, GLL2006, Muller2001, Shercliff1965, Walker1980}
such as  magnetic propulsion devices,   optical modulation and switch,  continuous metal casting,  semi-conductor manufacture, and nuclear reactor technology.  This paper is to consider a finite element analysis of  a steady incompressible MHD flow model.

The  incompressible MHD flow  is described by
 a coupling  system of   Navier-Stokes equations and Maxwell equations.
Some early research   on the finite element analysis of   MHD    can be found in  \cite{GMP1991, Peterson1988,Winowich1983}. In particular, in \cite{GMP1991} Gunzburger et al.  considered   a steady  incompressible MHD model  in three dimensions, showed the existence and uniqueness of  a weak solution, and proved an optimal
estimate for  a mixed finite element discretization.  In recent twenty years there have   developed  many finite element methods for the incompressible MHD equations; see, e.g.
\cite{G2000,WLFH2017,ZHZ2014,ZMZ2016,
GM2003, 
GLSW2010,S2004,SMF2019,SY2013,W2000,
HMX2017, HLMZ2018} for steady models
 and   \cite{GQ2019,H2015,P2008,YH2018, DLM2020,DH2018,ZYB2018} for  unsteady models.

There are two divergence
constraints in the incompressible MHD equations, i.e. the velocity and magnetic fields are both divergence-free, which correspond to the conservation of mass and magnetic flux, respectively.
How to obtain   exactly divergence-free   approximations  is an important issue in   numerically solving the related problems,  since numerical methods with poor  conservation may lead to instabilities \cite{T2000, JLMNR2017, ABLR2010, L2009, OR2004,BB1980,JWP1996 }.
In particular, for   incompressible fluid flows the exactly divergence-free discretizations automatically lead to pressure-robustness in the sense that the velocity approximation  error   is independent of    the pressure approximation \cite{LM2016,JLMNR2017}.
We refer to \cite{CKS2007,CFX2016,ZCX2017,HX2019, HLX2019 ,XZ2010, Chen-Xie2023} for some divergence-free finite element methods for  the incompressible fluid flows, and to \cite{CLS2004, BLS2007, HZ2012} for several   divergence-free finite element methods for Maxwell equations.

 For the  incompressible MHD equations, there are considerable efforts devoted to  divergence-free finite element approaches \cite{LXY2011,LX2012,GLSW2010,HMX2017,LZZ2021,HLMZ2018}.
  Greif et al.
  \cite{GLSW2010} proposed a  mixed  interior-penalty discontinuous Galerkin (DG) method with the exactly divergence-free velocity, where the velocity field is discretized by   $H(\text{div}; \Omega)$-conforming Brezzi-Douglas-Marini   elements,  the pressure by fully discontinuous finite elements, and  the magnetic field   by $H(\text{curl}; \Omega)$-conforming Nedelec elements.
Li et al.  \cite{LXY2011,LX2012} and Hu et al. \cite{HMX2017} developed    central DG methods and  stable finite element methods  with the exactly divergence-free magnetic field, respectively.
Hiptmair et al. \cite{HLMZ2018} developed a   mixed DG method with  the exactly divergence-free velocity and  magnetic field for three-dimensional transient incompressible magnetohydrodynamic equations, where   the velocity and pressure approximations are as same as those in  \cite{GLSW2010}, and the divergence-free property of magnetic field  is realized by means of a magnetic vector potential.   Li et al. \cite{LZZ2021} presented a constrained transport finite element method  with  the exactly divergence-free velocity and  magnetic field for  three-dimensional incompressible resistive MHD equations, by following the same ideas as in \cite{CKS2007,GLSW2010}.

This paper is to develop an arbitrary order weak Galerkin (WG) scheme with    exactly divergence-free  velocity and  magnetic approximations for the following   steady incompressible MHD equations:
 \begin{eqnarray}
\label{mhd1*}
-\frac{1}{H^{2}_a}\Delta \mathbf{u}+\frac{1}{N}\nabla \cdot(\mathbf{u}\otimes \mathbf{u})+\nabla p - \frac{1}{R_{m}} \nabla \times\mathbf{B}\times \mathbf{B}=\mathbf{f}, &&  \text{in}\ \ \Omega,
\\
\label{mhd2*}
\nabla \cdot \mathbf{u}=0,&&  \text{in}\ \ \Omega,
\\
\label{mhd3*}
\frac{1}{R_{m}}\nabla\times \nabla\times\mathbf{B}-\nabla\times(\mathbf{u}\times \mathbf{B})+\nabla r=\mathbf{g}, && \text{in}\ \ \Omega,
\\
\label{mhd4*}
\nabla \cdot \mathbf{B}=0,&& \text{in}\ \ \Omega,
\end{eqnarray}
with the homogeneous boundary conditions
\begin{eqnarray}\label{mhd5*}
\mathbf{u}=\mathbf{0}, && \text{on}\ \ \partial\Omega,\\
\label{mhd6*}
\mathbf{B}\times\mathbf{n}=\mathbf{0},  &&  \text{on}\ \ \partial\Omega,\\
\label{mhd7*}
r=0, && \text{on}\ \ \partial\Omega.
\end{eqnarray}
Here $\Omega\in \mathbb{R}^{d}(d=2,3)$ is a polygonal or polyhedral domain,
$ \mathbf{u}=(u_1,u_2,...,u_d)^T$ is the velocity vector,
$p$  the pressure,
$\mathbf{B}=(B_1,B_2,...,B_d)^T$   the magnetic field,
and $r$   the magnetic pseudo-pressure. The right-hand side terms $\mathbf{f}, \mathbf{g}\in [L^2(\Omega)]^d $ 
are the forcing functions.   $H_a, \ N $ and $ R_m$ are the Hartmann number, the interaction parameter and the magnetic Reynolds number, respectively.

 The WG method was first proposed by Wang and Ye for  second-order elliptic problems \cite{Wang;Ye2013,Wang;Ye2014}.
Due to the use of weakly defined gradient/divergence operators over
functions with discontinuity, this method allows the use of totally discontinuous
functions on finite element partitions with arbitrary shape of polygons/polyhedra. It also has  the local elimination property, i.e. the unknowns defined in the interior of elements can be locally eliminated by using the numerical traces  defined on the interfaces of  elements. We refer to \cite{CFX2016,ZCX2017,HX2019, HLX2019,MWYZ2015, WWZZ2016,WY2016,ZX2017,ZZLW2018}  for some applications of the WG method to the incompressible fluid flows and Maxwell equations.

In this paper,  we consider the WG discretization of the MHD model \eqref{mhd1*}-\eqref{mhd7*}. The main features of our scheme
 are as follows.
\begin{itemize}
\item
We apply  piecewise polynomials of degrees $k(k\geq 1),k,k-1$ and $k-1$,
 respectively to approximate  the velocity, the magnetic field, the pressure, and the magnetic pseudo-pressure in the interior of elements, and apply piecewise polynomials of degree $k$to approximate their numerical traces
on the interfaces of elements.

\item The scheme is ``parameter-friendly" in the sense that it  does not
require the stabilization parameters to be ``sufficiently large".

\item The scheme gives   globally divergence-free approximations of the velocity and magnetic fields.

\item The unknowns of the velocity, the magnetic field, the
pressure and the magnetic pseudo-pressure   in the interior of elements can be locally eliminated so as to obtain a reduced discrete system of smaller size.

\item  The obtained error estimates are optimal.

\end{itemize}

The rest of this paper is arranged as follows.
Section 2 gives  weak formulations of the model problem.
Section 3 is devoted to the WG finite element scheme and some preliminary results.
In Section 4 we discuss the existence and uniqueness of the discrete solution.
   Section 5 derives a priori error estimates.
Section 6 shows the local elimination property and proposes
  an iteration algorithm for the nonlinear WG scheme.
Finally, We provide some numerical results.

\section{Weak problem}

\subsection{Notation}

For any bounded domain $D\in R^s(s=d,d-1)$, nonnegative  integer $m$ and real number $1\leq q<\infty$, let $W^{m,q}(D)$ and $W_0^{m,q}(D)$ be the usual Sobolev spaces defined on $D$ with norm $||\cdot||_{m,q, D}$ and semi-norm $|\cdot|_{m,q,D}$. In particular,   $H^m(D):=W^{m,2}(D)$ and $H^m_0(D):=W^{m,2}_0(D)$, with $||\cdot||_{m,D}:=||\cdot||_{m,2, D}$ and $|\cdot|_{m,D}:=|\cdot|_{m,2, D}$.   We use $(\cdot,\cdot)_{m,D}$ to denote the inner product of $H^m(D)$, with $(\cdot,\cdot)_D :=(\cdot,\cdot)_{0,D}$.
When $D = \Omega$, we set $||\cdot||_m := ||\cdot||_{m,\Omega},|\cdot|_m := |\cdot|_{m,\Omega}$,
  and $(\cdot,\cdot):= (\cdot,\cdot)_{\Omega}$.
Especially, when
$D \subset R^{d-1}$ we use $\langle\cdot,\cdot\rangle_D$ to replace $(\cdot,\cdot)_{D}$.
For any integer
$k \geq 0$, let $P_k(D)$ denote the set of all polynomials on $D$ with degree no more than $k$. We also need the following spaces:
\begin{eqnarray*}
&&L_0^2(\Omega):=\left\{
v\in L^2(\Omega): \ (v,1)=0
\right\}, \\
&&\mathbf{H}(\text{div},\Omega):=\left\{
\mathbf{v}\in  {L}^2(\mathcal{D})^d:\ \nabla\cdot \mathbf{v}\in L^2(\mathcal{D})
\right\},\\
&&\mathbf{H}(\text{curl};\Omega):=\{\mathbf{v}\in  {L}^2(\mathcal{D})^d:\ \nabla\times\mathbf{v}\in L^2(\Omega)^{2d-3}
\},\\
&&\mathbf{H}_0(\text{curl};\Omega):=\{\mathbf{v}\in \mathbf{H}(\text{curl};\Omega):\ \mathbf{n}\times\mathbf{v}=0 \ \ \text{on} \ \ \partial\Omega
\},
\end{eqnarray*}
where the cross product $\times$ of two vectors is defined as following: for $\mathbf{v}=(v_1,\cdots,v_d)^T, \mathbf{w}=(w_1,\cdots,w_d)^T$,
$$\mathbf{v}\times \mathbf{w}=\left\{
\begin{array}{ll}
v_1w_2-v_2w_1, & \text{if }d=2,\\
(v_2w_3-v_3w_2, v_3w_1-v_1w_3,v_1w_2-v_2w_1)^T, & \text{if }d=3.
\end{array} \right.
$$

Let $\mathcal{T}_h$ be a shape regular  partition of $\Omega$ into closed simplexes, and let
$\varepsilon_h$ be the set of all edges(faces) of all the elements in $\Omega$.
For any $K\in \mathcal{T}_h$, $e\in \varepsilon_h$, we denote by $h_K$ the diameter of $K$
and by $h_e$ the diameter of $e$.
Let $\mathbf{n}_K$  and $\mathbf{n}_e$ denote the outward unit normal vectors
along the boundary $\partial K$ and $e$, respectively. Sometimes we may abbreviate $\mathbf{n}_K$ as $\mathbf{n}$.
 We use  $\nabla_h$, $\nabla_h\cdot$ and $\nabla_h\times$ to denote  respectively  the operators of piecewise-defined gradient, divergence and
   curl with respect to the decomposition $\mathcal{T}_h$.

Throughout this paper, we use $\alpha\lesssim\beta $ to denote    $\alpha\leq C\beta $, where $C$ is a positive constant independent of the  mesh size $h$.

\subsection{Weak form}
For simplicity, we set
\begin{align*}
\mathbf{V}:=[H_0^1(\Omega)]^d,\ \ \
\mathbf{W}:=\mathbf{H}_0(\text{curl};\Omega). 
\end{align*}

For all $ \mathbf{u},\mathbf{v},\Phi \in \mathbf{V}, \ \mathbf{B},\mathbf{w} \in
  \mathbf{W},\ q\in L_0^2(\Omega),\theta\in H_0^1(\Omega),$
we define the following  bilinear and trilinear forms:
\begin{eqnarray*}
\begin{aligned}
a(\mathbf{u},\mathbf{v})&:=\frac{1}{H_a^2}(\nabla\mathbf{u},\nabla\mathbf{v}),\ \ \ \ \ \ \ \ \ \ \ \
\ \
b(\mathbf{v},q):=(q,\nabla\cdot\mathbf{v}),\\
\tilde{a}(\mathbf{B},\mathbf{w})
&:=\frac{1}{R_m^2}(\nabla\times\mathbf{B},\nabla\times\mathbf{w}),\ \ \ \ \
\tilde{b}(\mathbf{w},\theta):=\frac{1}{R_m}(\nabla \theta,\mathbf{w}),\\
c(\Phi;\mathbf{u},\mathbf{v})&:=
\frac{1}{N}\big\{
\frac{1}{2}(\nabla\cdot(\Phi\otimes\mathbf{u}),\mathbf{v})-
\frac{1}{2}(\nabla\cdot(\Phi\otimes\mathbf{v}),\mathbf{u})
\big\},\\
\tilde{c}(\mathbf{v};\mathbf{B},\mathbf{w})&:=
\frac{1}{R_m}(\nabla\times \mathbf{w}, \mathbf{v}\times\mathbf{B}).
\end{aligned}\end{eqnarray*}
It is easy to see that $c(\Phi;\mathbf{v},\mathbf{v})=0$.

Then the weak form of the problem \eqref{mhd1*}-\eqref{mhd7*} reads:
find $\mathbf{u}\in \mathbf{V},\mathbf{B}\in \mathbf{W},p\in L_0^2(\Omega), r\in H_0^1(\Omega)$ such that
\begin{eqnarray}\label{weak1*}
&&a(\mathbf{u},\mathbf{v})+\tilde{a}(\mathbf{B},\mathbf{w})
+b(\mathbf{u},q)-b(\mathbf{v},p)
+\tilde{b}(\mathbf{w},r)-\tilde{b}(\mathbf{B},\theta)
+c(\mathbf{u};\mathbf{u},\mathbf{v})
+\tilde{c}(\mathbf{v};\mathbf{B},\mathbf{B})
-\tilde{c}(\mathbf{u};\mathbf{B},\mathbf{w})
\nonumber\\
&=&(\mathbf{f},\mathbf{v})+\frac{1}{R_m}(\mathbf{g},\mathbf{w}), \quad \forall \mathbf{v}\in \mathbf{V},\mathbf{w}\in \mathbf{W},q\in L_0^2(\Omega), \theta\in H_0^1(\Omega).
\end{eqnarray}

\begin{rem}
From  \cite[Corollary 2.18]{S2004}   for $d=3$, if $\mathbf{f} \in H^{-1}(\Omega)^3$ and $\mathbf{g} \in L^2(\Omega)^3$, then the weak problem (\ref{weak1*})  admits at least one solution, and under a certain  smallness condition the solution is unique. 
\end{rem}

\section{Weak Galerkin  finite element method}

\subsection{WG scheme}
To establish the WG method for the problem (\ref{mhd1*})-(\ref{mhd7*}),
We firstly introduce, for integer $m\geq 0$, the discrete weak gradient operator $ \nabla_{w,m} $ , the discrete weak divergence operator $ \nabla_{w,m}\cdotp $
and the discrete weak curl operator $ \nabla_{w,m}\times $ as follows:
\begin{myDef}
For any $ \mathrm{v}\in \mathrm{V}(K):=\big\{  \mathrm{v}=\{\mathrm{v}_{o},\mathrm{v}_{b}\}:\mathrm{v}_{o}\in L^{2}(K),$$ \mathrm{v}_{b}\in H^{1/2}(\partial K) \big\}  $
and $ K\in \mathcal{T}_{h} $, the discrete weak gradient, $\nabla_{w,m,K} \mathrm{v}\in[P_m(K)]^d$, of $ \mathrm{v}$ on $K$ is defined by
\begin{equation}\label{weak-gra}
(\nabla_{w,m,K}\mathrm{v},\bm{\phi})_{K}=-(\mathrm{v}_{0},\nabla\cdotp\bm{\phi})_{K}+\left\langle \mathrm{v}_{b},\bm{\phi}\cdotp\mathbf{n}_{K}\right\rangle_{\partial K},\ \ \ \forall\bm{\phi}\in [\mathcal{P}_{m}(K)]^d.
\end{equation}
Then  the global discrete weak gradient operator $ \nabla_{w,m} $ is defined by
\begin{equation*}
\nabla_{w,m}|_{K}:=\nabla_{w,m,K},\ \ \ \forall K\in \mathcal{T}_{h}.
\end{equation*}
Moreover, for a vector $ \mathbf{v}=(\mathrm{v}_1, ...,\mathrm{v}_d)^T,$
the discrete weak gradient $\nabla_{w,m}\mathbf{v}$ is defined by
\begin{equation*}
\nabla_{w,m}\mathbf{v}:=(\nabla_{w,m}\mathrm{v}_1,...,\nabla_{w,m}\mathrm{v}_d)^T.
\end{equation*}
\end{myDef}

\begin{myDef}
For any $ \bm{\mathrm{w}}\in\bm{\mathrm{W}}(K):=\big\{  \bm{\mathrm{w}}=\{\bm{\mathrm{w}}_{o},\bm{\mathrm{w}}_{b}\}:\bm{\mathrm{w}}_{o}\in\left[  L^{2}(K)\right] ^{d},$  $ \bm{\mathrm{w}}_{b}\cdotp\mathbf{n}_{K}\in H^{-1/2}(\partial K) \big\}  $ and $ K\in \mathcal{T}_{h} $, the discrete weak divergence, $\nabla_{w,m,K}\cdot \bm{\mathrm{w}}\in \mathcal{P}_m(K)$, of  $\bm{\mathrm{w}}$ on $K$ is defined by
\begin{equation*}
(\nabla_{w,m,K}\cdotp \bm{\mathrm{w}},\phi)_{K}=-(\bm{\mathrm{w}}_{o},\nabla\phi)_{K}+\left\langle \bm{\mathrm{w}}_{b}\cdotp\mathbf{n}_K,\phi\right\rangle_{\partial K},\ \ \ \forall\phi\in \mathcal{P}_{m}(K).
\end{equation*}
Then  the global discrete weak divergence operator $ \nabla_{w,m}\cdotp $ is defined by
\begin{equation*}
\nabla_{w,m}\cdotp|_{K}:=\nabla_{w,m,K}\cdotp,\ \ \ \forall K\in \mathcal{T}_{h}.
\end{equation*}
Moreover, for a tensor $ \mathbf{\hat{w}}=(\bm{\mathrm{w}}_1,...,\bm{\mathrm{w}}_d),
$
the discrete weak divergence   $\nabla_{w,m}\cdot\mathbf{\hat{w}}$ is defined by
\begin{equation*}
\nabla_{w,m}\cdot\mathbf{\hat{w}}:=(\nabla_{w,m}\cdot\bm{\mathrm{w}}_1,...,\nabla_{w,m}\cdot\bm{\mathrm{w}}_d)^T.
\end{equation*}
\end{myDef}

\begin{myDef}
For any $ \bm{w}\in \bm{\mathcal{W}}(K):=\{\bm{w}= \{\bm{w}_{o},\bm{w}_{b}\}:\bm{w}_{o}\in\left[  L^{2}(K)\right] ^{d},$  $  \bm{w}_{b}\times \mathbf{n}_K  \in \left[H^{-1/2}(\partial K)\right]^{2d-3} \}  $ and $ K\in \mathcal{T}_{h} $, the discrete weak curl $\nabla_{w,m,K}\times \bm{w}\in \left[\mathcal{P}_m(K)\right]^{2d-3}$ on $K$ is defined by
\begin{equation}\label{weak-curl}
(\nabla_{w,m,K}\times \bm{w},\phi)_{K}=(\bm{w}_{o},\nabla\times\phi)_{K}+\left\langle  \bm{w}_{b}\times \mathbf{n}_K,\phi\right\rangle_{\partial K},\ \ \ \forall\phi\in \left[\mathcal{P}_{m}(K)\right]^{2d-3},
\end{equation}
where

Then  the global discrete weak curl operator $ \nabla_{w,m}\times $ is defined by
\begin{equation*}
\nabla_{w,m}\times|_{K}:=\nabla_{w,m,K}\times,\ \ \ \forall K\in \mathcal{T}_{h}.
\end{equation*}

\end{myDef}


For any integer $ k\geq 1 $,
 we introduce the following finite dimensional spaces:
\begin{equation*}\begin{aligned}
&\mathbf{V}_h=\{\mathbf{v}_h=\{\mathbf{v}_{ho},\mathbf{v}_{hb}\}:
\mathbf{v}_{ho}|_K\in[\mathcal{P}_k(K)]^d,\mathbf{v}_{hb}|_e\in[\mathcal{P}_k(e)]^d,
\forall K\in \mathcal{T}_h, \forall e\in \varepsilon_h\},\\
&\mathbf{V}^0_h=\{\mathbf{v}_h=\{\mathbf{v}_{ho},\mathbf{v}_{hb}\}
\in\mathbf{V}_h; \mathbf{v}_{hb}|_{\partial\Omega}=0\},\\
&\mathbf{W}^0_h=\{\mathbf{w}_h=\{\mathbf{w}_{ho},\mathbf{w}_{hb}\}
\in\mathbf{V}_h;   \mathbf{w}_{hb}\times \mathbf{n}|_{\partial\Omega}=0\},\\
&Q_h=\{q_h=\{q_{ho},q_{hb}\}:q_{ho}|_K\in \mathcal{P}_{k-1}(K),
q_{hb}|_e\in \mathcal{P}_k(e),\forall K\in \mathcal{T}_h, \forall e\in \varepsilon_h\},\\
&Q_h^0=\{q_h=\{q_{ho},q_{hb}\}\in Q_h:q_{ho}\in L^2_0(\Omega)\},\\
&R_h^0=\{\theta_h=\{\theta_{ho},\theta_{hb}\}\in Q_h;{\theta}_{hb}|_{\partial\Omega}=0\}.
\end{aligned}\end{equation*}

We also define  the  following bilinear forms and trilinear terms:
\begin{equation*}\begin{aligned}
&a_{h}(\mathbf{u}_h,\mathbf{v}_h):=\frac{1}{H_a^2}(\nabla_{w,k-1}\mathbf{u}_h,\nabla_{w,k-1}\mathbf{v}_h)
+s_{h}(\mathbf{u}_h,\mathbf{v}_h),\\
&\quad s_h(\mathbf{u}_h,\mathbf{v}_h):=\frac{1}{H_a^2}\langle \tau (  \mathbf{u}_{ho}-\mathbf{u}_{hb}),
   \mathbf{v}_{ho}-\mathbf{v}_{hb}\rangle_{\partial\mathcal{T}_h},\\
&\tilde{a}_{h}(\mathbf{B}_h,\mathbf{w}_h)
:=\frac{1}{R_m^2}(\nabla_{w,k-1}\times\mathbf{B}_h,\nabla_{w,k-1}\times\mathbf{w}_h)
+\tilde{s}_h(\mathbf{B}_h,\mathbf{w}_h),\\
&\quad \tilde{s}_h(\mathbf{B}_h,\mathbf{w}_h)
:=\frac{1}{R_m^2}\langle \tau (   \mathbf{B}_{ho}-\mathbf{B}_{hb})\times \mathbf{n},
 ( \mathbf{w}_{ho}-\mathbf{w}_{hb})\times \mathbf{n} \rangle_{\partial\mathcal{T}_h},\\
&b_{h}(\mathbf{v}_h,q_h):=(\nabla_{w,k}q_h,\mathbf{v}_{ho}), \ \ \ \ \ \
\tilde{b}_{h}(\mathbf{w}_h,\theta_h):=\frac{1}{R_m}(\nabla_{w,k}\theta_h,\mathbf{w}_{ho}), \\
&c_{h}(\Phi_h;\mathbf{u}_h,\mathbf{v}_h):=
\frac{1}{2N}\big( (\nabla_{w,k}\cdot\{\mathbf{u}_{ho}\otimes\Phi_{ho},\mathbf{u}_{hb}\otimes\Phi_{hb}\},\mathbf{v}_{ho})-
 (\nabla_{w,k}\cdot\{\mathbf{v}_{ho}\otimes\Phi_{ho},\mathbf{v}_{hb}\otimes\Phi_{hb}\},\mathbf{u}_{ho})\big),\\
 &\tilde{c}_{h}(\mathbf{v}_h;\mathbf{B}_h,\mathbf{w}_h):=
\frac{1}{R_m}(\nabla_{w,k}\times\mathbf{w}_h,\mathbf{v}_{ho}\times\mathbf{B}_{ho}),
\end{aligned}\end{equation*}
where
\begin{align*}
& \mathbf{u}_h=\{\mathbf{u}_{ho},\mathbf{u}_{hb}\}, \mathbf{v}_h=\{\mathbf{v}_{ho},\mathbf{v}_{hb}\},
\Phi_h=\{\Phi_{ho},\Phi_{hb}\}\in\mathbf{V}_h^0,\\
&
 \mathbf{B}_h=\{\mathbf{B}_{ho},\mathbf{B}_{hb}\},
\mathbf{w}_h=\{\mathbf{w}_{ho},\mathbf{w}_{hb}\}\in\mathbf{W}_h^0,\\
&q_h=\{q_{ho},q_{hb}\}\in Q_h^0, \quad
\theta_h=\{\theta_{ho},\theta_{hb}\}\in R_h^0,
\end{align*}
and the stabilization parameter $\tau$ in $s_{h}(\cdot;\cdot,\cdot)$ and  $\tilde{s}_{h}(\cdot;\cdot,\cdot)$ is given by  $$\tau|_{\partial K}=h_K^{-1}, \quad \forall K\in \mathcal{T}_h.$$
We easily see that
\begin{equation}\label{ch=0}
c_{h}(\Phi_h;\mathbf{v}_h,\mathbf{v}_h)=0, \quad \forall \Phi_h, \ \mathbf{v}_h.
\end{equation}

With the above definitions, the WG scheme  for the problem (\ref{mhd1*})-(\ref{mhd7*}) reads as follows:
find
$ \mathbf{u}_h=\{\mathbf{u}_{ho},\mathbf{u}_{hb}\}\in\mathbf{V}_h^0$,
$ \mathbf{B}_h=\{\mathbf{B}_{ho},\mathbf{B}_{hb}\}\in\mathbf{W}_h^0$,
$p_h=\{p_{ho},p_{hb}\}\in Q_h^0$, $r_h=\{r_{ho},r_{hb}\}\in R_h^0$,
such that
\begin{align}\label{scheme01*}
&a_{h}(\mathbf{u}_h,\mathbf{v}_h)
+\tilde{a}_{h}(\mathbf{B}_h,\mathbf{w}_h)
+b_{h}(\mathbf{v}_h,p_h)-b_{h}(\mathbf{u}_h,q_h)
+\tilde{b}_{h}(\mathbf{w}_h,r_h)-\tilde{b}_{h}(\mathbf{B}_h,\theta_h)\nonumber\\
&\quad+c_{h}(\mathbf{u}_h;\mathbf{u}_h,\mathbf{v}_h)
+\tilde{c}_{h}(\mathbf{v}_h;\mathbf{B}_h,\mathbf{B}_h)
-\tilde{c}_{h}(\mathbf{u}_h;\mathbf{B}_h,\mathbf{w}_h)\nonumber\\
=&(\mathbf{f},\mathbf{v}_{ho})
+\frac{1}{R_m}(\mathbf{g},\mathbf{w}_{ho}), \quad \forall \mathbf{v}_h\in \mathbf{V}_h^0, \mathbf{w}_h\in \mathbf{W}_h^0 ,q_h\in Q_h^0,\theta_h\in   R_h^0.
\end{align}

The  existence and uniqueness of the discrete solution to this scheme will be discussed in next section.

Notice that the scheme  \eqref{scheme01*} is equivalent to the   following system: find
$ \mathbf{u}_h \in\mathbf{V}_h^0$,
$ \mathbf{B}_h \in\mathbf{W}_h^0$,
$p_h \in Q_h^0$, $r_h \in R_h^0$,
such that
\begin{subequations}\label{scheme0101*}
\begin{align}
a_{h}(\mathbf{u}_h,\mathbf{v}_h)
 +b_{h}(\mathbf{v}_h,p_h)
+c_{h}(\mathbf{u}_h;\mathbf{u}_h,\mathbf{v}_h)
+\tilde{c}_{h}(\mathbf{v}_h;\mathbf{B}_h,\mathbf{B}_h)
&=(\mathbf{f},\mathbf{v}_{ho}), \quad & \forall \mathbf{v}_h\in \mathbf{V}_h^0,\label{scheme0101*-a}\\
 b_{h}(\mathbf{u}_h,q_h)&=0
 , \quad & \forall  q_h\in Q_h^0,\label{scheme0101*-b}\\
\tilde{a}_{h}(\mathbf{B}_h,\mathbf{w}_h)
+\tilde{b}_{h}(\mathbf{w}_h,r_h)-\tilde{c}_{h}(\mathbf{u}_h;\mathbf{B}_h,\mathbf{w}_h)
&= \frac{1}{R_m}(\mathbf{g},\mathbf{w}_{ho}), \quad &\forall  \mathbf{w}_h\in \mathbf{W}_h^0 ,\label{scheme0101*-c}\\
 \tilde{b}_{h}(\mathbf{B}_h,\theta_h)&=0, \quad &\forall  \theta_h\in   R_h^0.\label{scheme0101*-d}
\end{align}
\end{subequations}
In what follows we shall show that the two relations \eqref{scheme0101*-b} and \eqref{scheme0101*-d} yield   globally divergence-free approximations of the velocity  and  the   magnetic  field, i.e.
\begin{align}\label{divfree-velocity}
 &\mathbf{u}_{ho}\in \mathbf{H}(\text{div},\Omega), \nabla\cdot\mathbf{u}_{ho}=0,
 \\
  \label{divfree-magnet}
& \mathbf{B}_{ho}\in \mathbf{H}(\text{div},\Omega), \nabla\cdot\mathbf{B}_{ho}=0.
 \end{align}
In fact, define a function $\varrho_{hb}\in L^2( \varepsilon_h)$ as follows: for any  $e \in \varepsilon_h$,
\begin{equation*}
\varrho_{hb}|_e =\left\{ \begin{array}{ll}
-\left((\mathbf{u}_{ho}\cdot\mathbf{n}_e)|_{K_1}\right)|_e
-\left((\mathbf{u}_{ho}\cdot\mathbf{n}_e)|_{K_2}\right)|_e,  & \text{if } e= {K_1}\cap  {K_2}, K_1,K_2\in \mathcal{T}_h,\\
0, & \text{if } e\subset\partial\Omega.
\end{array}
\right.
\end{equation*}
Setting $C_0:=\frac{1}{|\Omega|}\int_\Omega\nabla_h\cdot\mathbf{u}_{ho}d\mathbf{x}$ and taking   $q_h=
\{q_{ho},q_{hb}\}$ in \eqref{scheme0101*-b} with $q_{ho}=
\nabla_h\cdot\mathbf{u}_{ho}-C_0,q_{hb}=\varrho_{hb}-C_0$, we   obtain
\begin{equation*}\begin{aligned}
0&=-b_{h}(\mathbf{u}_h,q_h)=-(\nabla_{w,k}q_h,\mathbf{u}_{ho})\\
& =\left(\nabla_h \cdot \boldsymbol{u}_{h o}, q_{h o}\right)-\sum_{T \in \mathcal{T}_h}\left\langle\boldsymbol{u}_{h o} \cdot \boldsymbol{n}_T, q_{h b}\right\rangle_{\partial T} \\
& =\left(\nabla_h \cdot \boldsymbol{u}_{h o}, \nabla_h \cdot \boldsymbol{u}_{h o}-C_0\right)-\sum_{T \in \mathcal{T}_h}\left\langle\boldsymbol{u}_{h o} \cdot \boldsymbol{n}_T, \varrho_{h b}-C_0\right\rangle_{\partial T} \\
& =\left(\nabla_h \cdot \boldsymbol{u}_{h o}, \nabla_h \cdot \boldsymbol{u}_{h o}\right)-\sum_{T \in \mathcal{T}}\left\langle\boldsymbol{u}_{h o} \cdot \boldsymbol{n}_T, \varrho_{h b}\right\rangle_{\partial T}\\
&=\|\nabla_h\cdot\mathbf{u}_{ho}\|_0^2+\sum_{e\in \varepsilon_h, e\nsubseteq\partial\Omega}
\|(\mathbf{u}_{ho}\cdot\mathbf{n}_e)|_{K_1}
+(\mathbf{u}_{ho}\cdot\mathbf{n}_e)|_{K_2}\|_{0,e}^2.
\end{aligned}\end{equation*}
This gives $\mathbf{u}_{ho}\in \mathbf{H}(\text{div},\Omega)$ and $\nabla_h\cdot\mathbf{u}_{ho}=\nabla\cdot\mathbf{u}_{ho}=0$, i.e.
\eqref{divfree-velocity} holds.

Similarly, we can get \eqref{divfree-magnet}.

As a result, we have
the following conclusion.
\begin{myTheo}\label{div-free}
The scheme (\ref{scheme01*}) yields the globally divergence-free approximations of the velocity  and the   magnetic  field  in the sense that both
 \eqref{divfree-velocity}
  and
 \eqref{divfree-magnet} hold.

\end{myTheo}

We introduce spaces
\begin{align*}
\mathbf{\bar{V}}_h&:=\{\mathbf{v}_h\in\mathbf{V}_h^0: \ b_{h}(\mathbf{v}_h,q_h)=0,\forall q_h\in Q_h^0    \},
\\
\mathbf{\bar{W}}_h&:=\{\mathbf{w}_h\in\mathbf{W}_h^0: \ \tilde{b}_{h}(\mathbf{w}_h,\theta_h)=0,\forall \theta_h\in R_h^0    \}.
\end{align*}
Thus, the solution $(\mathbf{u}_h,\mathbf{B}_h)\in\mathbf{ V}_h^0\times\mathbf{W}_h^0$ of the scheme  (\ref{scheme01*}) also solves the following discretization problem: find
 $(\mathbf{u}_h,\mathbf{B}_h)\in\mathbf{\bar{V}}_h\times\mathbf{\bar{W}}_h$ such that
\begin{align}\label{3c1*}
&a_{h}(\mathbf{u}_h,\mathbf{v}_h)
+\tilde{a}_{h}(\mathbf{B}_h,\mathbf{w}_h)
+c_{h}(\mathbf{u}_h;\mathbf{u}_h,\mathbf{v}_h)
+\tilde{c}_{h}(\mathbf{v}_h;\mathbf{B}_h,\mathbf{B}_h)
-\tilde{c}_{h}(\mathbf{u}_h;\mathbf{B}_h,\mathbf{w}_h)
\nonumber\\
=&(\mathbf{f},\mathbf{v}_{ho})
+\frac{1}{R_m}(\mathbf{g},\mathbf{w}_{ho}),
\ \ \ \ \ \ \
\forall(\mathbf{v}_h,\mathbf{w}_h)\in\mathbf{\bar{V}}_h\times\mathbf{\bar{W}}_h,
\end{align}

\begin{rem}
It is easy to see that
\begin{align}\label{barVsubset}
\mathbf{\bar{V}}_h&\subset \{\mathbf{v}_h\in\mathbf{V}_h^0: \  \mathbf{v}_{ho}\in \mathbf{H}(\text{div},\Omega), \nabla\cdot\mathbf{v}_{ho}=0    \},\\
\label{barW=}
\mathbf{\bar{W}}_h&=\{\mathbf{w}_h\in\mathbf{W}_h^0: \ \mathbf{w}_{ho}\in \mathbf{H}(\text{div},\Omega), \nabla\cdot\mathbf{w}_{ho}=0  \}.
\end{align}
\end{rem}

To discuss  the existence and uniqueness of the discrete solution of the scheme (\ref{scheme01*}) and derive  error estimates, we will give some preliminary results in next subsection.

\subsection{Preliminary results}

In view of  the definitions of weak gradient and   curl operators, the Green's formula, the Cauchy-Schwarz inequality, the trace inequality and the inverse inequality, we can easily derive the following inequalities on $\mathbf{V}_h $.

\begin{myLem}\label{lemma3*}
Let $0\leq k-1\leq m  \leq k$.
For any  $ K\in\mathcal{T}_{h} $ and    $ \mathbf{v}_{h}=\{\mathbf{v}_{ho},\mathbf{v}_{hb}\}, \mathbf{w}_{h}=\{\mathbf{w}_{ho},\mathbf{w}_{hb}\} \in \mathbf{V}_h$,  there hold \begin{subequations}\label{lemma41*}
\begin{align}
\|\nabla \mathbf{v}_{ho}\|_{0,K}&\lesssim\|\nabla_{w,m}\mathbf{v}_{h}\|_{0,K}
+h_K^{-\frac{1}{2}}\|   \mathbf{v}_{ho}-\mathbf{v}_{hb}\|_{0,\partial K},
\label{lemma411*:sub1}\\
\|\nabla_{w,m}\mathbf{v}_{h}\|_{0,K}&\lesssim \|\nabla \mathbf{v}_{ho}\|_{0,K}+h_K^{-\frac{1}{2}}\|  \mathbf{v}_{ho}-\mathbf{v}_{hb}\|_{0,\partial K},
\label{lemma411*:sub2}\\
\|\nabla \times \mathbf{w}_{ho}\|_{0,K}&\lesssim\|\nabla_{w,m}\times\mathbf{w}_{h}\|_{0,K}+h_K^{-\frac{1}{2}}
\| (  \mathbf{w}_{ho}-\mathbf{w}_{hb}) \|_{0,\partial K},
\label{lemma41*:sub1}\\
\|\nabla_{w,m}\times\mathbf{w}_{h}\|_{0,K}&\lesssim \|\nabla \times \mathbf{w}_{ho}\|_{0,K}+h_K^{-\frac{1}{2}}\| ( \mathbf{w}_{ho}-\mathbf{w}_{hb})\times\mathbf{n} \|_{0,\partial K}.
\label{lemma41*:sub2}
\end{align}
\end{subequations}
\end{myLem}

Introduce the following semi-norms  respectively on $ \mathbf{V}_{h}^{0}, \mathbf{W}_{h}^{0}$, $Q_h^0$
and $R^{0}_h$: 
\begin{equation*}
\begin{aligned}
&|||\mathbf{v}_h|||_V :=\left(\|\nabla_{w,k-1}\mathbf{v}_h\|_0^2+\|\tau^{\frac{1}{2}}
( \mathbf{v}_{ho}-\mathbf{v}_{hb})\|_{0,\partial\mathcal{T}_h}^2\right)^{1/2},\quad \forall \mathbf{v}_h \in \mathbf{V}_h^0,
\\
&|||\mathbf{w}_h|||_W:=\left(\|\nabla_{w,k-1}\times\mathbf{w}_h\|_0^2+\|\tau^{\frac{1}{2}}
( \mathbf{w}_{ho}-\mathbf{w}_{hb})\times \mathbf{n} \|_{0,\partial\mathcal{T}_h}^2\right)^{1/2}, \quad \forall
 \mathbf{w}_h\in \mathbf{W}_h^0,\\
&|||q_h|||_Q:=\left(\|q_{ho}\|_0^2+\sum_{K\in\mathcal{T}_h}h_K^2\|\nabla_{w,k}q_h\|_{0,K}^2\right)^{1/2},\quad \forall
 q_h\in Q_h^0,\\
&|||\theta_h|||_R:=\left(\|\theta_{ho}-\bar\theta_{ho}\|_0^2+
\sum_{K\in\mathcal{T}_h}h_K^2\|\nabla_{w,k}\theta_h\|_{0,K}^2\right)^{1/2}, \quad \forall  \theta_h\in  R_h^0,
\end{aligned}
\end{equation*}
where $\bar\theta_{ho}:=\frac{1}{|\Omega|}\int_\Omega\theta_{ho}  d\mathbf{x}$ denotes the mean value of $\theta_{ho}$, and we recall that $\tau|_{\partial K}=h_K^{-1}$.
It is  easy to see that $|||\cdot|||_V, |||\cdot|||_Q $ and $|||\cdot|||_R$
  are   norms  on $ \mathbf{V}_{h}^{0},   Q_h^0$
and $R^{0}_h$, respectively  (cf.  \cite{CFX2016}). As for the semi-norm $|||\cdot|||_W$, we have the following result. 

\begin{myLem}\label{lemma7w*}
 $|||\cdot|||_W$ is a norm on $\mathbf{\bar{W}}_h$.
 \end{myLem}
\begin{proof}
For any $\mathbf{w}_h\in \mathbf{\bar{W}}_h$, it suffices to show that $|||\mathbf{w}_h|||_W=0$ leads to $\mathbf{w}_h=0$. From the definition of $|||\cdot|||_W$ and the estimate \eqref{lemma41*:sub1}, we immediately get
$$\nabla\times\mathbf{w}_{ho}|_K=0, \quad    ( \mathbf{w}_{ho}-\mathbf{w}_{hb})\times \mathbf{n}|_{\partial K}=0, \quad  \forall K\in\mathcal{T}_h.$$
Hence we have
$$\mathbf{w}_{ho}\in \mathbf{H}(\text{curl};\Omega), \quad \nabla\times\mathbf{w}_{ho} =0, \ \text{and }
 \mathbf{w}_{ho}\times\mathbf{n}|_{\partial\Omega}=\mathbf{w}_{hb}\times\mathbf{n}|_{\partial\Omega}=0.$$
Then  there exists a potential function $\varphi$ such that
$\mathbf{w}_{ho}=\nabla\varphi$ in $\Omega$.

 On the other hand, from  \eqref{barW=} we also have
$$\mathbf{w}_{ho}\in \mathbf{H}(\text{div},\Omega), \quad  \nabla\cdot\mathbf{w}_{ho}|_\Omega=0. $$
As a result, we obtain $\Delta\varphi=0$  in $\Omega$.  As  the boundary condition $\nabla\varphi\times\mathbf{n}|_{\partial\Omega}=\mathbf{w}_{ho}\times\mathbf{n}|_{\partial\Omega}=0$   implies that $ \varphi$ is a constant on $\partial\Omega$, we   know that $ \varphi$ is a constant on $ \Omega$, which means that  $\mathbf{w}_{ho}=\nabla\varphi=0$. Finally, from the relation $( \mathbf{w}_{ho}-\mathbf{w}_{hb})\times \mathbf{n}|_{\partial K}=0$ for any  $ K\in\mathcal{T}_h$ it follows that $\mathbf{w}_{hb}=0$. This finishes the proof.
\end{proof}

\begin{myLem}\label{lemma7*}
There hold
\begin{subequations}\label{lemma42*}
\begin{align}
&\|\nabla_{h}\mathbf{v}_{ho}\|_{0}\lesssim |||\mathbf{v}_{h}|||_V,\ \ \forall \mathbf{v}_{h}\in \mathbf{V}^{0}_{h},\\
&\|\nabla_{h}\times \mathbf{w}_{ho}\|_{0}\lesssim |||\mathbf{w}_{h}|||_W,\ \ \forall \mathbf{w}_{h}\in \mathbf{W}^{0}_{h},
\label{lemma42*:sub2}
\end{align}
\end{subequations}
and
\begin{subequations}\label{lemma7*:3}
\begin{align}
&\|\mathbf{v}_{ho}\|_{0,q}\lesssim |||\mathbf{v}_{h}|||_V,\quad \forall \mathbf{v}_{h}\in \mathbf{V}^{0}_{h},
\end{align}
\end{subequations}
for $ 2\leq q<\infty $ when $d=2$, and for $ 2\leq q\leq 6 $ when $d=3$.

\end{myLem}
\begin{proof}
The first two inequalities follow from   Lemma \ref{lemma3*} directly and the third one comes from \cite[Lemma 3.5]{HX2019}.
\end{proof}

We  introduce the following mesh-dependent inner products and norms:
\begin{eqnarray*}
(u,v)_{\mathcal{T}_h}:=\sum_{K\in \mathcal{T}_h}(u,v)_K,\ \ \ \ \ \ \ \|u\|_{0,\mathcal{T}_h}:=\left(\sum_{K\in \mathcal{T}_h}\|u\|^2_{0,K}\right)^{1/2},\ \\
\langle u,v\rangle_{\partial\mathcal{T}_h}:=\sum_{K\in \mathcal{T}_h}\langle u,v\rangle_{\partial K},\ \ \ \ \|u\|_{0,\partial\mathcal{T}_h}:=\left(\sum_{K\in \mathcal{T}_h}\|u\|^2_{0,\partial K}\right)^{1/2}.
\end{eqnarray*}

\begin{myLem}\label{lemma12*}
There holds
\begin{eqnarray}\label{woL03}
\|\mathbf{w}_{ho}\|_{0,3,\Omega}\lesssim  |||\mathbf{w}_{h}|||_W, \quad \forall \mathbf{w}_h=\{\mathbf{w}_{ho},\mathbf{w}_{hb}\}\in \mathbf{\bar W}_h.
\end{eqnarray}
\end{myLem}
\begin{proof} By \eqref{barW=} 
and  the Green's formula  we   get
 \begin{eqnarray*}\label{woL03-1}
(\mathbf{w}_{ho},\nabla \phi)_{\mathcal{T}_h}=0, \ \ \forall\phi\in H_0^1(\Omega). 
\end{eqnarray*}
 Then, form \cite[Theorem 3.1]{QS2020} we have
\begin{align*}
\|\mathbf{w}_{ho}\|_{0,3,\Omega}
&\lesssim
\|\nabla_h\times\mathbf{w}_{ho}\|_{0,\mathcal{T}_h}+\|\tau^{\frac{1}{2}}
[\mathbf{w}_{ho}]\|_{0, \partial\mathcal{T}_h}\\
&=
\|\nabla_h\times\mathbf{w}_{ho}\|_{0,\mathcal{T}_h}+\|\tau^{\frac{1}{2}}
[\mathbf{w}_{ho}-\mathbf{w}_{hb}]\|_{0 ,\partial\mathcal{T}_h}\\
&\lesssim
\|\nabla_h\times\mathbf{w}_{ho}\|_{0,\mathcal{T}_h}+\|\tau^{\frac{1}{2}}
(\mathbf{w}_{ho}-\mathbf{w}_{hb})\times \mathbf{n}\|_{0, \partial\mathcal{T}_h}.
\end{align*}
This estimate plus \eqref{lemma42*:sub2} yields (\ref{woL03}).
\end{proof}

In light of the trace theorem, the inverse inequality and scaling arguments, we can get the following lemma  (cf. \cite{Shi;Wang2013,HX2019}).

\begin{myLem}\label{lemma2*}
For all $ K\in\mathcal{T}_{h}, \varphi\in H^{1}(K) $, and $ 1\leq q\leq \infty $, there holds
\begin{equation*}
\|\varphi\|_{0,q,\partial K}\lesssim h_K^{-\frac{1}{q}}\|\varphi\|_{0,q,K}+h_K^{1-\frac{1}{q}}|\varphi|_{1,q,K}.
\end{equation*}
In particular, for all $ \varphi\in \mathcal{P}_{k}(K) $,
\begin{equation*}
\|\varphi\|_{0,q,\partial K}\lesssim h_K^{-\frac{1}{q}}\|\varphi\|_{0,q,K}.
\end{equation*}
\end{myLem}

For any integer $ s\geq 0$, $ K\in \mathcal{T}_h $, $e\in \varepsilon_h$,
 let
$ Q^o_s:L^2(K)\rightarrow \mathcal{P}_s(K)$ and $ Q^b_s:L^2(e)\rightarrow \mathcal{P}_s(e)$ be the standard $L^2$ projection operators.
There vector/tensor analogues are denoted by
  $\mathbf{Q}^o_s$ and $\mathbf{Q}^b_s$, respectively.
\begin{myLem}\label{lemma1*}\cite{Shi;Wang2013}
For any $ K\in\mathcal{T}_{h} $, $ e\in \varepsilon_{h}$,
and 
$ 1\leq j\leq s+1 $, there hold
\begin{equation*}
\begin{aligned}
&\|\mathrm{v}-Q^{o}_{s}\mathrm{v}\|_{0,K}+h_K|\mathrm{v}-Q_{s}^{o}\mathrm{v}|_{1,K}
\lesssim h_K^{j}|\mathrm{v}|_{j,K},\ \forall \mathrm{v}\in H^{j}(K),
\\
&\|\mathrm{v}-Q^{o}_{s}\mathrm{v}\|_{0,\partial K}
+\|\mathrm{v}-Q^{b}_{s}\mathrm{v}\|_{0,\partial K}
\lesssim
h_K^{j-1/2}|\mathrm{v}|_{j,K},\ \forall \mathrm{v}\in H^{j}(K),
\\
&\|Q^{o}_{s}\mathrm{v}\|_{0,K}\leq \|\mathrm{v}\|_{0,K},\ \forall \mathrm{v}\in L^{2}(K),\\
&\|Q^{b}_{s}\mathrm{v}\|_{0,e}\leq \|\mathrm{v}\|_{0,e},\ \forall \mathrm{v}\in L^{2}(e).
\end{aligned}
\end{equation*}
\end{myLem}

For any
$K\in\mathcal{T}_h$, we introduce  the local Raviart-Thomas($\mathcal{RT}$) element space
\begin{equation*}
\mathbb{RT}_s(K)=[\mathcal{P}_s(K)]^d+\mathbf{x}\mathcal{P}_s(K)
\end{equation*}
and the $\mathcal{RT}$ projection operator  $\mathbf{P}_s^{\mathcal{RT}}: [H^1(K)]^d \rightarrow  \mathbb{RT}_s(K)$  (cf. \cite{BBDDFF2008}) defined by

\begin{subequations}\label{lemma91*}
\begin{align}
\langle\mathbf{P}_s^{\mathcal{RT}}\mathbf{v}\cdot\mathbf{n}_e,w\rangle_e
&=\langle\mathbf{v}\cdot\mathbf{n}_e,w\rangle_e,\ \ \forall w\in \mathcal{P}_s(e),e\in\partial K, \quad \text{ for }s\geq 0,
\label{lemma91*:sub1}\\
(\mathbf{P}_s^{\mathcal{RT}}\mathbf{v},\mathbf{w} )_K&=(\mathbf{v},\mathbf{w} )_K,\ \ \forall \mathbf{w} \in [\mathcal{P}_s(K)]^d, \quad  \text{ for } s\geq 1.
\label{lemma91*:sub2}
\end{align}
\end{subequations}
  Lemmas \ref{lemma8*}-\ref{lemma17*} give some properties of the $\mathcal{RT}$ element space and  the $\mathcal{RT}$ projection. 

\begin{myLem}\label{lemma8*}\cite{BBDDFF2008}
For any $\mathbf{v}_{ho}\in\mathbb{RT}_s(K), $ the relation $\nabla\cdot\mathbf{v}_{ho}|_K=0$ implies that $\mathbf{v}_{ho}\in [\mathcal{P}_s(K)]^d$.
\end{myLem}

\begin{myLem}\label{lemma9*}\cite{BBDDFF2008}
For any $K\in \mathcal{T}_h$ and $\mathbf{v}\in[H^1(K)]^d$, 
 the following properties   hold:
\begin{equation*}
(\nabla\cdot\mathbf{P}_s^{RT}\mathbf{v},\phi_h)_K=(\nabla\cdot\mathbf{v},\phi_h)_K, \ \
\forall \mathbf{v}\in[H^1(K)]^d,\phi_h\in \mathcal{P}_s(K), 
\end{equation*}
\begin{equation*}
\|\mathbf{v}-\mathbf{P}_s^{\mathcal{RT}}\mathbf{v}\|_{0,K}\lesssim h_K^j|\mathbf{v}|_{j,K},\ \ \
\forall 1\leq j\leq s+1,\ \ \forall \mathbf{v}\in[H^j(K)]^d.
\end{equation*}
\end{myLem}

By using   the triangle inequality, the inverse inequality, Lemma \ref{lemma1*} and
Lemma \ref{lemma9*} we can get more estimates for the $\mathcal{RT}$ projection (cf. \cite{HX2019}):
\begin{myLem}\label{lemma17*}
  For any $K\in \mathcal{T}_h$, $\mathbf{v}\in[H^j(K)]^d$ and $1\leq j\leq s+1$,
the following estimates hold:
\begin{align*}
|\mathbf{v}-\mathbf{P}_s^{\mathcal{RT}}\mathbf{v}|_{1,K}
&\lesssim
h_K^{j-1}|\mathbf{v}|_{j,K}, \\
|\mathbf{v}-\mathbf{P}_s^{\mathcal{RT}}\mathbf{v}|_{0,\partial K}
&\lesssim
h_K^{j-\frac{1}{2}}|\mathbf{v}|_{j,K},
\\
|\mathbf{v}-\mathbf{P}_s^{\mathcal{RT}}\mathbf{v}|_{0,3,K}&\lesssim h_K^{j-\frac{d}{6}}|\mathbf{v}|_{j,K},\\
|\mathbf{v}-\mathbf{P}_s^{\mathcal{RT}}\mathbf{v}|_{0,3,\partial K}&\lesssim h_K^{j-\frac{1}{3}-\frac{d}{6}}|\mathbf{v}|_{j,K}.
\end{align*}
\end{myLem}

We also have the following commutativity properties for the   $\mathcal{RT}$ projection, the $L^2$ projections and the discrete weak operators:
\begin{myLem}\label{lemma11*}
\cite{CFX2016,MWYZ2015} For $k\geq 1$, 
 there hold
\begin{eqnarray*}
&\nabla_{w,k}\{Q_{k-1}^oq,Q_k^bq\}=\mathbf{Q}_k^o(\nabla q),&
\forall q\in H^1(\Omega),\\
&\nabla_{w,k-1}\{\mathbf{P}_k^{\mathcal{RT}}\mathbf{v},\mathbf{Q}_k^b\mathbf{v}\}=\mathbf{Q}_{k-1}^o(\nabla\mathbf{v}), &
\forall\mathbf{v}\in[H^1(\Omega)]^d,\\
&\nabla_{w,k-1}\times\{\mathbf{Q}_k^o\mathbf{w},\mathbf{Q}_k^b\mathbf{w}\}=\mathbf{Q}_{k-1}^o(\nabla\times\mathbf{w}),
&\forall\mathbf{w}\in\mathbf{H}(\text{curl};\Omega).
\end{eqnarray*}
\end{myLem}

\section{ Existence and uniqueness of discrete solution }

\subsection{Stability conditions}
\begin{myLem}\label{lemma13*}
 For any $ \mathbf{u}_h, \mathbf{v}_h,\Phi_h \in\mathbf{V}_h^0$, and $\mathbf{B}_h, \mathbf{w}_h \in\mathbf{W}_h^0$, there hold the following stability conditions:
\begin{subequations}\label{Stability conditions*}
\begin{eqnarray}
&&a_{h}(\mathbf{u}_h,\mathbf{v}_h)\leq\frac{1}{H_a^2}|||\mathbf{u}_h|||_V|||\mathbf{v}_h|||_V,
\label{Stability conditions:sub1}\\
&&a_{h}(\mathbf{v}_h,\mathbf{v}_h)=\frac{1}{H_a^2}|||\mathbf{v}_h|||_V^2,
\label{Stability conditions:sub2}\\
&&\tilde{a}_{h}(\mathbf{B}_h,\mathbf{w}_h)
\leq\frac{1}{R_m^2}|||\mathbf{B}_h|||_W|||\mathbf{w}_h|||_W,
\label{Stability conditions:sub3}\\
&&\tilde{a}_{h}(\mathbf{w}_h,\mathbf{w}_h)
=\frac{1}{R_m^2}|||\mathbf{w}_h|||_W^2,
\label{Stability conditions:sub4}\\
&&c_{h}(\Phi_h;\mathbf{v}_h,\mathbf{v}_h)=0, \label{Stability conditions:sub4-1}\\
&&c_{h}(\Phi_h;\mathbf{u}_h,\mathbf{v}_h)\leq M_h |||\Phi_h|||_V|||\mathbf{u}_h|||_V|||\mathbf{v}_h|||_V \lesssim
|||\Phi_h|||_V|||\mathbf{u}_h|||_V|||\mathbf{v}_h|||_V,
\label{Stability conditions:sub5}\\
&&\tilde{c}_{h}(\mathbf{v}_h;\mathbf{B}_h,\mathbf{w}_h)\leq \tilde M_h |||\mathbf{B}_h|||_W|||\mathbf{w}_h|||_W|||\mathbf{v}_h|||_V
\lesssim
|||\mathbf{B}_h|||_W|||\mathbf{w}_h|||_W|||\mathbf{v}_h|||_V,
\label{Stability conditions:sub6}
\end{eqnarray}
\end{subequations}
where
\begin{eqnarray}
&&M_{h}:=\sup_{\mathbf{0}\neq\Phi_h,\mathbf{u}_h,\mathbf{v}_h\in \mathbf{\bar{V}}_h}
\frac{c_{h}(\Phi_h;\mathbf{u}_h,\mathbf{v}_h)}
{|||\Phi_h|||_V|||\mathbf{u}_h|||_V|||\mathbf{v}_h|||_V},\label{Mh}
\\
&&\tilde{M}_{h}:=\sup_{
\substack{
\mathbf{0}\neq\mathbf{B}_h,\Psi_h\in\mathbf{\bar{W}}_h,\ \
\mathbf{0}\neq\mathbf{v}_h\in \mathbf{\bar{V}}_h}
}
\frac{\tilde{c}_{h}(\mathbf{v}_h;\mathbf{B}_h,\Psi_h)}
{|||\Psi_h|||_W|||\mathbf{v}_h|||_V|||\mathbf{B}_h|||_W}.\label{tilde-Mh}
\end{eqnarray}

 \end{myLem}
\begin{proof}
From the definitions of $a_{h}(\cdot,\cdot)$, $\tilde{a}_{h}(\cdot,\cdot)$,
the Cauchy-Schwarz inequality and Lemmas \ref{lemma3*} we can easily get (\ref{Stability conditions:sub1})-(\ref{Stability conditions:sub4}). The relation (\ref{Stability conditions:sub4-1}) follows   from the definition of $c_{h}(\cdot;\cdot,\cdot)$, and the inequalities
$$c_{h}(\Phi_h;\mathbf{u}_h,\mathbf{v}_h)\leq M_h |||\Phi_h|||_V|||\mathbf{u}_h|||_V|||\mathbf{v}_h|||_V $$
and
$$\tilde{c}_{h}(\mathbf{v}_h;\mathbf{B}_h,\mathbf{w}_h)\leq \tilde M_h |||\mathbf{B}_h|||_W|||\mathbf{w}_h|||_W|||\mathbf{v}_h|||_V$$
follow from the definitions of $ {M}_{h}$  and $\tilde{M}_{h}$, respectively. The thing left is to prove
$c_{h}(\Phi_h;\mathbf{u}_h,\mathbf{v}_h)\lesssim |||\Phi_h|||_V|||\mathbf{u}_h|||_V|||\mathbf{v}_h|||_V $
and
$\tilde{c}_{h}(\mathbf{v}_h;\mathbf{B}_h,\mathbf{w}_h)\lesssim |||\mathbf{B}_h|||_W|||\mathbf{w}_h|||_W|||\mathbf{v}_h|||_V$.

For all $\mathbf{u}_h,\mathbf{v}_h\in\mathbf{V}_h^0$, using
the definitions of $c_{h}(\cdot;\cdot,\cdot)$ and $\nabla_{w,k}\cdot$ we have
\begin{align*}
2c_{h}(\Phi_h;\mathbf{u}_h,\mathbf{v}_h)&=
\frac{1}{N}\{
(\mathbf{v}_{ho}\otimes\Phi_{ho},\nabla_h\mathbf{u}_{ho})-
(\mathbf{u}_{ho}\otimes\Phi_{ho},\nabla_h\mathbf{v}_{ho})\\
&\ \ \ \
-\langle\mathbf{v}_{hb}\otimes\Phi_{hb}\mathbf{n},\mathbf{u}_{ho}\rangle_{\partial\mathcal{T}_h}
+\langle\mathbf{u}_{hb}\otimes\Phi_{hb}\mathbf{n},\mathbf{v}_{ho}\rangle_{\partial\mathcal{T}_h}\}
\\
&=\frac{1}{N}\{(\mathbf{v}_{ho}\otimes\Phi_{ho},\nabla_h\mathbf{u}_{ho})-
(\mathbf{u}_{ho}\otimes\Phi_{ho},\nabla_h\mathbf{v}_{ho})\\
&\ \ \ \ +\langle(\mathbf{u}_{ho}-\mathbf{u}_{hb})
\otimes(\Phi_{ho}-\Phi_{hb})
\mathbf{n},\mathbf{v}_{ho}\rangle_{\partial\mathcal{T}_h}
-\langle(\mathbf{u}_{ho}-\mathbf{u}_{hb})
\otimes\Phi_{ho}\mathbf{n},\mathbf{v}_{ho}\rangle_{\partial\mathcal{T}_h}\\
&\ \ \ \ -\langle(\mathbf{v}_{ho}-\mathbf{v}_{hb})
\otimes(\Phi_{ho}-\Phi_{hb})
\mathbf{n},\mathbf{u}_{ho}\rangle_{\partial\mathcal{T}_h}
+\langle(\mathbf{v}_{ho}-\mathbf{v}_{hb})
\otimes\Phi_{ho}\mathbf{n},\mathbf{u}_{ho}\rangle_{\partial\mathcal{T}_h}\}\\
&
=:\sum_{i=1}^5\mathcal{I}_i.
\end{align*}
Combining the H\"{o}lder's inequality and Lemma \ref{lemma7*} we obtain
\begin{align*}
|\mathcal{I}_1|&=\frac{1}{N}|(\mathbf{v}_{ho}\otimes\Phi_{ho},\nabla_h\mathbf{u}_{ho})-
(\mathbf{u}_{ho}\otimes\Phi_{ho},\nabla_h\mathbf{v}_{ho})|\\
&\leq\frac{1}{N}(
\|\mathbf{v}_{ho}\|_{0,4}\|\Phi_{ho}\|_{0,4}\|\nabla_h\mathbf{u}_{ho}\|_{0,2}+
\|\mathbf{u}_{ho}\|_{0,4}\|\Phi_{ho}\|_{0,4}\|\nabla_h\mathbf{v}_{ho}\|_{0,2})\\
&
\lesssim |||\Phi_h|||_V |||\mathbf{u}_h|||_V |||\mathbf{v}_h|||_V.
\end{align*}
Using the H\"{o}lder's inequality,  the inverse inequality, Lemma \ref{lemma2*} and Lemma \ref{lemma7*}, we have
\begin{align*}
|\mathcal{I}_2|&=\frac{1}{N}|\langle(\mathbf{u}_{ho}-\mathbf{u}_{hb})
\otimes(\Phi_{ho}-\Phi_{hb})
\mathbf{n},\mathbf{v}_{ho}\rangle_{\partial\mathcal{T}_h} |\\
&\leq\frac{1}{N}\sum_{K\in\mathcal{T}_h}\|\Phi_{ho}-\Phi_{hb}\|_{0,3,\partial K}
\|\mathbf{u}_{ho}-\mathbf{u}_{hb}\|_{0,2,\partial K}
\|\mathbf{v}_{ho}\|_{0,6,\partial K}\\
&\leq\frac{1}{N}\sum_{K\in\mathcal{T}_h}h_K^{-\frac{d-1}{6}}\|\Phi_{ho}-\Phi_{hb}\|_{0,2,\partial K}
\|\mathbf{u}_{ho}-\mathbf{u}_{hb}\|_{0,2,\partial K}
h_K^{-\frac{1}{6}}\|\mathbf{v}_{ho}\|_{0,6,K}\\
&\leq\frac{1}{N}\sum_{K\in\mathcal{T}_h}h_K^{-\frac{1}{2}}\|\Phi_{ho}-\Phi_{hb}\|_{0,2,\partial K}
h_K^{-\frac{1}{2}}\|\mathbf{u}_{ho}-\mathbf{u}_{hb}\|_{0,2,\partial K}
h_K^{1-\frac{d}{6}}\|\mathbf{v}_{ho}\|_{0,6,K}\\
&\lesssim |||\Phi_h|||_V |||\mathbf{u}_h|||_V |||\mathbf{v}_h|||_V,
\\
|\mathcal{I}_3|&=\frac{1}{N}|-\langle(\mathbf{u}_{ho}-\mathbf{u}_{hb})
\otimes\Phi_{ho}\mathbf{n},\mathbf{v}_{ho}\rangle_{\partial\mathcal{T}_h}|\\
&\leq\frac{1}{N}\sum_{K\in\mathcal{T}_h}\|\Phi_{ho}\|_{0,4,\partial K}
\|\mathbf{u}_{ho}-\mathbf{u}_{hb}\|_{0,2,\partial K}
\|\mathbf{v}_{ho}\|_{0,4,\partial K}\\
&\leq\frac{1}{N}\sum_{K\in\mathcal{T}_h}h_K^{-\frac{1}{4}}\|\Phi_{ho}\|_{0,4,K}
\|\mathbf{u}_{ho}-\mathbf{u}_{hb}\|_{0,2,\partial K}
h_K^{-\frac{1}{4}}\|\mathbf{v}_{ho}\|_{0,4,K}\\
&\leq\frac{1}{N}\sum_{K\in\mathcal{T}_h}\|\Phi_{ho}\|_{0,4,K}
h_K^{-\frac{1}{2}}\|\mathbf{u}_{ho}-\mathbf{u}_{hb}\|_{0,2,\partial K}
\|\mathbf{v}_{ho}\|_{0,4,K}\\
&\lesssim |||\Phi_h|||_V |||\mathbf{u}_h|||_V |||\mathbf{v}_h|||_V.
\end{align*}
Similarly, we can get
\begin{align*}
|\mathcal{I}_4|+|\mathcal{I}_5|\lesssim |||\Phi_h|||_V |||\mathbf{u}_h|||_V |||\mathbf{v}_h|||_V.
\end{align*}
Combining the above estimates yields the inequalities \eqref{Stability conditions:sub5}.

For any $\mathbf{B}_h, \mathbf{w}_h \in \mathbf{W}_h^0$, by the definition of $\nabla_{w,k}\times$ we have
\begin{align*}
\tilde{c}_{h}(\mathbf{v}_h;\mathbf{B}_h,\mathbf{w}_h)
=\frac{1}{R_m}(\nabla_h\times\mathbf{w}_{ho},\mathbf{v}_{ho}\times\mathbf{B}_{ho})
-\frac{1}{R_m}\langle(\mathbf{w}_{ho}-\mathbf{w}_{hb})
\times\mathbf{n},\mathbf{v}_{ho}\times\mathbf{B}_{ho}\rangle_{\partial\mathcal{T}_h}.
\end{align*}
Using the H\"{o}lder's inequality,   the inverse inequality, Lemmas \ref{lemma7*}, \ref{lemma12*}  and \ref{lemma2*}  again  gives
\begin{align*}
 | (\nabla_h\times\mathbf{w}_{ho},\mathbf{v}_{ho}\times\mathbf{B}_{ho})|
&\leq|\nabla_h\times\mathbf{w}_{ho}|_{0,2}|\mathbf{B}_{ho}|_{0,3}|\mathbf{v}_{ho}|_{0,6}\\
&\lesssim|||\mathbf{w}_{h}|||_W|||\mathbf{B}_{h}|||_W|||\mathbf{v}_{h}|||_V,
\\
 |\langle(\mathbf{w}_{ho}-\mathbf{w}_{hb})
\times\mathbf{n},\mathbf{v}_{ho}\times\mathbf{B}_{ho}\rangle_{\partial\mathcal{T}_h}|
&\leq \sum\limits_{K\in \mathcal{T}_h}|(\mathbf{w}_{ho}-\mathbf{w}_{hb})\times\mathbf{n}|_{0,2,\partial K}
|\mathbf{B}_{ho}|_{0,3,\partial K}
|\mathbf{v}_{ho}|_{0,6,\partial K}
\\
&\lesssim
\sum\limits_{K\in \mathcal{T}_h}|(\mathbf{w}_{ho}-\mathbf{w}_{hb})\times\mathbf{n}|_{0,2,\partial K}
h_K^{-\frac{1}{3}}|\mathbf{B}_{ho}|_{0,3,K}
h_K^{-\frac{1}{6}}|\mathbf{v}_{ho}|_{0,6,K}
\\
&\lesssim
\sum\limits_{K\in \mathcal{T}_h}h_K^{-\frac{1}{2}}|(\mathbf{w}_{ho}-\mathbf{w}_{hb})\times\mathbf{n}|_{0,2,\partial K}
|\mathbf{B}_{ho}|_{0,3,K}
|\mathbf{v}_{ho}|_{0,6,K}\\
&\lesssim
 |||\mathbf{w}_{h}|||_W|||\mathbf{B}_{h}|||_W|||\mathbf{v}_{h}|||_V.
\end{align*}
As a result, the desired inequalities \eqref{Stability conditions:sub6} follows.
\end{proof}

 We have the following inf-sup inequalities.
\begin{myLem}\label{lemma15*}
There hold
\begin{align}\label{inf-sup-bh}
&\sup_{\mathbf{v}_h\in \mathbf{V}_h^0}\frac{b_{h}(\mathbf{v}_h,q_h)}{|||\mathbf{v}_h|||_V}\gtrsim|||q_h|||_Q, \quad \forall  q_h\in   Q_h^0,\\
&\sup_{\mathbf{w}_h\in \mathbf{W}_h^0}\frac{\tilde{b}_{h}(\mathbf{w}_h,\theta_h)}{|||\mathbf{w}_h|||_W}\gtrsim|||\theta_h|||_R, \quad \forall  \theta_h \in   R_h^0.\label{inf-sup-tildebh}
\end{align}
\end{myLem}
\begin{proof}
The first inequality follows from \cite[Theorem 3.1]{CFX2016}.

For any $\theta_h=\{\theta_{ho},\theta_{hb}\} \in   R_h^0$, let $c_0:=\frac{1}{|\Omega|}\int_\Omega\theta_{ho}  d\mathbf{x}$ and $\tilde\theta_h:=\{\theta_{ho}-c_0,\theta_{hb}-c_0\}\in Q_h^0$. Then,
according to the definition of $\tilde{b}(\cdot,\cdot)$,   (\ref{lemma41*:sub2}), the relation $\nabla_{w,k}\{c_0,c_0\}=0$, $ \mathbf{W}_h^0\supset  \mathbf{V}_h^0$, the inequality $|||\mathbf{w}_h|||_W\leq |||\mathbf{w}_h|||_V$ and \eqref{inf-sup-bh}, we   get
\begin{align*}
\sup_{\mathbf{w}_h\in \mathbf{W}_h^0}\frac{\tilde{b}_{h}(\mathbf{w}_h,\theta_h)}{|||\mathbf{w}_h|||_W}
&=\sup_{\mathbf{w}_h\in \mathbf{W}_h^0}
\frac{(\mathbf{w}_{ho},\nabla_{w,k}\theta_h)}{|||\mathbf{w}_h|||_W}\\
&=\sup_{\mathbf{w}_h\in \mathbf{W}_h^0}
\frac{(\mathbf{w}_{ho},\nabla_{w,k}\tilde\theta_h)+(\mathbf{w}_{ho},\nabla_{w,k}\{c_0,c_0\})}{|||\mathbf{w}_h|||_W}\\
&\geq \sup_{\mathbf{w}_h\in \mathbf{V}_h^0}
\frac{(\mathbf{w}_{ho},\nabla_{w,k}\tilde\theta_h)}{|||\mathbf{w}_h|||_V}\\
&\gtrsim |||\tilde\theta_h|||_Q=\left(\|\theta_{ho}-c_0\|_0^2+\sum_{K\in\mathcal{T}_h}h_K^2\|\nabla_{w,k}\theta_h\|_{0,K}^2\right)^{1/2}.
\end{align*}
Taking $\mathbf{w}_h= \{\mathbf{w}_{ho},\mathbf{w}_{hb}\}=\{\nabla_{w,k}\theta_h,\mathbf{0}\}$ in this relation and using  Lemma \ref{lemma2*}, we further obtain
\begin{align*}
\sup_{\mathbf{w}_h\in \mathbf{W}_h^0}\frac{\tilde{b}_{h}(\mathbf{w}_h,\theta_h)}{|||\mathbf{w}_h|||_W}
&\gtrsim 
\frac{(\nabla_{w,k}\theta_h,\nabla_{w,k}\theta_h)}{||\nabla_h\times\nabla_{w,k}\theta_h||_0
+h^{-1/2}||\nabla_{w,k}\theta_h\times\mathbf{n}||_{0,\partial\mathcal{T}_h}}\\
&\gtrsim  
\frac{||\nabla_{w,k}\theta_h||_0^2}{h^{-1}||\nabla_{w,k}\theta_h||_0}\\
&\gtrsim |||\theta_h|||_R,
\end{align*}
i.e. (\ref{inf-sup-tildebh}) holds.
This completes the proof.
\end{proof}

\subsection{Existence and uniqueness results}

For the discretization problems   (\ref{scheme01*}) and (\ref{3c1*}), we readily have   the following equivalence result.

\begin{myLem}\label{lemma15a*}

The  problems (\ref{scheme01*}) and (\ref{3c1*}) are equivalent in the sense that both (I) and
(II) hold:
\begin{itemize}
\item[(I)] If $(\mathbf{u}_h,\mathbf{B}_h,p_h,r_h)
\in\mathbf{V}_h^0\times\mathbf{W}_h^0\times Q_h^0\times R_h^0$ is the solution to the problem (\ref{scheme01*}), then $\mathbf{u}_h$ and $\mathbf{B}_h$  solve the problem (\ref{3c1*});

\item[(II)] If $\mathbf{u}_h \in \mathbf{\bar{V}}_h $ and $\mathbf{B}_h
\in \mathbf{\bar{W}}_h$ solve the problem (\ref{3c1*}), then
$(\mathbf{u}_h,\mathbf{B}_h,p_h,r_h)$ is   a solution to the problem (\ref{scheme01*}), where $ p_h\in Q_h^0 $ and $r_h\in R_h^0$ are determined by
\begin{align}\label{aux-ph}
b_{h}(\mathbf{v}_h,p_h)&=(\mathbf{f},\mathbf{v}_{ho})-a_{h}(\mathbf{u}_h,\mathbf{v}_h)
-c_{h}(\mathbf{u}_h;\mathbf{u}_h,\mathbf{v}_h)-\tilde{c}_{h}(\mathbf{v}_h;\mathbf{B}_h,\mathbf{B}_h), \ \ \ \
\forall\mathbf{v}_h\in\mathbf{V}_h^0,\\
\label{aux-rh}
\tilde{b}_{h}(\mathbf{w}_h,r_h)&=\frac{1}{R_m}(\mathbf{g},\mathbf{w}_{ho})-\tilde{a}_{h}(\mathbf{B}_h,\mathbf{w}_h)
+\tilde{c}_{h}(\mathbf{u}_h;\mathbf{B}_h,\mathbf{w}_h), \ \ \ \ \ \ \ \ \ \ \ \ \ \ \ \ \ \
\forall\mathbf{w}_h\in\mathbf{W}_h^0.
\end{align}
\end{itemize}
\end{myLem}

From Lemma \ref{lemma13*} it is  easy to know that   $M_{h}$   and $\tilde{M}_{h}$ are bounded and depend on  the parameters $N$ and  $R_m$, respectively.


\begin{myLem}\label{results11*}
The problem (\ref{3c1*}) admits at least one solution $(\mathbf{u}_h,\mathbf{B}_h)
\in\mathbf{\bar{V}}_h\times\mathbf{\bar{W}}_h$. In addition, there holds 
\begin{align}\label{3c41*}
|||\mathbf{u}_h|||_V+|||\mathbf{B}_h|||_W\leq
2\zeta\left(H_a\|\mathbf{f}\|_{h}+ \|\mathbf{g}\|_{\tilde{h}}\right),
\end{align}
where
\begin{align}\label{3c41*-cons}
&\zeta:=\max\{H_a,R_m\},\\
&\|\mathbf{f}\|_h:=\sup_{\mathbf{0}\neq\mathbf{v}_h\in \mathbf{\bar{V}}_h}\frac{(\mathbf{f},\mathbf{v}_{ho})}{|||\mathbf{v}_h|||_V},\quad \|\mathbf{g}\|_{\tilde{h}}:=\sup_{\mathbf{0}\neq\mathbf{w}_h\in \mathbf{\bar{W}}_h}\frac{(\mathbf{g},\mathbf{w}_{ho})}{|||\mathbf{w}_h|||_W}
\end{align}
\end{myLem}

\begin{proof}

Taking $\mathbf{v}_h=\mathbf{u}_h$ and $\mathbf{w}_h=\mathbf{B}_h$ in \eqref{3c1*},   by Lemma \ref{lemma13*}  we have
\begin{align*}
\frac{1}{H_a^2}|||\mathbf{u}_h|||_V^2+\frac{1}{R_m^2}|||\mathbf{B}_h|||_W^2=
(\mathbf{f},\mathbf{u}_h)+\frac{1}{R_m}(\mathbf{g},\mathbf{B}_h)
\leq\|\mathbf{f}\|_{h}|||\mathbf{u}_h|||_V+\frac{1}{R_m}\|\mathbf{g}\|_{\tilde{h}}|||\mathbf{B}_h|||_W,
\end{align*}
which yields
%
\begin{align*}
\frac12\min\{\frac{1}{H_a^2},\frac{1}{R_m^2}\}\left(|||\mathbf{u}_h|||_V+|||\mathbf{B}_h|||_W\right)^2
\leq &\min\{\frac{1}{H_a^2},\frac{1}{R_m^2}\}\left(|||\mathbf{u}_h|||_V^2+|||\mathbf{B}_h|||_W^2\right)\\
\leq &\frac{1}{H_a^2}|||\mathbf{u}_h|||_V^2+\frac{1}{R_m^2}|||\mathbf{B}_h|||_W^2\\
\leq &
H_a^2\|\mathbf{f}\|_{h}^2+ \|\mathbf{g}\|_{\tilde{h}}^2\\
\leq & (H_a \|\mathbf{f}\|_{h} + \|\mathbf{g}\|_{\tilde{h}})^2.
\end{align*}
Thus, we get the boundedness result \eqref{3c41*}.

Introduce a mapping $\mathbb{A}:\mathbf{\bar{V}}_h\times\mathbf{\bar{W}}_h\rightarrow
\mathbf{\bar{V}}_h\times\mathbf{\bar{W}}_h$, defined by $\mathbb{A}(\mathbf{u}_h,\mathbf{B}_h)=
(\mathbf{w}_{u},\mathbf{w}_{B})$, where $(\mathbf{w}_{u},\mathbf{w}_{B})\in\mathbf{\bar{V}}_h\times\mathbf{\bar{W}}_h$ is given by
\begin{align}\label{3c51*}
&a_{h}(\mathbf{w}_{u},\mathbf{v}_h)
+\tilde{a}_{h}(\mathbf{w}_{B},\mathbf{w}_h)\nonumber\\
=&(\mathbf{f},\mathbf{v}_{ho})+\frac{1}{R_m}(\mathbf{g},\mathbf{w}_{ho})
-c_{h}(\mathbf{u}_h;\mathbf{u}_h,\mathbf{v}_h)-\tilde{c}_{h}(\mathbf{v}_h;\mathbf{B}_h,\mathbf{B}_h)
+\tilde{c}_{h}(\mathbf{u}_h;\mathbf{B}_h,\mathbf{w}_h),\
\forall(\mathbf{v}_h,\mathbf{w}_h)
\in\mathbf{\bar{V}}_h\times\mathbf{\bar{W}}_h.
\end{align}

Clearly $(\mathbf{u}_h,\mathbf{B}_h)$
is a solution to \eqref{3c1*} if it is a solution to
\begin{align}\label{newsys}
\mathbb{A}(\mathbf{u}_h,\mathbf{B}_h)=
(\mathbf{u}_{h},\mathbf{B}_{h}).
\end{align}
In order to show the system \eqref{newsys} has a solution,
from the Leray-Schauder's principle it suffices to prove the following two assertions:
\begin{itemize}
\item[(i)] $\mathbb{A}$ is a continuous and compact mapping;
\item[(ii)] For any $0\leq \lambda\leq 1$, the set $\mathbf{\bar{V}}_{\lambda,h}\times
\mathbf{\bar{W}}_{\lambda,h}:=\{(\mathbf{v}_h,\mathbf{w}_h)\in \mathbf{\bar{V}}_h\times\mathbf{\bar{W}}_h    : \ (\mathbf{v}_h,\mathbf{w}_h)=\lambda\mathbb{A}(\mathbf{v}_h,\mathbf{w}_h)\}$ is bounded.
\end{itemize}

In fact, let $\mathbf{u}_{1h},\mathbf{u}_{2h}\in\mathbf{\bar{V}}_h$, $\mathbf{B}_{1h},\mathbf{B}_{2h}\in\mathbf{\bar{W}}_h$, and set
$\mathbb{A}(\mathbf{u}_{1h},\mathbf{B}_{1h})=(\mathbf{w}_{1u},\mathbf{w}_{1B})$ and
$\mathbb{A}(\mathbf{u}_{2h},\mathbf{B}_{2h})=(\mathbf{w}_{2u},\mathbf{w}_{2B})$, then we have
\begin{align}
&a_{h}(\mathbf{w}_{1u},\mathbf{v}_h)
+\tilde{a}_{h}(\mathbf{w}_{1B},\mathbf{w}_h)
+c_{h}(\mathbf{u}_{1h};\mathbf{u}_{1h},\mathbf{v}_h)
+\tilde{c}_{h}(\mathbf{v}_{1h};\mathbf{B}_{1h},\mathbf{B}_{1h})-
\tilde{c}_{h}(\mathbf{u}_{1h};\mathbf{B}_{1h},\mathbf{w}_h)\nonumber\\
=&(\mathbf{f},\mathbf{v}_{ho})+\frac{1}{R_m}(\mathbf{g},\mathbf{w}_{ho}), \label{3c61*}\\
\label{3c71*}
&a_{h}(\mathbf{w}_{2u},\mathbf{v}_h)
+\tilde{a}_{h}(\mathbf{w}_{2B},\mathbf{w}_h)v
+c_{h}(\mathbf{u}_{2h};\mathbf{u}_{2h},\mathbf{v}_h)
+\tilde{c}_{h}(\mathbf{v}_{2h};\mathbf{B}_{2h},\mathbf{B}_{2h})-
\tilde{c}_{h}(\mathbf{u}_{2h};\mathbf{B}_{2h},\mathbf{w}_h)\nonumber\\
=&(\mathbf{f},\mathbf{v}_{ho})+\frac{1}{R_m}(\mathbf{g},\mathbf{w}_{ho}),
\end{align}
for all $(\mathbf{v}_h,\mathbf{w}_h)
\in\mathbf{\bar{V}}_h\times\mathbf{\bar{W}}_h.$
Subtracting (\ref{3c71*}) from (\ref{3c61*}), and taking $\mathbf{v}_h=\mathbf{w}_{1u}-\mathbf{w}_{2u}$,
$\mathbf{w}_h=\mathbf{w}_{1B}-\mathbf{w}_{2B}$, we obtain
\begin{align*}
&a_{h}(\mathbf{w}_{1u}-\mathbf{w}_{2u},\mathbf{w}_{1u}-\mathbf{w}_{2u})+
\tilde{a}_{h}(\mathbf{w}_{1B}-\mathbf{w}_{2B},\mathbf{w}_{1B}-\mathbf{w}_{2B})\\
=&-c_{h}(\mathbf{u}_{1h};\mathbf{u}_{1h}-\mathbf{u}_{2h},\mathbf{w}_{1u}-\mathbf{w}_{2u})
-c_{h}(\mathbf{u}_{1h}-\mathbf{u}_{2h};\mathbf{u}_{2h},\mathbf{w}_{1u}-\mathbf{w}_{2u})\\
&-\tilde{c}_{h}(\mathbf{w}_{1u}-\mathbf{w}_{2u};\mathbf{B}_{1h}-\mathbf{B}_{2h};\mathbf{B}_{1h})
-\tilde{c}_{h}(\mathbf{w}_{1u}-\mathbf{w}_{2u};\mathbf{B}_{2h},\mathbf{B}_{1h}-\mathbf{B}_{2h})\\
&+\tilde{c}_{h}(\mathbf{u}_{1h};\mathbf{B}_{1h}-\mathbf{B}_{2h},\mathbf{w}_{1B}-\mathbf{w}_{2B})
+\tilde{c}_{h}(\mathbf{u}_{1h}-\mathbf{u}_{2h};\mathbf{B}_{2h},\mathbf{w}_{1B}-\mathbf{w}_{2B}),
\end{align*}
which, together with  Lemma \ref{lemma13*}, leads to 
\begin{align*}
&\ \frac{1}{H_a^2}|||\mathbf{w}_{1u}-\mathbf{w}_{2u}|||_V^2+
\frac{1}{R_m^2}|||\mathbf{w}_{1B}-\mathbf{w}_{2B}|||_W^2\\
\leq&\ M_{h}\left(|||\mathbf{u}_{1h}|||_V+|||\mathbf{u}_{2h}|||_V\right)
|||\mathbf{u}_{1h}-\mathbf{u}_{2h}|||_V|||\mathbf{w}_{1u}-\mathbf{w}_{2u}|||_V\\
&\ +\tilde{M}_{h}\left(|||\mathbf{B}_{1h}|||_W+|||\mathbf{B}_{2h}|||_W\right)
|||\mathbf{B}_{1h}-\mathbf{B}_{2h}|||_W|||\mathbf{w}_{1u}-\mathbf{w}_{2u}|||_V\\
&\ +\tilde{M}_{h}
\left(|||\mathbf{u}_{1h}|||_V|||\mathbf{B}_{1h}-\mathbf{B}_{2h}|||_W
+|||\mathbf{B}_{2h}|||_W|||\mathbf{u}_{1h}-\mathbf{u}_{2h}|||_V\right)|||\mathbf{w}_{1B}-\mathbf{w}_{2B}|||_W.
\end{align*}
This estimate plus \eqref{3c41*}   yields
\begin{align}\label{ineq1}
&\ |||\mathbf{w}_{1u}-\mathbf{w}_{2u}|||_V+
|||\mathbf{w}_{1B}-\mathbf{w}_{2B}|||_W\nonumber\\
\leq&\ 2\zeta \left[H_a M_{h}\left(|||\mathbf{u}_{1h}|||_V+|||\mathbf{u}_{2h}|||_W\right)
|||\mathbf{u}_{1h}-\mathbf{u}_{2h}|||_V
+H_a\tilde{M}_{h}\left(|||\mathbf{B}_{1h}|||_W+|||\mathbf{B}_{2h}|||_W\right)
|||\mathbf{B}_{1h}-\mathbf{B}_{2h}|||_W\right.\nonumber\\
&\quad \left. +R_m\tilde{M}_{h}
\left(|||\mathbf{u}_{1h}|||_V|||\mathbf{B}_{1h}-\mathbf{B}_{2h}|||_W
+|||\mathbf{B}_{2h}|||_W|||\mathbf{u}_{1h}-\mathbf{u}_{2h}|||_V\right)\right]\nonumber\\
\leq & \ 4\zeta^2 \left(H_a\|\mathbf{f}\|_{h}+\|\mathbf{g}\|_{\tilde{h}}\right) \left[(2H_a M_{h} +R_m\tilde{M}_{h})
|||\mathbf{u}_{1h}-\mathbf{u}_{2h}|||_V
+(2H_a\tilde{M}_{h}+R_m\tilde{M}_{h})
|||\mathbf{B}_{1h}-\mathbf{B}_{2h}|||_W\right]\nonumber\\
\leq & \ 12 \zeta^3 \max\{M_{h}, \tilde{M}_{h}\}\left(H_a\|\mathbf{f}\|_{h}+\|\mathbf{g}\|_{\tilde{h}}\right) \left( 
|||\mathbf{u}_{1h}-\mathbf{u}_{2h}|||_V
+
|||\mathbf{B}_{1h}-\mathbf{B}_{2h}|||_W\right).
\end{align}
This result implies that $\mathbb{A}$ is equicontinuous and uniformly bounded,
since
$$
\mathbb{A}(\mathbf{u}_{1h},\mathbf{B}_{1h})-
\mathbb{A}(\mathbf{u}_{2h},\mathbf{B}_{2h}) =
(\mathbf{w}_{1u}-\mathbf{w}_{2u},\mathbf{w}_{1B}-\mathbf{w}_{2B}).$$
Thus, $\mathbb{A}$ is compact by the
Arzel\'{a}-Ascoli theorem \cite{B2010}, and (i) holds.

The work left is to show (ii). If $\lambda=0$, then $\mathbf{\bar{V}}_{\lambda,h}\times\mathbf{\bar{W}}_{\lambda,h}=\{(0,0)\}$. For $\lambda\in(0,1]$ and $(\mathbf{v}_u,\mathbf{v}_B)
\in\mathbf{\bar{V}}_{\lambda,h}\times\mathbf{\bar{W}}_{\lambda,h}$, using (\ref{3c51*}) we have
\begin{align*}
\lambda^{-1}(a_{h}(\mathbf{v}_{u},\mathbf{v}_h)
+\tilde{a}_{h}(\mathbf{v}_{B},\mathbf{w}_h))
+c_{h}(\mathbf{v}_{u};\mathbf{v}_{u},\mathbf{v}_h)
+\tilde{c}_{h}(\mathbf{v}_h;\mathbf{v}_{B},\mathbf{v}_{B})-
\tilde{c}_{h}(\mathbf{v}_{u};\mathbf{v}_{B};\mathbf{w}_h)
=(\mathbf{f},\mathbf{v}_h)+\frac{1}{R_m}(\mathbf{g},\mathbf{w}_h).
\end{align*}
Similar to \eqref{3c41*}, there holds
 \begin{align*}
|||\mathbf{v}_{u}|||_V+|||\mathbf{v}_{B}|||_W
\leq2\zeta\lambda(H_a\|\mathbf{f}\|_{h}+\|\mathbf{g}\|_{\tilde{h}}).
\end{align*}
As a result,  (ii) holds. This completes the proof.
\end{proof}

Denote
$$\delta:=12\zeta^3
\max\{M_{h},\tilde{M}_{h}\}=12\max\{H_a,R_m\}^3
\max\{M_{h},\tilde{M}_{h}\},$$
and we have   the following  uniqueness result.
\begin{myLem}\label{results211*}
Under the smallness condition that
\begin{align}\label{3c81*}
\delta\left(H_a\|\mathbf{f}\|_{h}+\|\mathbf{g}\|_{\tilde{h}}\right)<1,
\end{align}
  the problem (\ref{3c1*}) admits a unique solution $(\mathbf{u}_h,\mathbf{B}_h)
\in\mathbf{\bar{V}}_h\times\mathbf{\bar{W}}_h$.
\end{myLem}

\begin{proof}
Let $(\mathbf{u}_{1h},\mathbf{B}_{1h})$,
$(\mathbf{u}_{2h},\mathbf{B}_{2h})\in\mathbf{\bar{V}}_h\times\mathbf{\bar{W}}_h$ be two solutions of problem (\ref{3c1*}). Then it suffices to show $\mathbf{u}_{1h}=\mathbf{u}_{2h}$ and $\mathbf{B}_{1h}=\mathbf{B}_{2h}$. In  fact, from (\ref{3c1*}) it follows that, for any $(\mathbf{v}_{h},\mathbf{w}_{h})\in\mathbf{\bar{V}}_h\times\mathbf{\bar{W}}_h$,
\begin{align*}
a_{h}(\mathbf{u}_{1h},\mathbf{v}_h)
+\tilde{a}_{h}(\mathbf{B}_{1h},\mathbf{w}_h)
+c_{h}(\mathbf{u}_{1h};\mathbf{u}_{1h},\mathbf{v}_h)
+\tilde{c}_{h}(\mathbf{v}_h;\mathbf{B}_{1h},\mathbf{B}_{1h})-
\tilde{c}_{h}(\mathbf{u}_{1h};\mathbf{B}_{1h},\mathbf{w}_h)
=(\mathbf{f},\mathbf{v}_h)+\frac{1}{R_m}(\mathbf{g},\mathbf{w}_h),\\
a_{h}(\mathbf{u}_{2h},\mathbf{v}_h)
+\tilde{a}_{h}(\mathbf{B}_{2h},\mathbf{w}_h)
+c_{h}(\mathbf{u}_{2h};\mathbf{u}_{2h},\mathbf{v}_h)
+\tilde{c}_{h}(\mathbf{v}_h;\mathbf{B}_{2h},\mathbf{B}_{2h})-
\tilde{c}_{h}(\mathbf{u}_{2h};\mathbf{B}_{2h},\mathbf{w}_h)
=(\mathbf{f},\mathbf{v}_h)+\frac{1}{R_m}(\mathbf{g},\mathbf{w}_h).
\end{align*}
Subtracting the above first  equation from the second one   and choosing $\mathbf{v}_h=\mathbf{u}_{1h}-\mathbf{u}_{2h}$,
$\mathbf{w}_h=\mathbf{B}_{1h}-\mathbf{B}_{2h}$, we get
\begin{align*}
&a_{h}(\mathbf{u}_{1h}-\mathbf{u}_{2h},\mathbf{u}_{1h}-\mathbf{u}_{2h})+
\tilde{a}_{h}(\mathbf{B}_{1h}-\mathbf{B}_{2h},\mathbf{B}_{1h}-\mathbf{B}_{2h})\\
=&-c_{h}(\mathbf{u}_{1h};\mathbf{u}_{1h}-\mathbf{u}_{2h},\mathbf{u}_{1h}-\mathbf{u}_{2h})
-c_{h}(\mathbf{u}_{1h}-\mathbf{u}_{2h};\mathbf{u}_{2h},\mathbf{u}_{1h}-\mathbf{u}_{2h})\\
&-\tilde{c}_{h}(\mathbf{u}_{1h}-\mathbf{u}_{2h};\mathbf{B}_{1h}-\mathbf{B}_{2h};\mathbf{B}_{1h})
-\tilde{c}_{h}(\mathbf{u}_{1h}-\mathbf{u}_{2h};\mathbf{B}_{2h},\mathbf{B}_{1h}-\mathbf{B}_{2h})\\
&+\tilde{c}_{h}(\mathbf{u}_{1h};\mathbf{B}_{1h}-\mathbf{B}_{2h},\mathbf{B}_{1h}-\mathbf{B}_{2h})
+\tilde{c}_{h}(\mathbf{u}_{1h}-\mathbf{u}_{2h};\mathbf{B}_{2h},\mathbf{B}_{1h}-\mathbf{B}_{2h}),
\end{align*}
which, together with  Lemma \ref{lemma13*}, leads to
\begin{align*}
&\frac{1}{H_a^2}|||\mathbf{u}_{1h}-\mathbf{u}_{2h}|||_V^2
+\frac{1}{R_m^2}|||\mathbf{B}_{1h}-\mathbf{B}_{2h}|||_W^2\\
\leq&M_{h}|||\mathbf{u}_{1h}-\mathbf{u}_{2h}|||_V^2|||\mathbf{u}_{1h}|||_V
+\tilde{M}_{h}\left(|||\mathbf{B}_{1h}|||_W+|||\mathbf{B}_{2h}|||_W\right)
|||\mathbf{B}_{1h}-\mathbf{B}_{2h}|||_W
|||\mathbf{u}_{1h}-\mathbf{u}_{2h}|||_V\\
&+\tilde{M}_{h}|||\mathbf{B}_{1h}-\mathbf{B}_{2h}|||_W^2
|||\mathbf{u}_{1h}|||_V
+\tilde{M}_{h}|||\mathbf{B}_{2h}|||_W|||\mathbf{B}_{1h}-\mathbf{B}_{2h}|||_W
|||\mathbf{u}_{1h}-\mathbf{u}_{2h}|||_V .
\end{align*}
This estimate plus \eqref{3c41*} yields
 \begin{align*}
&|||\mathbf{u}_{1h}-\mathbf{u}_{2h}|||_V
+|||\mathbf{B}_{1h}-\mathbf{B}_{2h}|||_W\\
\leq&2\zeta\{H_aM_{h}|||\mathbf{u}_{1h}-\mathbf{u}_{2h}|||_V|||\mathbf{u}_{1h}|||_V
+H_a\tilde{M}_{h}(|||\mathbf{B}_{1h}|||_W+|||\mathbf{B}_{2h}|||_W)
|||\mathbf{B}_{1h}-\mathbf{B}_{2h}|||_W\\
&+R_m\tilde{M}_{h}|||\mathbf{B}_{1h}-\mathbf{B}_{2h}|||_W
|||\mathbf{u}_{1h}|||_V
+R_M\tilde{M}_{h}|||\mathbf{B}_{2h}|||_W|||\mathbf{u}_{1h}-\mathbf{u}_{2h}|||_W\}\\
 \leq&4\zeta^2(H_a\|\mathbf{f}\|_{h}+\|\mathbf{g}\|_{\tilde{h}})
\left((H_aM_{h}+R_M\tilde{M}_{h})|||\mathbf{u}_{1h}-\mathbf{u}_{2h}|||_W
+(2H_a\tilde{M}_{h}+R_m\tilde{M}_{h})|||\mathbf{B}_{1h}-\mathbf{B}_{2h}|||_W
\right)\\
\leq& 12\zeta^3
\max\{M_{h}, \tilde{M}_{h}\}\left(H_a\|\mathbf{f}\|_{h}+\|\mathbf{g}\|_{\tilde{h}}\right)
\left(|||\mathbf{u}_{1h}-\mathbf{u}_{2h}|||_V
+|||\mathbf{B}_{1h}-\mathbf{B}_{2h}|||_W\right).
\end{align*}

In view of the assumption (\ref{3c81*}), the above inequality implies
$$\mathbf{u}_{1h}=\mathbf{u}_{2h}, \quad \mathbf{B}_{1h}=\mathbf{B}_{2h}. $$
This completes the proof.
\end{proof}

Finally, we have the following existence and uniqueness results for the WG scheme (\ref{scheme01*}).
\begin{myTheo}\label{main-results11*}
The scheme (\ref{scheme01*}) admits at least one solution $(\mathbf{u}_h,\mathbf{B}_h,p_h,r_h)
\in\mathbf{V}_h^0\times\mathbf{W}_h^0\times Q_h^0\times R_h^0$ and there holds the boundedness result \eqref{3c41*}. In addition,  the scheme admits a unique solution  under the smallness condition
\eqref{3c81*}.
 \end{myTheo}
 \begin{proof}
 The existence and uniqueness of the discrete solutions $\mathbf{u}_h$ and $\mathbf{B}_h$ follow from Lemmas \ref{lemma15a*}, \ref{results11*} and \ref{results211*}, and the existence and uniqueness of the discrete solution, $p_h$,   to  \eqref{aux-ph} and the discrete solution, $r_h$,   to    \eqref{aux-rh}  follow from  the two discrete inf-sup inequalities in Lemma \ref{lemma15*}.
 \end{proof}

\section{Error estimates}

This section is devoted to establish the error estimates of the WG scheme (\ref{scheme01*}). To this end, we assume that  the weak solution, $(\mathbf{u},\mathbf{B},p,r)$, to the problem (\ref{mhd1*})-(\ref{mhd5*}) satisfies the following regularity conditions:
\begin{equation}
\label{regularity}
    \mathbf{u}\in\mathbf{V}\cap[H^{k+1}(\Omega)]^d, \quad  \mathbf{B}\in  \mathbf{W}\cap[H^{k+1}(\Omega)]^d, \quad p \in L_0^2(\Omega)\cap  H^{k}(\Omega) ,\quad r\in H_0^1(\Omega)\cap H^{k}(\Omega).
    \end{equation}
    Here we recall that $k\geq 1$.
We set
\begin{align*}
\Pi_1\mathbf{u}|_K&:=\{\mathbf{P}_k^{\mathcal{RT}}(\mathbf{u}|_K),\mathbf{Q}_k^b(\mathbf{u}|_K)\},\ \ \
\Pi_2\mathbf{B}|_K:=\{\mathbf{P}_k^{\mathcal{RT}}(\mathbf{B}|_K),\mathbf{Q}_k^b(\mathbf{B}|_K)\},\\
\Pi_3p|_K&:=\{Q_{k-1}^o(p|_K),Q_k^b(p|_K)\}, \ \ \
\Pi_4 r|_K:=\{Q_{k-1}^o(r|_K),Q_k^b(r|_K)\}.
\end{align*}
for any $K\in \mathcal{T}_h$. 

\begin{myLem}\label{lemma21*}
For for any $(\mathbf{v}_h,\mathbf{w}_h,q_h,r_h)\in\mathbf{V}_h^0\times\mathbf{W}_h^0\times Q_h^0\times R_h^0$,
%
  there hold 
\begin{align}\label{d1*}
&
a_{h}(\Pi_1\mathbf{u},\mathbf{v}_h)
+\tilde{a}_{h}(\Pi_2\mathbf{B},\mathbf{w}_h)
+b_{h}(\mathbf{v}_h,\Pi_3p)-b_{h}(\Pi_1\mathbf{u},q_h)
+\tilde{b}_{h}(\mathbf{w}_h,\Pi_4r)-\tilde{b}_{h}(\Pi_2\mathbf{B},\theta_h)
\nonumber\\
&\qquad
+c_{h}(\Pi_1\mathbf{u};\Pi_1\mathbf{u},\mathbf{v}_h)
+\tilde{c}_{h}(\mathbf{v}_h;\Pi_2\mathbf{B},\Pi_2\mathbf{B})
-\tilde{c}_{h}(\Pi_1\mathbf{u};\Pi_2\mathbf{B},\mathbf{w}_h)\nonumber\\
=&
(\mathbf{f},\mathbf{v}_{ho})+\frac{1}{R_m}(\mathbf{g},\mathbf{w}_h)
+E_u({\mathbf{u},\mathbf{v}_h})+E_B({\mathbf{B},\mathbf{w}_h})
+E_{\tilde{u}}( \mathbf{u},\mathbf{v}_h)\nonumber\\
&\qquad
+E_{\tilde{B}1}(\mathbf{v}_h;\mathbf{B}_h,\mathbf{w}_h)
+E_{\tilde{B}2}(\mathbf{u}_h;\mathbf{B}_h,\mathbf{w}_h),
\end{align}
where
\begin{align*}
&E_u({\mathbf{u},\mathbf{v}_h}):=
\frac{1}{H_a^2}\langle(\nabla\mathbf{u}-\mathbf{Q}_{k-1}^o\nabla\mathbf{u}
)\cdot\mathbf{n},\mathbf{v}_{hb}-\mathbf{v}_{ho}\rangle_{\partial\mathcal{T}_h}
+\frac{1}{H_a^2}\langle \tau(\mathbf{P}_k^{\mathcal{RT}}\mathbf{u}-\mathbf{Q}_k^b\mathbf{u}),
\mathbf{v}_{ho}-\mathbf{v}_{hb}\rangle_{\partial\mathcal{T}_h},
\\
&E_B({\mathbf{B},\mathbf{w}_h}):
=-\frac{1}{R_m^2}\langle\nabla\times\mathbf{B}-\mathbf{Q}_{k-1}^o(\nabla\times\mathbf{B}),
(\mathbf{w}_{ho}-\mathbf{w}_{hb})\times\mathbf{n}\rangle_{\partial\mathcal{T}_h}\\
&\ \ \ \ \ \ \ \ \ \ \ \ \ \ \
+\frac{1}{R_m^2}\langle \tau(\mathbf{P}_k^{\mathcal{RT}}\mathbf{B}-\mathbf{Q}_k^b\mathbf{B})\times\mathbf{n},
(\mathbf{w}_{ho}-\mathbf{w}_{hb})\times\mathbf{n}\rangle_{\partial\mathcal{T}_h},
\\
&E_{\tilde{u}}({ \mathbf{u},\mathbf{v}_h}):
=\frac{1}{2N}(\mathbf{u}\otimes\mathbf{u}-
\mathbf{P}_k^{\mathcal{RT}}\mathbf{u}
\otimes\mathbf{P}_k^{\mathcal{RT}}\mathbf{u},\nabla_h\mathbf{v}_{ho})
-\frac{1}{2N}\langle(\mathbf{u}\otimes\mathbf{u}-
\mathbf{Q}_k^b\mathbf{u}\otimes\mathbf{Q}_k^b\mathbf{u}) \mathbf{n},
\mathbf{v}_{ho}\rangle_{\partial\mathcal{T}_h}\\
&\qquad \qquad\qquad
-\frac{1}{2N}(\mathbf{u}\cdot\nabla\mathbf{u}-
\mathbf{P}_k^{\mathcal{RT}}\mathbf{u}\cdot
\nabla_h\mathbf{P}_k^{\mathcal{RT}}\mathbf{u},\mathbf{v}_{ho})
-\frac{1}{2N}\langle\mathbf{v}_{hb}\otimes\mathbf{Q}_k^b\mathbf{u} \mathbf{n},
\mathbf{P}_k^{\mathcal{RT}}\mathbf{u}\rangle_{\partial\mathcal{T}_h},\\
&E_{\tilde{B}1}(\mathbf{B},\mathbf{v}_h)
:=-\frac{1}{R_m}(\nabla_h\times(\mathbf{B}-\mathbf{P}_k^{\mathcal{RT}}\mathbf{B}),\mathbf{v}_{ho}\times\mathbf{B})
+\frac{1}{R_m}(\nabla_h\times\mathbf{P}_k^{\mathcal{RT}}\mathbf{B},\mathbf{v}_{ho}\times(\mathbf{P}_k^{\mathcal{RT}}\mathbf{B}-\mathbf{B}))\\
&\qquad \qquad \qquad
-\frac{1}{R_m}\langle(\mathbf{P}_k^{\mathcal{RT}}\mathbf{B}-\mathbf{Q}_k^b\mathbf{B})\times\mathbf{n},
\mathbf{v}_{ho}\times\mathbf{P}_k^{\mathcal{RT}}\mathbf{B}
\rangle_{\partial\mathcal{T}_h},\\
&E_{\tilde{B}2}(\mathbf{u};\mathbf{B},\mathbf{w}_h)
:=-\frac{1}{R_m}(\nabla_h\times\mathbf{w}_{ho},
(\mathbf{u}\times\mathbf{B}-
\mathbf{P}_k^{\mathcal{RT}}\mathbf{u}\times\mathbf{P}_k^{\mathcal{RT}}\mathbf{B}))
-\frac{1}{R_m}\langle\mathbf{w}_{ho}\times\mathbf{n},\mathbf{u}\times\mathbf{B}
\rangle_{\partial\mathcal{T}_h}\\
& \qquad\qquad\qquad \qquad
-\frac{1}{R_m}\langle(\mathbf{w}_{ho}-\mathbf{w}_{hb})\times\mathbf{n},
\mathbf{P}_k^{\mathcal{RT}}\mathbf{u}\times\mathbf{P}_k^{\mathcal{RT}}\mathbf{B}
\rangle_{\partial\mathcal{T}_h}.
\end{align*}
In addition, we have
\begin{align} \label{d2*}
\mathbf{P}_k^{\mathcal{RT}}\mathbf{u}|_K\in[\mathcal{P}_k(K)]^d \ \text{ and } \
\mathbf{P}_k^{\mathcal{RT}}\mathbf{B}|_K\in[\mathcal{P}_k(K)]^d, \ \ \forall K\in\mathcal{T}_h.
\end{align}
\end{myLem}

\begin{proof}
We first show (\ref{d2*}). For any $K\in\mathcal{T}_h,$ using Lemma \ref{lemma9*} we get
\begin{align*}
(\nabla\cdot\mathbf{P}_k^{\mathcal{RT}}\mathbf{u},\phi_h)_K=(\nabla\cdot\mathbf{u},\phi_h)_K=0, \quad \forall  \phi_h\in \mathcal{P}_k(K),\\
(\nabla\cdot\mathbf{P}_k^{\mathcal{RT}}\mathbf{B},\theta_h)_K=(\nabla\cdot\mathbf{B},\theta_h)_K=0, \quad \forall  \theta_h\in \mathcal{P}_k(K),
\end{align*}
which give
\begin{align}\label{divP_k}
\nabla\cdot\mathbf{P}_k^{\mathcal{RT}}\mathbf{u}=0, \quad
\nabla\cdot\mathbf{P}_k^{\mathcal{RT}}\mathbf{B}=0.
\end{align}
Then  the result (\ref{d2*}) follows from Lemma \ref{lemma8*}.

From the definitions of the bilinear forms $a_{h}(\cdot,\cdot)$ and the weak gradient, the second commutativity property in Lemma \ref{lemma11*}, the properties of the projection $\mathbf{Q}_m^o$ ($m=k,k-1$), the Green's formula, the relation $\langle\nabla\mathbf{u}\ \mathbf{n},\mathbf{v}_{hb}\rangle_{\partial\mathcal{T}_h}=0$ and the definition of $E_u({\mathbf{u},\mathbf{v}_h})$, we immediately get,  for any $\mathbf{v}_h \in\mathbf{V}_h^0$,
\begin{align*}
&a_{h}(\Pi_1\mathbf{u},\mathbf{v}_h)
=\frac{1}{H_a^2}(\nabla_{w,k-1}\Pi_1\mathbf{u},\nabla_{w,k-1}\mathbf{v}_h)
+\frac{1}{H_a^2}\langle \tau(\mathbf{P}_k^{\mathcal{\mathcal{RT}}}\mathbf{u}-\mathbf{Q}_k^b\mathbf{u}),
\mathbf{v}_{ho}-\mathbf{v}_{hb}\rangle_{\partial\mathcal{T}_h}
\nonumber\\
 =& \frac{1}{H_a^2}(\mathbf{Q}_{k-1}^o\nabla\mathbf{u},\nabla_{w,k-1}\mathbf{v}_h)+\frac{1}{H_a^2}\langle \tau(\mathbf{P}_k^{\mathcal{\mathcal{RT}}}\mathbf{u}-\mathbf{Q}_k^b\mathbf{u}),
\mathbf{v}_{ho}-\mathbf{v}_{hb}\rangle_{\partial\mathcal{T}_h}\nonumber\\
=&-\frac{1}{H_a^2}(\nabla_h\cdot\mathbf{Q}_{k-1}^o\nabla\mathbf{u},\mathbf{v}_{ho})
+\frac{1}{H_a^2}\langle\mathbf{Q}_{k-1}^o\nabla\mathbf{u}
\ \mathbf{n},\mathbf{v}_{hb}\rangle_{\partial\mathcal{T}_h}+\frac{1}{H_a^2}\langle \tau(\mathbf{P}_k^{\mathcal{\mathcal{RT}}}\mathbf{u}-\mathbf{Q}_k^b\mathbf{u}),
\mathbf{v}_{ho}-\mathbf{v}_{hb}\rangle_{\partial\mathcal{T}_h}
\nonumber\\
=&\frac{1}{H_a^2}(\mathbf{Q}_{k-1}^o\nabla\mathbf{u},\nabla_h\mathbf{v}_{ho})
+\frac{1}{H_a^2}\langle\mathbf{Q}_{k-1}^o\nabla\mathbf{u}
\ \mathbf{n},\mathbf{v}_{hb}-\mathbf{v}_{ho}
\rangle_{\partial\mathcal{T}_h}+\frac{1}{H_a^2}\langle \tau(\mathbf{P}_k^{\mathcal{\mathcal{RT}}}\mathbf{u}-\mathbf{Q}_k^b\mathbf{u}),
\mathbf{v}_{ho}-\mathbf{v}_{hb}\rangle_{\partial\mathcal{T}_h}\nonumber\\
=&-\frac{1}{H_a^2}(\Delta\mathbf{u},\mathbf{v}_{ho})
+\frac{1}{H_a^2}\langle(\nabla\mathbf{u}-\mathbf{Q}_{k-1}^o\nabla\mathbf{u})\mathbf{n},\mathbf{v}_{hb}
-\mathbf{v}_{ho}\rangle_{\partial\mathcal{T}_h}+\frac{1}{H_a^2}\langle \tau(\mathbf{P}_k^{\mathcal{\mathcal{RT}}}\mathbf{u}-\mathbf{Q}_k^b\mathbf{u}),
\mathbf{v}_{ho}-\mathbf{v}_{hb}\rangle_{\partial\mathcal{T}_h}\nonumber\\
=&-\frac{1}{H_a^2}(\Delta\mathbf{u},\mathbf{v}_{ho})
+E_u({\mathbf{u},\mathbf{v}_h}).
\end{align*}
Similarly, in light of  the definitions of the bilinear forms $\tilde{a}_{h}(\cdot,\cdot)$ and the weak curl, the third commutativity property in Lemma \ref{lemma11*}, the properties of the projection $\mathbf{Q}_m^o$, the Green's formula, the relation $\langle\nabla\times\mathbf{B},\mathbf{w}_{hb}\times\mathbf{n}\rangle_{\partial\mathcal{T}_h}=0$, and the definition of $E_B({\mathbf{B},\mathbf{w}_h})$, we obtain,  for any $\mathbf{w}_h \in\mathbf{W}_h^0$,
 \begin{align*}
&\tilde{a}_{h}(\Pi_2\mathbf{B},\mathbf{w}_h)
=\frac{1}{R_m^2}(\nabla_{w,k-1}\times\Pi_2\mathbf{B},\nabla_{w,k-1}\times\mathbf{w}_h)
+\frac{1}{R_m^2}\langle \tau(\mathbf{P}_k^{\mathcal{RT}}\mathbf{B}-\mathbf{Q}_k^b\mathbf{B})\times\mathbf{n},
(\mathbf{w}_{ho}-\mathbf{w}_{hb})\times\mathbf{n}\rangle_{\partial\mathcal{T}_h}\\
=&\frac{1}{R_m^2}(\mathbf{Q}_{k-1}^o(\nabla\times\mathbf{B}),\nabla_{w,k-1}\times\mathbf{w}_h)
+\frac{1}{R_m^2}\langle \tau(\mathbf{P}_k^{\mathcal{RT}}\mathbf{B}-\mathbf{Q}_k^b\mathbf{B})\times\mathbf{n},
(\mathbf{w}_{ho}-\mathbf{w}_{hb})\times\mathbf{n}\rangle_{\partial\mathcal{T}_h}\\
=&\frac{1}{R_m^2}(\nabla_h\times(\mathbf{Q}_{k-1}^o(\nabla\times\mathbf{B})),\mathbf{w}_{ho})
+\frac{1}{R_m^2}\langle\mathbf{Q}_{k-1}^o(\nabla\times\mathbf{B}),\mathbf{w}_{hb}\times\mathbf{n}
\rangle_{\partial\mathcal{T}_h}\\
&+\frac{1}{R_m^2}\langle \tau(\mathbf{P}_k^{\mathcal{RT}}\mathbf{B}-\mathbf{Q}_k^b\mathbf{B})\times\mathbf{n},
(\mathbf{w}_{ho}-\mathbf{w}_{hb})\times\mathbf{n}\rangle_{\partial\mathcal{T}_h}\\
=&\frac{1}{R_m^2}(\mathbf{Q}_{k-1}^o(\nabla\times\mathbf{B}),\nabla_h\times\mathbf{w}_{ho})
-\frac{1}{R_m^2}\langle\mathbf{Q}_{k-1}^o(\nabla\times\mathbf{B}),(\mathbf{w}_{ho}-\mathbf{w}_{hb})\times\mathbf{n}
\rangle_{\partial\mathcal{T}_h}\\
&+\frac{1}{R_m^2}\langle \tau(\mathbf{P}_k^{\mathcal{RT}}\mathbf{B}-\mathbf{Q}_k^b\mathbf{B})\times\mathbf{n},
(\mathbf{w}_{ho}-\mathbf{w}_{hb})\times\mathbf{n}\rangle_{\partial\mathcal{T}_h}\\
=&\frac{1}{R_m^2}(\nabla\times\nabla\times\mathbf{B},\mathbf{w}_{ho})
-\frac{1}{R_m^2}\langle\nabla\times\mathbf{B}-\mathbf{Q}_{k-1}^o(\nabla\times\mathbf{B}),
(\mathbf{w}_{ho}-\mathbf{w}_{hb})\times\mathbf{n}\rangle_{\partial\mathcal{T}_h}\\
&
+\frac{1}{R_m^2}\langle \tau(\mathbf{P}_k^{\mathcal{RT}}\mathbf{B}-\mathbf{Q}_k^b\mathbf{B})\times\mathbf{n},
(\mathbf{w}_{ho}-\mathbf{w}_{hb})\times\mathbf{n}\rangle_{\partial\mathcal{T}_h}\\
=&\frac{1}{R_m^2}(\nabla\times\nabla\times\mathbf{B},\mathbf{w}_{ho})
+E_B({\mathbf{B},\mathbf{w}_h}).
\end{align*}

In view of  the definitions of $b_h(\cdot,\cdot)$ and the weak gradient, the first commutativity property in Lemma \ref{lemma11*}, the projection property, and the relations (\ref{lemma91*:sub1}), (\ref{divP_k})  and  $\langle\mathbf{u}\cdot\mathbf{n},q_{hb}\rangle_{\partial\mathcal{T}_h}=0$, we get
\begin{align*}
b_{h}(\mathbf{v}_h,\Pi_3p)-b_{h}(\Pi_1\mathbf{u},q_h)
=&(\nabla_{w,k}\{Q_{k-1}^op,Q_k^bp\},\mathbf{v}_{ho}) 
-(\nabla_{w,k}q_h,\mathbf{P}_k^{\mathcal{RT}}\mathbf{u})\\
=&(Q_k^o\nabla p,\mathbf{v}_{ho})+(\nabla\cdot \mathbf{P}_k^{\mathcal{RT}}\mathbf{u},q_{ho})-
\langle \mathbf{P}_k^{\mathcal{RT}}\mathbf{u}\cdot\mathbf{n},q_{hb}\rangle_{\partial\mathcal{T}_h}\\
=&(\nabla p,\mathbf{v}_{ho})-\langle\mathbf{u}\cdot\mathbf{n},q_{hb}\rangle_{\partial\mathcal{T}_h}\\
=&(\nabla p,\mathbf{v}_{ho}), \quad \forall \mathbf{v}_h \in\mathbf{V}_h^0.
\end{align*}
Similarly, we have
\begin{align*}
\tilde{b}_{h}(\mathbf{w}_h,\Pi_4r)-\tilde{b}_{h}(\Pi_2\mathbf{B},\theta_h)
=&\frac{1}{R_m}(\nabla_{w,k}\{Q_{k-1}^or,Q_k^br\},\mathbf{w}_{ho})
-\frac{1}{R_m}(\nabla_{w,k}\theta_h,\mathbf{P}_k^{\mathcal{RT}}\mathbf{B})
\nonumber\\
=&\frac{1}{R_m}(Q_k^o\nabla r,\mathbf{w}_{ho})+\frac{1}{R_m}(\nabla\cdot \mathbf{P}_k^{\mathcal{RT}}\mathbf{B},\theta_{ho})-
\frac{1}{R_m}\langle \mathbf{P}_k^{\mathcal{RT}}\mathbf{B}\cdot\mathbf{n},\theta_{hb}\rangle_{\partial\mathcal{T}_h}\\
=&\frac{1}{R_m}(\nabla r,\mathbf{w}_{ho})-\frac{1}{R_m}\langle\mathbf{B}\cdot\mathbf{n},\theta_{hb}\rangle_{\partial\mathcal{T}_h}\\
=&\frac{1}{R_m}(\nabla r,\mathbf{w}_{ho}), \quad \forall \mathbf{w}_h \in\mathbf{W}_h^0.
\end{align*}

By the Green's formula and the definitions of $c_{h}(\cdot;\cdot,\cdot)$, the weak divergence and $E_{\tilde{u}}
(\cdot,\cdot)$  we   get
\begin{align*}
c_{h}(\Pi_1\mathbf{u};\Pi_1\mathbf{u},\mathbf{v}_h)
=&\frac{1}{2N}(\nabla_{w,k}\cdot
\{\mathbf{P}_k^{\mathcal{RT}}\mathbf{u}\otimes\mathbf{P}_k^{\mathcal{RT}}\mathbf{u},
\mathbf{Q}_k^{b}\mathbf{u}\otimes\mathbf{Q}_k^{b}\mathbf{u}\},\mathbf{v}_{ho})\\
&-\frac{1}{2N}(\nabla_{w,k}\cdot\{\mathbf{v}_{ho}\otimes\mathbf{P}_k^{\mathcal{RT}}\Phi,
\mathbf{v}_{hb}\otimes\mathbf{Q}_k^{b} \mathbf{u}\}
,\mathbf{P}_k^{\mathcal{RT}}\mathbf{u})\\
=&\frac{1}{2N}(\nabla\cdot(\mathbf{u}\otimes\mathbf{u}),\mathbf{v}_{ho})+
\frac{1}{2N}(\mathbf{u}\otimes\mathbf{u}-\mathbf{P}_k^{\mathcal{RT}}\mathbf{u}\otimes\mathbf{P}_k^{\mathcal{RT}}\mathbf{u},
\nabla_h\mathbf{v}_{ho})\nonumber\\
& -\frac{1}{2N}\langle(\mathbf{u}\otimes\mathbf{u}-\mathbf{Q}_k^{b}\mathbf{u}\otimes\mathbf{Q}_k^{b}\mathbf{u}) \mathbf{n}
,\mathbf{v}_{ho}\rangle_{\partial_{\mathcal{T}_h}}\\
&+
\frac{1}{2N}(\nabla\cdot(\mathbf{u}\otimes \mathbf{u}),\mathbf{v}_{ho})
+\frac{1}{2N}( \mathbf{u}\cdot\nabla\mathbf{u}-\mathbf{P}_k^{\mathcal{RT}} \mathbf{u}\cdot\nabla_h\mathbf{P}_k^{\mathcal{RT}}\mathbf{u},
\mathbf{v}_{ho})\nonumber\\
&
+\frac{1}{2N}\langle(\mathbf{v}_{hb}\otimes\mathbf{Q}_k^{b} \mathbf{u} \mathbf{n}
,\mathbf{P}_k^{\mathcal{RT}}\mathbf{u}\rangle_{\partial_{\mathcal{T}_h}}\\
 =&\frac{1}{N}(\nabla\cdot(\mathbf{u}\otimes\mathbf{u}),\mathbf{v}_{ho})+E_{\tilde{u}}
(\mathbf{u},\mathbf{v}_h).
\end{align*}
Similarly, we can obtain
\begin{align*}
&\tilde{c}_{h}(\mathbf{v}_h;\Pi_2\mathbf{B},\Pi_2\mathbf{B})=-\frac{1}{R_m}(\nabla\times\mathbf{B}\times\mathbf{B},\mathbf{v}_{ho})
+E_{\tilde{B}1}(\mathbf{B}_h,\mathbf{v}_h),\nonumber\\
&-\tilde{c}_{h}(\Pi_1\mathbf{u};\Pi_2\mathbf{B},\mathbf{w}_h)=-\frac{1}{R_m}(\nabla\times(\mathbf{u}\times\mathbf{B}),\mathbf{w}_{ho})
-E_{\tilde{B}2}(\mathbf{u}_h;\mathbf{B}_h,\mathbf{w}_h).
\end{align*}

Combining the above  relations   and \eqref{mhd1*}, we finally arrive at the desired conclusion (\ref{d1*}).
\end{proof}
\begin{myLem}\label{lemma18*}
For any 
$\mathbf{v}_h\in\mathbf{ {V}}_h^0$ and
$\mathbf{w}_h\in\mathbf{ {W}}_h^0$, there  hold
\begin{align}
&|E_u(\mathbf{u},\mathbf{v}_h)|\lesssim
h^k\|\mathbf{u}\|_{k+1}|||\mathbf{v}_h|||_V,\label{lemma18*4}\\
&|E_B(\mathbf{B},\mathbf{w}_h)|\lesssim
h^k\|\mathbf{B}\|_{k+1}|||\mathbf{w}_h|||_W,\label{lemma18*5}\\
&|E_{\tilde{u}}(\mathbf{u},\mathbf{v}_h)|\lesssim 
h^k\|\mathbf{u}\|_2\|\mathbf{u}\|_{k+1}|||\mathbf{v}_h|||_V,
\label{lemma18*1}\\
&|E_{\tilde{B}1}(\mathbf{B},\mathbf{v}_h)|
\lesssim
h^k\|\mathbf{B}\|_2\|\mathbf{B}\|_{k+1}|||\mathbf{v}_h|||_V,
\label{lemma18*2}\\
&|E_{\tilde{B2}}(\mathbf{u};\mathbf{B},\mathbf{w}_h)|
\lesssim 
h^k(\|\mathbf{u}\|_2\|\mathbf{B}\|_{k+1}+
\|\mathbf{B}\|_2\|\mathbf{u}\|_{k+1})|||\mathbf{w}_h|||_W.\label{lemma18*3}
\end{align}
\end{myLem}

\begin{proof} We only show (\ref{lemma18*1}), since the other results  can be derived similarly.

We shall estimate the four terms of $E_{\tilde{u}}(\mathbf{u},\mathbf{v}_h)$ one by one.
Using the Cauchy-Schwarz inequality, the  H\"{o}lder's inequality, the Sobolev embedding theorem,
  and Lemmas \ref{lemma9*} , \ref{lemma17*} and \ref{lemma7*}, we have
\begin{align*}
&\ \ \ \
|(\mathbf{u}\otimes\mathbf{u}-\mathbf{P}_k^{\mathcal{RT}}\mathbf{u}\otimes\mathbf{P}_k^{\mathcal{RT}}\mathbf{u},
\nabla_h\mathbf{v}_{ho})|\\
&\leq|((\mathbf{u}-\mathbf{P}_k^{\mathcal{RT}}\mathbf{u})\otimes\mathbf{u},\nabla_h\mathbf{v}_{ho})|
+|(\mathbf{P}_k^{\mathcal{RT}}\mathbf{u}\otimes(\mathbf{u}-\mathbf{P}_k^{\mathcal{RT}}\mathbf{u}),
\nabla_h\mathbf{v}_{ho})|\\
&\leq |\mathbf{u}|_{0,\infty,\Omega}
\sum_{K\in\mathcal{T}_h}|\mathbf{u}-\mathbf{P}_k^{\mathcal{RT}}\mathbf{u}|_{0,K}
\|\nabla_h\mathbf{v}_{ho}\|_{0,K}
+\sum_{K\in\mathcal{T}_h}|\mathbf{u}-\mathbf{P}_k^{\mathcal{RT}}\mathbf{u}|_{0,3,K}
|\mathbf{P}_k^{\mathcal{RT}}\mathbf{u}|_{0,6,K}
\|\nabla_h\mathbf{v}_{ho}\|_{0, K}\\
&\leq|\mathbf{u}|_{0,\infty,\Omega}\sum_{K\in\mathcal{T}_h}|\mathbf{u}-\mathbf{P}_k^{\mathcal{RT}}\mathbf{u}|_{0,K}
\|\nabla_h\mathbf{v}_{ho}\|_{0,K}
+\sum_{K\in\mathcal{T}_h}|\mathbf{u}-\mathbf{P}_k^{\mathcal{RT}}\mathbf{u}|_{0,3,K}
(|\mathbf{u}-\mathbf{P}_k^{\mathcal{RT}}\mathbf{u}|_{0,6,K}+|\mathbf{u}|_{0,6,K})
\|\nabla_h\mathbf{v}_{ho}\|_{0, K}\\
&\leq |\mathbf{u}|_{0,\infty,\Omega}\sum_{K\in\mathcal{T}_h}|\mathbf{u}-\mathbf{P}_k^{\mathcal{RT}}\mathbf{u}|_{0,K}
\|\nabla_h\mathbf{v}_{ho}\|_{0,K}
+(|\mathbf{u}-\mathbf{P}_k^{\mathcal{RT}}\mathbf{u}|_{0,6,\Omega}+|\mathbf{u}|_{0,6,\Omega})\sum_{K\in\mathcal{T}_h}|\mathbf{u}-\mathbf{P}_k^{\mathcal{RT}}\mathbf{u}|_{0,3,K}
\|\nabla_h\mathbf{v}_{ho}\|_{0, K}\\
&\lesssim h^{k+1}|\mathbf{u}|_{0,\infty,\Omega}
|\mathbf{u}|_{k+1}|||\mathbf{v}_h|||_V
+||\mathbf{u}||_{1}
\sum_{K\in\mathcal{T}_h}|\mathbf{u}-\mathbf{P}_k^{\mathcal{RT}}\mathbf{u}|_{0,3,K}
\|\nabla_h\mathbf{v}_{ho}\|_{0,K}\\
&\lesssim h^{k+1}|\mathbf{u}|_{0,\infty}
|\mathbf{u}|_{k+1}|||\mathbf{v}_h|||_V
+h^{k+1-d/6}||\mathbf{u}||_{1}|\mathbf{u}|_{k+1}
||\nabla_h\mathbf{v}_{ho}||_{0}\\
&\lesssim h^{k}\|\mathbf{u}\|_2\|\mathbf{u}\|_{k+1}|||\mathbf{v}_h|||_V.
\end{align*}

Similarly,
  we can obtain
\begin{align*}
&
|\langle(\mathbf{u}\otimes\mathbf{u}-\mathbf{Q}_k^{b}\mathbf{u}
\otimes\mathbf{Q}_k^{b}\mathbf{u})\cdot\mathbf{n}
,\mathbf{v}_{ho}\rangle_{\partial_{\mathcal{T}_h}}|\\
=&|\langle(\mathbf{u}\otimes\mathbf{u}-\mathbf{Q}_k^{b}\mathbf{u}
\otimes\mathbf{Q}_k^{b}\mathbf{u})\cdot\mathbf{n}
,\mathbf{v}_{ho}-\mathbf{v}_{hb}\rangle_{\partial_{\mathcal{T}_h}}|
\\
\leq&
|\langle(\mathbf{u}-\mathbf{Q}_k^{b}\mathbf{u})
\otimes(\mathbf{u}-\mathbf{Q}_k^{o}\mathbf{u}) \mathbf{n}
,\mathbf{v}_{ho}-\mathbf{v}_{hb}\rangle_{\partial_{\mathcal{T}_h}}|
+|\langle(\mathbf{u}-\mathbf{Q}_k^{b}\mathbf{u})\otimes\mathbf{Q}_k^{o}\mathbf{u} \mathbf{n}
,\mathbf{v}_{ho}-\mathbf{v}_{hb}\rangle_{\partial_{\mathcal{T}_h}}|
\\
&\ \ \ \
+|\langle(\mathbf{Q}_k^{o}\mathbf{u}-\mathbf{Q}_k^{b}\mathbf{u})
\otimes(\mathbf{u}-\mathbf{Q}_k^{b}\mathbf{u}) \mathbf{n}
,\mathbf{v}_{ho}-\mathbf{v}_{hb}\rangle_{\partial_{\mathcal{T}_h}}|
+|\langle\mathbf{Q}_k^{o}\mathbf{u}\otimes(\mathbf{u}-\mathbf{Q}_k^{b}\mathbf{u}) \mathbf{n}
,\mathbf{v}_{ho}-\mathbf{v}_{hb}\rangle_{\partial_{\mathcal{T}_h}}|
\\
\leq &
\sum_{K\in\mathcal{T}_h}\left(|\mathbf{u}-\mathbf{Q}_k^{b}\mathbf{u}|_{0,\partial K}
|\mathbf{u}-\mathbf{Q}_k^{o}\mathbf{u}|_{0,\partial K}
|\mathbf{v}_{ho}-\mathbf{v}_{hb}|_{0,\infty,\partial K}
+|\mathbf{u}-\mathbf{Q}_k^{b}\mathbf{u}|_{0,\partial K}
|\mathbf{Q}_k^{o}\mathbf{u}|_{0,6,\partial K}
|\mathbf{v}_{ho}-\mathbf{v}_{hb}|_{0,3,\partial K}
\right)
\\
&\
+\sum_{K\in\mathcal{T}_h}\left(|\mathbf{Q}_k^{o}\mathbf{u}-\mathbf{Q}_k^{b}\mathbf{u}|_{0,\partial K}
|\mathbf{u}-\mathbf{Q}_k^{b}\mathbf{u}|_{0,\partial K}
|\mathbf{v}_{ho}-\mathbf{v}_{hb}|_{0,\infty,\partial K}
+|\mathbf{u}-\mathbf{Q}_k^{b}\mathbf{u}|_{0,\partial K}
|\mathbf{Q}_k^{o}\mathbf{u}|_{0,6,\partial K}
|\mathbf{v}_{ho}-\mathbf{v}_{hb}|_{0,3,\partial K}
\right)\\
\lesssim & h^{k}\|\mathbf{u}\|_2\|\mathbf{u}\|_{k+1}|||\mathbf{v}_h|||_V,
\\\\
&
|(\mathbf{u}\cdot\nabla\mathbf{u}-\mathbf{P}_k^{\mathcal{RT}}\mathbf{u}\cdot\nabla_h\mathbf{P}_k^{\mathcal{RT}}\mathbf{u},
\mathbf{v}_{ho})|
\\
\leq &
|((\mathbf{u}-\mathbf{P}_k^{\mathcal{RT}}\mathbf{u})\cdot\nabla\mathbf{u},\mathbf{v}_{ho})|
+|(\mathbf{P}_k^{\mathcal{RT}}\mathbf{u}
\cdot(\nabla\mathbf{u}-\nabla_h\mathbf{P}_k^{\mathcal{RT}}\mathbf{u}),\mathbf{v}_{ho})|
\\
\leq&
\sum_{K\in\mathcal{T}_h}|\mathbf{u}-\mathbf{P}_k^{\mathcal{RT}}\mathbf{u}|_{0,3,K}
|\nabla\mathbf{u}|_{0,K}
\|\mathbf{v}_{ho}\|_{0,6,K}
+
\sum_{K\in\mathcal{T}_h}|\nabla\mathbf{u}-\nabla_h\mathbf{P}_k^{\mathcal{RT}}\mathbf{u}|_{0,K}
|\mathbf{P}_k^{\mathcal{RT}}\mathbf{u}|_{0,6,K}
\|\mathbf{v}_{ho}\|_{0,3,K}
\\
\lesssim &h^{k}\|\mathbf{u}\|_2\|\mathbf{u}\|_{k+1}|||\mathbf{v}_h|||_V,
\end{align*}
and
\begin{align*}
&
|\langle\mathbf{v}_{hb}\otimes\mathbf{Q}_k^{b}\mathbf{u}\cdot\mathbf{n}
,\mathbf{P}_k^{\mathcal{RT}}\mathbf{u}\rangle_{\partial_{\mathcal{T}_h}}|
=|\langle\mathbf{v}_{hb}\otimes\mathbf{Q}_k^{b}\mathbf{u}\cdot\mathbf{n}
,\mathbf{P}_k^{\mathcal{RT}}\mathbf{u}-\mathbf{Q}_k^{b}\mathbf{u}\rangle_{\partial_{\mathcal{T}_h}}|\\
\leq &
|\langle(\mathbf{v}_{ho}-\mathbf{v}_{hb})\otimes(\mathbf{Q}_k^{b}\mathbf{u}-
\mathbf{Q}_k^{o}\mathbf{u})\cdot\mathbf{n}
,\mathbf{P}_k^{\mathcal{RT}}\mathbf{u}-\mathbf{Q}_k^{b}\mathbf{u}\rangle_{\partial_{\mathcal{T}_h}}|
+|\langle\mathbf{v}_{ho}\otimes
(\mathbf{Q}_k^{b}\mathbf{u}-
\mathbf{Q}_k^{o}\mathbf{u})\cdot\mathbf{n}
,\mathbf{P}_k^{\mathcal{RT}}\mathbf{u}-\mathbf{Q}_k^{b}\mathbf{u}\rangle_{\partial_{\mathcal{T}_h}}|
\\
&\ \ \ \
+|\langle(\mathbf{v}_{ho}-\mathbf{v}_{hb})\otimes\mathbf{Q}_k^{o}\mathbf{u}\cdot\mathbf{n}
,\mathbf{P}_k^{\mathcal{RT}}\mathbf{u}-\mathbf{Q}_k^{b}\mathbf{u}\rangle_{\partial_{\mathcal{T}_h}}|
+|\langle\mathbf{v}_{ho}\otimes\mathbf{Q}_k^{o}\mathbf{u}\cdot\mathbf{n}
,\mathbf{P}_k^{\mathcal{RT}}\mathbf{u}-\mathbf{Q}_k^{b}\mathbf{u}\rangle_{\partial_{\mathcal{T}_h}}|
\\
\lesssim & h^{k}\|\mathbf{u}\|_2\|\mathbf{u}\|_{k+1}|||\mathbf{v}_h|||_V.
\end{align*}
As a result,  the desired estimate (\ref{lemma18*1}) follows.
\end{proof}

\begin{myTheo}\label{estimates1*}
Let 
$(\mathbf{u}_h,\mathbf{B}_h,p_h,r_h)\in
\mathbf{V}_h^0\times\mathbf{W}_h^0\times Q_h^0\times R_h^0$ be the solution to
the WG scheme (\ref{scheme01*}). 
Under the regularity assumption \eqref{regularity}
and the smallness condition \eqref{3c81*} 
there hold the following estimates:
\begin{eqnarray}
&&|||\Pi_1\mathbf{u}-\mathbf{u}_h|||_V
+|||\Pi_2\mathbf{B}-\mathbf{B}_h|||_W
\lesssim h^kC_1(\mathbf{u},\mathbf{B})  ,\label{est-1}\\
&&|||\Pi_3p-p_h|||_Q+|||\Pi_4r-r_h|||_R\lesssim h^k C_1(\mathbf{u},\mathbf{B})
+h^{2k}C_2(\mathbf{u},\mathbf{B}) ,\label{est-2}
\end{eqnarray}
where $$C_1(\mathbf{u},\mathbf{B}):=\left(\|\mathbf{u}\|_{k+1} +\|\mathbf{B}\|_{k+1}\right)\left(1+||\mathbf{u}||_2+||\mathbf{B}||_2\right), $$
$$C_2(\mathbf{u},\mathbf{B}):=\left(\|\mathbf{u}\|_{k+1} +\|\mathbf{B}\|_{k+1}\right)^2\left(1+||\mathbf{u}||_2+||\mathbf{B}||_2\right)^2 $$

\end{myTheo}

\begin{proof}
From   (\ref{scheme01*}) and Lemma \ref{lemma21*},
we can get the  error equation
\begin{align*}
 &a_{h}(\Pi_1\mathbf{u}-\mathbf{u}_h,\mathbf{v}_h)
+\tilde{a}_{h}(\Pi_2\mathbf{B}-\mathbf{B}_h,\mathbf{w}_h)
+b_{h}(\mathbf{v}_h,\Pi_3p-p_h)\nonumber\\
&\ \  -b_{h}(\Pi_1\mathbf{u}-\mathbf{u}_h,q_h)
+\tilde{b}_{h}(\mathbf{w}_h,\Pi_4r-r_h)-\tilde{b}(\Pi_2\mathbf{B}-\mathbf{B}_h,\theta_h)\nonumber\\
&\ \  +c_{h}(\Pi_1\mathbf{u};\Pi_1\mathbf{u},\mathbf{v}_h)
-c_{h}(\mathbf{u}_h;\mathbf{u}_h,\mathbf{v}_h)
+\tilde{c}_{h}(\mathbf{v}_h;\Pi_2\mathbf{B};\Pi_2\mathbf{B})\nonumber\\
&\ \  -\tilde{c}_{h}(\mathbf{v}_h;\mathbf{B}_h,\mathbf{B}_h)
-\tilde{c}_{h}(\Pi_1\mathbf{u};\Pi_2\mathbf{B},\mathbf{w}_h)
+\tilde{c}_{h}(\mathbf{u}_h;\mathbf{B}_h,\mathbf{w}_h)
\nonumber\\
=&E_u({\mathbf{u},\mathbf{v}_h})+E_B({\mathbf{B},\mathbf{w}_h})
+E_{\tilde{u}}(\mathbf{u},\mathbf{v}_h)
+E_{\tilde{B}1}(\mathbf{B}_h,\mathbf{v}_h)
+E_{\tilde{B}2}(\mathbf{u};\mathbf{B},\mathbf{w}_h),
\end{align*}
for any  $(\mathbf{v}_h,\mathbf{w}_h,q_h,\theta_h)\in\mathbf{V}_h^0\times\mathbf{W}_h^0\times Q_h^0\times R_h^0$.
Taking $(\mathbf{v}_h,\mathbf{w}_h,q_h,\theta_h)
=(\Pi_1\mathbf{u}-\mathbf{u}_h,\Pi_2\mathbf{B}-\mathbf{B}_h,
\Pi_3p-p_h,\Pi_4r-r_h)$ in this relation   and using 
Lemma \ref{lemma13*} we get
\begin{align*}
&
\frac{1}{H_a^2}|||\Pi_1\mathbf{u}-\mathbf{u}_h|||_V^2
+\frac{1}{R_m^2}|||\Pi_2\mathbf{B}-\mathbf{B}_h|||_W^2\nonumber\\
=&a_{h}(\Pi_1\mathbf{u}-\mathbf{u}_h,\Pi_1\mathbf{u}-\mathbf{u}_h)
+\tilde{a}_{h}(\Pi_2\mathbf{B}-\mathbf{B}_h,\Pi_2\mathbf{B}-\mathbf{B}_h)\\
=&E_u(\mathbf{u},\Pi_1\mathbf{u}-\mathbf{u}_h)+E_B({\mathbf{B},\Pi_2\mathbf{B}-\mathbf{B}_h})
+E_{\tilde{u}}(\mathbf{u},\Pi_1\mathbf{u}-\mathbf{u}_h)\nonumber\\
 &\ \
+E_{\tilde{B}1}(\mathbf{B}_h,\Pi_1\mathbf{u}-\mathbf{u}_h)
+E_{\tilde{B}2}(\mathbf{u};\mathbf{B},\Pi_2\mathbf{B}-\mathbf{B}_h)
\nonumber\\
 & \ \
-c_{h}(\Pi_1\mathbf{u};\Pi_1\mathbf{u},\Pi_1\mathbf{u}-\mathbf{u}_h)
+c_{h}(\mathbf{u}_h;\mathbf{u}_h,\Pi_1\mathbf{u}-\mathbf{u}_h)
-\tilde{c}_{h}(\Pi_1\mathbf{u}-\mathbf{u}_h;\Pi_2\mathbf{B},\Pi_2\mathbf{B})
\nonumber\\
& \ \
+\tilde{c}_{h}(\Pi_1\mathbf{u}-\mathbf{u}_h;\mathbf{B}_h,\mathbf{B}_h)
+\tilde{c}_{h}(\Pi_1\mathbf{u};\Pi_2\mathbf{B},\Pi_2\mathbf{B}-\mathbf{B}_h)
-\tilde{c}_{h}(\mathbf{u}_h;\mathbf{B},\Pi_2\mathbf{B}-\mathbf{B}_h)\nonumber\\
=&E_u(\mathbf{u},\Pi_1\mathbf{u}-\mathbf{u}_h)+E_B({\mathbf{B},\Pi_2\mathbf{B}-\mathbf{B}_h})
+E_{\tilde{u}}(\mathbf{u},\Pi_1\mathbf{u}-\mathbf{u}_h)\nonumber\\
 &\ \  +E_{\tilde{B}1}(\mathbf{B}_h,\Pi_1\mathbf{u}-\mathbf{u}_h)
+E_{\tilde{B}2}(\mathbf{u};\mathbf{B},\Pi_2\mathbf{B}-\mathbf{B}_h)
\nonumber\\
 & \ \
-c_{h}(\Pi_1\mathbf{u}-\mathbf{u}_h;\mathbf{u}_h,\Pi_1\mathbf{u}-\mathbf{u}_h)
\nonumber\\
& \ \
+\tilde{c}_{h}(\mathbf{u}_h;\Pi_2\mathbf{B}-\mathbf{B}_h,\Pi_2\mathbf{B}-\mathbf{B}_h)
-\tilde{c}_{h}(\Pi_1\mathbf{u}-\mathbf{u}_h;\Pi_2\mathbf{B}-\mathbf{B}_h,\mathbf{B}_h),
\end{align*}
where in the last '=' we have used the relation $c_{h}(\Pi_1\mathbf{u};\Pi_1\mathbf{u}-\mathbf{u}_h,\Pi_1\mathbf{u}-\mathbf{u}_h)=0$. In view of Lemma \ref{lemma18*} and the definitions of $M_h$ and $\tilde M_h$ in \eqref{Mh} and  \eqref{tilde-Mh}, we further have
\begin{eqnarray*}
&&
\frac{1}{H_a^2}|||\Pi_1\mathbf{u}-\mathbf{u}_h|||_V^2
+\frac{1}{R_m^2}|||\Pi_2\mathbf{B}-\mathbf{B}_h|||_W^2\nonumber\\
&\leq&
C\left(  h^k\|\mathbf{u}\|_{k+1}|||\Pi_1\mathbf{u}-\mathbf{u}_h|||_V
+ h^k\|\mathbf{B}\|_{k+1}|||\Pi_2\mathbf{B}-\mathbf{B}_h|||_W
+ h^k\|\mathbf{u}\|_2\|\mathbf{u}\|_{k+1}|||\Pi_1\mathbf{u}-\mathbf{u}_h|||_V\right.\\
&&
\left.\quad + h^k\|\mathbf{B}\|_2\|\mathbf{B}\|_{k+1}|||\Pi_1\mathbf{u}-\mathbf{u}_h|||_V
+ h^k(\|\mathbf{u}\|_2\|\mathbf{B}\|_{k+1}+\|\mathbf{B}\|_2\|\mathbf{u}\|_{k+1})
|||\Pi_2\mathbf{B}-\mathbf{B}_h|||_W\right)\\
&&+M_{h}|||\mathbf{u}_h|||_V|||\Pi_1\mathbf{u}-\mathbf{u}_h|||_V^2
+\tilde{M}_{h}|||\mathbf{u}_h|||_V|||\Pi_2\mathbf{B}-\mathbf{B}_h|||_W^2
+\tilde{M}_{h}|||\mathbf{B}_h|||_W|||\Pi_1\mathbf{u}-\mathbf{u}_h|||_V
|||\Pi_2\mathbf{B}-\mathbf{B}_h|||_W,
\end{eqnarray*}
which, together with  (\ref{3c41*}), yields
\begin{eqnarray*}
&&\ \ \ \
|||\Pi_1\mathbf{u}-\mathbf{u}_h|||_V
+|||\Pi_2\mathbf{B}-\mathbf{B}_h|||_W\nonumber\\
&\leq &
2\zeta C\left( h^k||\mathbf{u}||_{k+1}
+ h^k||\mathbf{B}||_{k+1}
+ h^k\|\mathbf{B}\|_2||\mathbf{B}||_{k+1} + h^k\|\mathbf{u}\|_2||\mathbf{u}||_{k+1}
+h^k\|\mathbf{u}\|_2||\mathbf{B}||_{k+1}+h^k\|\mathbf{B}\|_2||\mathbf{u}||_{k+1}\right)
\\
&&\ \ \ \ +2\zeta \left( H_aM_{h}|||\mathbf{u}_h|||_V|||\Pi_1\mathbf{u}-\mathbf{u}_h|||_V
+R_m\tilde{M}_{h}|||\mathbf{u}_h|||_V|||\Pi_2\mathbf{B}-\mathbf{B}_h|||_W\right.\\
&&\ \ \ \
\quad +\left.\frac{1}{2}R_m\tilde{M}_{h}|||\mathbf{B}_h|||_W|||\Pi_1\mathbf{u}-\mathbf{u}_h|||_V
+\frac{1}{2}R_m\tilde{M}_{h}|||\mathbf{B}_h|||_W|||\Pi_2\mathbf{B}-\mathbf{B}_h|||_W\right)\\
& \leq & 2\zeta C  h^k\left(\|\mathbf{u}\|_{k+1} +\|\mathbf{B}\|_{k+1}\right)\left(1+||\mathbf{u}||_2+||\mathbf{B}||_2\right)\\
&&+4\zeta^2(H_a\|\mathbf{f}\|_{h}+ \|\mathbf{g}\|_{\tilde{h}})
\left(H_aM_{h}|||\Pi_1\mathbf{u}-\mathbf{u}_h|||_V
+R_m\tilde{M}_{h}|||\Pi_2\mathbf{B}-\mathbf{B}_h|||_W\right.\\
&&\ \ \ \
\left. \qquad +\frac{1}{2}R_m\tilde{M}_{h}|||\Pi_1\mathbf{u}-\mathbf{u}_h|||_V
+\frac{1}{2}R_m\tilde{M}_{h}|||\Pi_2\mathbf{B}-\mathbf{B}_h|||_W\right)\\
& \leq &
2\zeta Ch^k\left(\|\mathbf{u}\|_{k+1} +\|\mathbf{B}\|_{k+1}\right)\left(1+||\mathbf{u}||_2+||\mathbf{B}||_2\right)\\
&& \quad +4\zeta^2(H_a\|\mathbf{f}\|_{h}+ \|\mathbf{g}\|_{\tilde{h}})
(H_aM_{h}+\frac{1}{2}R_m\tilde{M}_{h}+R_m\tilde{M}_{h}+\frac{1}{2}R_m\tilde{M}_{h})
(|||\Pi_1\mathbf{u}-\mathbf{u}_h|||_V+
|||\Pi_2\mathbf{B}-\mathbf{B}_h|||_W)\\
& \leq & 2\zeta Ch^k\left(\|\mathbf{u}\|_{k+1} +\|\mathbf{B}\|_{k+1}\right)\left(1+||\mathbf{u}||_2+||\mathbf{B}||_2\right)\\
&&\ \ \ \ +12\zeta^3\max\{M_{h}, \tilde{M}_{h}\}(H_a\|\mathbf{f}\|_{h}+ \|\mathbf{g}\|_{\tilde{h}})
(|||\Pi_1\mathbf{u}-\mathbf{u}_h|||_V+
|||\Pi_2\mathbf{B}-\mathbf{B}_h|||_W),
\end{eqnarray*}
Since the smallness condition   (\ref{3c81*}) implies
$$ 12\zeta^3\max\{M_{h}, \tilde{M}_{h}\}(H_a\|\mathbf{f}\|_{h}+ \|\mathbf{g}\|_{\tilde{h}})=\delta\left(H_a\|\mathbf{f}\|_{h}+\|\mathbf{g}\|_{\tilde{h}}\right)<1,$$
  we immediately obtain the desired estimate \eqref{est-1}.

Next let us estimate the pressure error.
Taking $(\mathbf{w}_h,q_h,r_h)=(0,0,0)$ in  the equation \eqref{d1*},   we have
\begin{align*}
a_{h}(\Pi_1\mathbf{u},\mathbf{v}_h)
+b_{h}(\mathbf{v}_h,\Pi_3p)
+c_{h}(\Pi_1\mathbf{u};\Pi_1\mathbf{u},\mathbf{v}_h)
+\tilde{c}_{h}(\mathbf{v}_h;\Pi_2\mathbf{B},\Pi_2\mathbf{B})
=(\mathbf{f},\mathbf{v}_{ho})
+E_u({\mathbf{u},\mathbf{v}_h})
+E_{\tilde{u}}( \mathbf{u},\mathbf{v}_h) ,
 \end{align*}
which, together with   \eqref{scheme0101*-a},   gives
\begin{align*}
 b_{h}(\mathbf{v}_h,\Pi_3p-p_h) =&E_u({\mathbf{u},\mathbf{v}_h})
+E_{\tilde{u}}(\mathbf{u},\mathbf{v}_h)
-a_{h}(\Pi_1\mathbf{u}-\mathbf{u}_h,\mathbf{v}_h)-c_{h}(\Pi_1\mathbf{u};\Pi_1\mathbf{u},\mathbf{v}_h)
+c_{h}(\mathbf{u}_h;\mathbf{u}_h,\mathbf{v}_h)\\
&\quad
-\tilde{c}_{h}(\mathbf{v}_h;\Pi_2\mathbf{B},\Pi_2\mathbf{B})
+\tilde{c}_{h}(\mathbf{v}_h;\mathbf{B}_h,\mathbf{B}_h)\\
=&E_u({\mathbf{u},\mathbf{v}_h})
+E_{\tilde{u}}(\mathbf{u},\mathbf{v}_h)
-a_{h}(\Pi_1\mathbf{u}-\mathbf{u}_h,\mathbf{v}_h)\\
&\quad -c_{h}(\Pi_1\mathbf{u}-\mathbf{u}_h;\Pi_1\mathbf{u}-\mathbf{u}_h,\mathbf{v}_h)-c_{h}(\mathbf{u}_h;\Pi_1\mathbf{u}-\mathbf{u}_h,\mathbf{v}_h)
-c_{h}(\Pi_1\mathbf{u}-\mathbf{u}_h;\mathbf{u}_h,\mathbf{v}_h)\\
&\quad
-\tilde{c}_{h}(\mathbf{v}_h;\Pi_2\mathbf{B}-\mathbf{B}_h,\Pi_2\mathbf{B}-\mathbf{B}_h)
-\tilde{c}_{h}(\mathbf{v}_h;\mathbf{B}_h,\Pi_2\mathbf{B}-\mathbf{B}_h)
-\tilde{c}_{h}(\mathbf{v}_h;\Pi_2\mathbf{B}-\mathbf{B}_h,\mathbf{B}_h).
\end{align*}
Thus, using the inf-sup condition (\ref{inf-sup-bh}), Lemmas \ref{lemma13*} and  \ref{lemma18*}, and the estimate \eqref{est-1}, we get
\begin{align*}
|||\Pi_3p-p_h|||_Q\lesssim&\sup_{0\neq\mathbf{v}_h\in\mathbf{V}_h^0}
\frac{b_{h}(\mathbf{v}_h,\Pi_3p-p_h)}{|||\mathbf{v}_h|||_V}\\
\lesssim&   h^k \left(\|\mathbf{u}\|_{k+1} +\|\mathbf{B}\|_{k+1}\right)\left(1+||\mathbf{u}||_2+||\mathbf{B}||_2\right)+h^{2k} \left(\|\mathbf{u}\|_{k+1} +\|\mathbf{B}\|_{k+1}\right)^2 \left(1+||\mathbf{u}||_2+||\mathbf{B}||_2\right)^2\\
\lesssim& h^k C_1(\mathbf{u},\mathbf{B})
+h^{2k}C_2(\mathbf{u},\mathbf{B}) .
\end{align*}
Similarly,
by using the inf-sup condition (\ref{inf-sup-tildebh}), Lemmas \ref{lemma13*} and  \ref{lemma18*}, and  \eqref{est-1}, we can obtain
$$ |||\Pi_4r-r_h|||_R\lesssim h^k C_1(\mathbf{u},\mathbf{B})
+h^{2k}C_2(\mathbf{u},\mathbf{B}) .$$
Combining the above two inequalities leads to the desired result \eqref{est-2}.
\end{proof}

In light of  Theorem \ref{estimates1*},  Lemmas \ref{lemma3*}, \ref{lemma1*}, \ref{lemma9*} and \ref{lemma17*}, and  the triangle inequality,   we can finally get the following main error estimates.
\begin{myTheo}\label{estimates2*}
Under the same conditions as in Theorem \ref{estimates1*}, there hold
\begin{align}\label{error-est1}
\|\nabla\mathbf{u}-\nabla_h\mathbf{u}_{ho}\|_0+\|\nabla\mathbf{u}-\nabla_{w,k-1}\mathbf{u}_{h}\|_0&\lesssim h^kC_1(\mathbf{u},\mathbf{B}),\\
\label{error-est2}
 \|\nabla\times\mathbf{B}-\nabla_h\times\mathbf{B}_{ho}\|_0
+\|\nabla\times\mathbf{B}-\nabla_{w,k-1}\times\mathbf{B}_{h}\|_0
&\lesssim h^kC_1(\mathbf{u},\mathbf{B}),\\
 \|p-p_{ho}\|_0
+h||\nabla p-\nabla_{w,k}p||_0&\lesssim h^k C_1(\mathbf{u},\mathbf{B})  +h^k||p||_k
+h^{2k}C_2(\mathbf{u},\mathbf{B}),\\
\|r-r_{ho}-(\bar r -\bar r_{ho})\|_0
+h||\nabla r-\nabla_{w,k}r||_0&\lesssim h^k C_1(\mathbf{u},\mathbf{B})  +h^k||r||_k
+h^{2k}C_2(\mathbf{u},\mathbf{B}),
\end{align}
where $\bar r$ and $\bar r_{ho}$ denote the mean values of $r$ and $r_{ho}$ on $\Omega$, respectively.
\end{myTheo}
\begin{rem}
From the estimates \eqref{error-est1} and \eqref{error-est2} we see that the errors of the velocity  and the magnetic field are  independent of the
pressure and  the magnetic pseudo-pressure. This means that our WG scheme is pressure-robust.
\end{rem}

\section{Local elimination property and iteration scheme}

\subsection{Local elimination}

In this subsection, we shall show that in the WG scheme (\ref{scheme01*}) the  approximations $(\mathbf{u}_{ho},\mathbf{B}_{ho},p_{ho},r_{ho})$ of the velocity, the magnetic field, the
pressure and the magnetic pseudo-pressure
  defined in the interior of elements can be locally eliminated by the using the numerical traces $(\mathbf{u}_{hb},\mathbf{B}_{hb},p_{hb},r_{hb})$ defined  on the boundaries of the elements.
After the local elimination the resulting system  only contains the degrees of freedom of $(\mathbf{u}_{hb},\mathbf{B}_{hb},r_{hb},p_{hb})$ as unknowns.


For any $K\in\mathcal{T}_h$, we take $\mathbf{v}_{ho}|_{\mathcal{T}_h/K}=0$, $\mathbf{v}_{hb}=0$,
$\mathbf{w}_{ho}|_{\mathcal{T}_h/K}=0$, $\mathbf{w}_{hb}=0$,
$q_{ho}|_{\mathcal{T}_h/K}=0$, $q_{hb}=0$,
$\theta_{ho}|_{\mathcal{T}_h/K}=0$, $\theta_{hb}=0$
in  the   scheme (\ref{scheme01*}),  and  obtain the following  local problem:

Find $(\mathbf{u}_{ho},\mathbf{B}_{ho},r_{ho},p_{ho})
\in [\mathcal{P}_k(K)]^d\times[\mathcal{P}_k(K)]^d\times \mathcal{P}_{k-1}(K)\times \mathcal{P}_{k-1}(K)$ such that
\begin{align}\label{local elimination2*}
&a_{h,K}(\mathbf{u}_{ho},\mathbf{v}_{ho})
+\tilde{a}_{h,K}(\mathbf{B}_{ho},\mathbf{w}_{ho})\nonumber\\
&+b_{h,K}(\mathbf{v}_{ho},p_{ho})-b_{h,K}(\mathbf{u}_{ho},q_{ho})
+\tilde{b}_{h,K}(\mathbf{w}_{ho},r_{ho})-\tilde{b}_{h,K}(\mathbf{B}_{ho},\theta_{ho})\nonumber\\
&+c_{h,K}(\mathbf{u}_{ho};\mathbf{u}_{ho},\mathbf{v}_{ho})
+\tilde{c}_{h,K}(\mathbf{v}_{ho};\mathbf{B}_{ho},\mathbf{B}_{ho})
-\tilde{c}_{h,K}(\mathbf{u}_{ho};\mathbf{B}_{ho},\mathbf{w}_{ho})
\nonumber\\
=&\mathbf{F}_{K}(\mathbf{v}_{ho})+\mathbf{G}_{K}(\mathbf{w}_{ho}), \quad \forall (\mathbf{v}_{ho},\mathbf{w}_{ho},\theta_{ho},q_{ho})
\in [\mathcal{P}_k(K)]^d\times[\mathcal{P}_k(K)]^d\times \mathcal{P}_{k-1}(K)\times \mathcal{P}_{k-1}(K),
\end{align}
where
\begin{align*}
&a_{h,K}(\mathbf{u}_{ho},\mathbf{v}_{ho})
:=\frac{1}{H_a^2}\left(\nabla_{w,k-1}\{\mathbf{u}_{ho},\mathbf{0}\},\nabla_{w,k-1}\{\mathbf{v}_{ho},\mathbf{0}\}\right)_K
+s_{h,K}(\mathbf{u}_{ho},\mathbf{v}_{ho}),
\\
&\ \ s_{h,K}(\mathbf{u}_{ho},\mathbf{v}_{ho}):
=\frac{1}{H_a^2}\langle \tau\mathbf{u}_{ho},
\mathbf{v}_{ho}\rangle_{\partial K},
\\
&\tilde{a}_{h,K}(\mathbf{B}_{ho},\mathbf{w}_{ho})
:=\frac{1}{R_m^2}\left(\nabla_{w,k-1}\times\{\mathbf{B}_{ho},\mathbf{0}\},
\nabla_{w,k-1}\times\{\mathbf{w}_{ho},\mathbf{0}\}\right)_K
+\tilde{s}_{h,K}(\mathbf{B}_{ho},\mathbf{w}_{ho}),
\\
&\ \ \tilde{s}_{h,K}(\mathbf{B}_{ho},\mathbf{w}_{ho})
:=\frac{1}{R_m^2}\langle \tau\mathbf{B}_{ho}\times \mathbf{n},
\mathbf{w}_{ho}\times \mathbf{n}\rangle_{\partial K};\\
&b_{h,K}(\mathbf{v}_{ho},q_{ho}):=\left(\nabla_{w,k}\{p_{ho},0\},\mathbf{v}_{ho}\right)_K,\ \
\tilde{b}_{h,K}(\mathbf{w}_{ho},r_{ho}):=\frac{1}{R_m}\left(\nabla_{w,k}\{r_{ho},0\},\mathbf{w}_{ho}\right)_K,
\\&c_{h,K}(\mathbf{u}_{ho};\mathbf{u}_{ho},\mathbf{v}_{ho}):=
 \frac{1}{2N}\left(\nabla_{w,k}
\cdot\{\mathbf{u}_{ho}\otimes\mathbf{u}_{ho},\mathbf{0}\otimes\mathbf{0}\},\mathbf{v}_{ho}\right)_K
-
\frac{1}{2N}\left(\nabla_{w,k}
\cdot\{\mathbf{v}_{ho}\otimes\mathbf{u}_{ho},\mathbf{0}\otimes\mathbf{0}\},\mathbf{u}_{ho}\right)_K,\\
&\tilde{c}_{h,K}(\mathbf{v}_{ho};\mathbf{B}_{ho},\mathbf{B}_{ho}):=
\frac{1}{R_m}\left(\nabla_{w,k}\times\{\mathbf{B}_{ho},\mathbf{0}\},\mathbf{v}_{ho}\times\mathbf{B}_{ho}\right)_K,
\\
&\mathbf{F}_{K}(\mathbf{v}_{ho}):=(\bm{f},\bm{v}_{h0})_K
-\frac{1}{H_a^2}\left(\nabla_{w,k-1}\{\bm{0},\mathbf{u}_{hb}\},\nabla_{w,k-1}\{\mathbf{v}_{ho},\mathbf{0}\}\right)_K
+\frac{1}{H_a^2}\langle \tau\mathbf{u}_{hb},
\mathbf{v}_{ho}\rangle_{\partial K}\\
&\ \ \ \ \ \ \ \ \ \ \ \ \ \ \
-(\nabla_{w,k}\{0,p_{hb}\},\mathbf{v}_{ho})_K
- \frac{1}{2N}\left(\nabla_{w,k}
\cdot\{\mathbf{0}\otimes\mathbf{0},\mathbf{u}_{hb}\otimes\mathbf{u}_{hb}\},\mathbf{v}_{ho}\right)_K,\\
&\mathbf{G}_{K}(\mathbf{w}_{ho}):=\frac{1}{R_m}(\bm{g},\bm{w}_{h0})_K
-\frac{1}{R_m^2}\left(\nabla_{w,k-1}\times\{\mathbf{0},\mathbf{B}_{hb}\},
\nabla_{w,k-1}\times\{\mathbf{w}_{ho},\mathbf{0}\}\right)_K\\
&\ \ \ \ \ \ \ \ \ \ \ \ \ \ \quad
+\frac{1}{R_m^2}\langle \tau\mathbf{B}_{hb}\times \mathbf{n},
\mathbf{w}_{ho}\times \mathbf{n}\rangle_{\partial K}-\frac{1}{R_m}\left(\nabla_{w,k}\{0,r_{hb}\},\mathbf{w}_{ho}\right)_K
.\\
\end{align*}

For any $K\in \mathcal{T}_h$, we define the semi-norms as follows:
\begin{eqnarray*}
&&|||\mathbf{v}_{ho}|||_{D,K}=(\|\nabla_{w,k-1}\{\mathbf{v}_{ho},\mathbf{0}\}\|_{0,K}^2
+\|\tau^{\frac{1}{2}}\mathbf{v}_{ho}\|_{0,\partial K}^2)^{\frac{1}{2}},\\
&&|||\mathbf{w}_{ho}|||_{C,K}:=(\|\nabla_{w,k-1}\times\{\mathbf{w}_{ho},\mathbf{0}\}\|_{0,K}^2
+\|\tau^{\frac{1}{2}}\mathbf{w}_{ho}\times\mathbf{n}\|_{0,\partial K}^2)^{\frac{1}{2}}.
\end{eqnarray*}

By following the same routine as in subsection 4.2, we can derive the existence and uniqueness results for  the local problem (\ref{local elimination2*}).
\begin{myTheo}\label{local elimination3*}
For any $K\in \mathcal{T}_h$  and given numerical traces $\mathbf{u}_{hb}|_{\partial K}, \mathbf{B}_{hb}|_{\partial K}, p_{hb}|_{\partial K}$ and $r_{hb}|_{\partial K}$,
the local problem (\ref{local elimination2*}) admits at least one solution. Moreover,
under the smallness condition
\begin{eqnarray*}
\tilde{\delta}\left(H_a\|\mathbf{F}_K\|_{h}+\|\mathbf{G}_K\|_{\tilde{h}}\right)<1,
\end{eqnarray*}
the problem (\ref{local elimination2*}) admits a unique solution,
where
\begin{align*}
&\tilde{\delta}:= 12\max\{H_a,R_m\}^3
\max\{M_{h,K},\tilde{M}_{h,K}\},\\
&M_{h,K}:=\sup_{\mathbf{0}\neq\Phi_{ho},\mathbf{u}_{ho},\mathbf{v}_{ho}\in \mathbf{\bar{V}}_{h,K}}
\frac{c_{h,K}(\Phi_{ho};\mathbf{u}_{ho},\mathbf{v}_{ho})}
{|||\Phi_{ho}|||_{D,K}|||\mathbf{u}_{ho}|||_{D,K}|||\mathbf{v}_{ho}|||_{D,K}},\\
&\tilde{M}_{h,K}:=\sup_{
\substack{
\mathbf{0}\neq\mathbf{w}_{ho},\mathbf{B}_{ho}\in\mathbf{\bar{W}}_{h,K},\\
\mathbf{0}\neq\mathbf{v}_{ho}\in \mathbf{\bar{V}}_{h,K}}
}
\frac{\tilde{c}_{h,K}(\mathbf{v}_{ho};\mathbf{B}_{ho},\mathbf{w}_{ho})}
{|||\mathbf{w}_{ho}|||_{C,K}|||\mathbf{v}_{ho}|||_{D,K}|||\mathbf{B}_{ho}|||_{C,K}},\\
&\|\mathbf{F}_K\|_h:=\sup_{\mathbf{0}\neq\mathbf{v}_{ho}\in \mathbf{\bar{V}}_{h,K}}\frac{\mathbf{F}_K(\mathbf{v}_{ho})}{|||\mathbf{v}_{ho}|||_{D,K}},\ \ \ \ \
\|\mathbf{G}_K\|_{\tilde{h}}:=\sup_{\mathbf{0}\neq\mathbf{w}_{ho}\in \mathbf{\bar{W}}_{h,K}}\frac{\mathbf{G}_K(\mathbf{w}_{ho})}{|||\mathbf{w}_{ho}|||_{C,K}},\\
&\mathbf{\bar{V}}_{h,K}:=\big\{\mathbf{v}_{ho}\in[\mathcal{P}_k(K)]^d;b_{h,K}(\mathbf{v}_{ho},q_{ho})=0,\forall q_{ho}\in [\mathcal{P}_{k-1}(K)]\big\},\\
&\mathbf{\bar{W}}_{h,K}:=\big\{\mathbf{w}_{ho}\in[\mathcal{P}_k(K)]^d;\tilde{b}_{h,K}(\mathbf{w}_{ho},\theta_{ho})=0,\forall \theta_{ho}\in[\mathcal{P}_{k-1}(K)]\big\},\\
&|||\mathbf{v}_{ho}|||_{D,K}:=(\|\nabla_{w,k-1}\{\mathbf{v}_{ho},\mathbf{0}\}\|_{0,K}^2
+\|\tau^{\frac{1}{2}}\mathbf{v}_{ho}\|_{0,\partial K}^2)^{\frac{1}{2}},\\
 &|||\mathbf{w}_{ho}|||_{C,K}:=(\|\nabla_{w,k-1}\times\{\mathbf{w}_{ho},\mathbf{0}\}\|_{0,K}^2
+\|\tau^{\frac{1}{2}}\mathbf{w}_{ho}\times\mathbf{n}\|_{0,\partial K}^2)^{\frac{1}{2}}.
\end{align*}

\end{myTheo}

\subsection{Oseen's iteration scheme}

The WG scheme (\ref{scheme01*}) is nonlinear, and we shall
  adopt the following Oseen's iterative algorithm for it:
given $\mathbf{u}_h^0$ and $\mathbf{B}_h^0$,  find $(\mathbf{u}_h^n,\mathbf{B}_h^n,p_h^n,r_h^n)$ with $n=1,2,...,$ such that
\begin{subequations}\label{iteration scheme1*}
\begin{eqnarray}
&&a_{h}(\mathbf{u}_h^n,\mathbf{v}_h)
+b_{h}(\mathbf{v}_h,p_h^n)-b_{h}(\mathbf{u}_h^n,q_h)
+c_{h}(\mathbf{u}_h^{n-1};\mathbf{u}_h^n,\mathbf{v}_h)
+\tilde{c}_{h}(\mathbf{v}_h;\mathbf{B}_h^{n-1},\mathbf{B}_h^{n})
=(\mathbf{f},\mathbf{v}_{ho}),
\label{iteration scheme1*:sub1}\\
&&
\tilde{a}_{h}(\mathbf{B}_h^n,\mathbf{w}_h)
 +\tilde{b}_{h}(\mathbf{w}_h,r_h^n)
-\tilde{b}_{h}(\mathbf{B}_h^n,\theta_h)
-\tilde{c}_{h}(\mathbf{u}_h^{n-1};\mathbf{B}_h^{n-1},\mathbf{w}_h)
=\frac{1}{R_m}(\mathbf{g},\mathbf{w}_{ho}),
\label{iteration scheme1*:sub2}
\end{eqnarray}
\end{subequations}
for all $(\mathbf{v}_h,\mathbf{w}_h,q_h,\theta_h)\in\mathbf{V}_h^0\times\mathbf{W}_h^0\times Q_h^0
\times R_h^0$.

\begin{rem} Notice that the  above  Oseen's iterative scheme can be rewritten as follows:
given $\mathbf{u}_h^0$ and $\mathbf{B}_h^0$,  for $n=1,2,...$,

\noindent Step 1:  find  $(\mathbf{B}_h^n, r_h^n)$  such that
\begin{eqnarray*}
 && \tilde{a}_{h}(\mathbf{B}_h^n,\mathbf{w}_h)
 +\tilde{b}_{h}(\mathbf{w}_h,r_h^n)
 +\tilde{b}_{h}(\mathbf{B}_h^n,\theta_h)
\\
&=& \frac{1}{R_m}(\mathbf{g},\mathbf{w}_{ho})+ \tilde{c}_{h}( {  \mathbf{u}_h^{n-1}};\mathbf{B}_h^{n-1},\mathbf{w}_h), \quad \forall  ( \mathbf{w}_h, \theta_h)\in \mathbf{W}_h^0
\times R_h^0;
\end{eqnarray*}
Step 2: find $(\mathbf{u}_h^n, p_h^n )$ such that
\begin{eqnarray*}
&& a_{h}(\mathbf{u}_h^n,\mathbf{v}_h)
+b_{h}(\mathbf{v}_h,p_h^n)+b_{h}(\mathbf{u}_h^n,q_h)
 +c_{h}( {  \mathbf{u}_h^{n-1}};\mathbf{u}_h^n,\mathbf{v}_h)
\\
&=& (\mathbf{f},\mathbf{v}_{ho})- \tilde{c}_{h}(\mathbf{v}_h; {  \mathbf{B}_h^{n-1}},\mathbf{B}_h^{n}), \quad \forall (\mathbf{v}_h, q_h )\in\mathbf{V}_h^0 \times Q_h^0.
\end{eqnarray*}
\end{rem}

We have the following convergence result.
\begin{myTheo}\label{iteration scheme2*}
Let $(\mathbf{u}_h,\mathbf{B}_h,p_h,r_h)\in\mathbf{V}_h^0\times\mathbf{W}_h^0\times Q_h^0
\times R_h^0$ be the solution of the WG scheme (\ref{scheme01*}).   Under the condition (\ref{3c81*})
the Oseen's iteration scheme (\ref{iteration scheme1*}) is convergent in the following sense:
\begin{eqnarray*}
\lim_{n\rightarrow \infty}|||\mathbf{u}_h^n-\mathbf{u}_h|||_V=0,\ \
\lim_{n\rightarrow \infty}|||\mathbf{B}_h^n-\mathbf{B}_h|||_W=0,\ \
\lim_{n\rightarrow \infty}|||p_h^n-p_h|||_Q=0,\ \ \
\lim_{n\rightarrow \infty}|||r_h^n-r_h|||_R=0.
\end{eqnarray*}

\end{myTheo}

\begin{proof}
Denote $e_u^n:=\mathbf{u}_h^n-\mathbf{u}_h,
e_B^n:=\mathbf{B}_h^n-\mathbf{B}_h,
e_p^n:=p_h^n-p_h,
e_r^n:=r_h^n-r_h
$.
Subtracting (\ref{iteration scheme1*}) from (\ref{scheme01*}),
for all $(\mathbf{v}_h,\mathbf{w}_h,q_h,\theta_h)\in\mathbf{V}_h^0\times\mathbf{W}_h^0\times Q_h^0
\times R_h^0$, we have
\begin{align}\label{iteration scheme3*}
&a_{h}(e_u^n,\mathbf{v}_h)+\tilde{a}_{h}(e_B^n,\mathbf{w}_h)
\nonumber\\
&=
-b_{h}(\mathbf{v}_h,e_p^n)+b_{h}(e_u^n,q_h)
+c_{h}(\mathbf{u}_h;\mathbf{u}_h,\mathbf{v}_h)
-c_{h}(\mathbf{u}_h^{n-1};\mathbf{u}_h^n,\mathbf{v}_h)
+\tilde{c}_{h}(\mathbf{v}_h;\mathbf{B}_h,\mathbf{B}_h)
\nonumber\\
&\ \ \ \
-\tilde{c}_{h}(\mathbf{v}_h;\mathbf{B}_h^{n-1},\mathbf{B}_h^{n})\
-\tilde{b}_{h}(\mathbf{w}_h,e_r^n)+\tilde{b}_{h}(e_B^n,\theta_h)
-\tilde{c}_{h}(\mathbf{u}_h;\mathbf{B}_h,\mathbf{w}_h)
+\tilde{c}_{h}(\mathbf{u}_h^{n-1};\mathbf{B}_h^{n-1},\mathbf{w}_h).
\end{align}
Taking $\mathbf{v}_h=e_u^n,
\mathbf{w}_h=e_B^n,
q_h=e_p^n,
\theta_h=e_r^n$ in (\ref{iteration scheme3*}) and using Lemma \ref{lemma13*}, we get
\begin{eqnarray*}\label{iteration scheme4*}
&&\ \ \ \
\frac{1}{H_a^2}|||e_u^n|||_V^2+\frac{1}{R_m^2}|||e_B^n|||_W^2
\nonumber\\
&&=c_{h}(\mathbf{u}_h;\mathbf{u}_h,e_u^n)
-c_{h}(\mathbf{u}_h^{n-1};\mathbf{u}_h^n,e_u^n)
+\tilde{c}_{h}(e_u^n;\mathbf{B}_h,\mathbf{B}_h)
-\tilde{c}_{h}(e_u^n;\mathbf{B}_h^{n-1},\mathbf{B}_h^{n})
-\tilde{c}_{h}(\mathbf{u}_h;\mathbf{B}_h,e_B^n)
+\tilde{c}_{h}(\mathbf{u}_h^{n-1};\mathbf{B}_h^{n-1},e_B^n)
\nonumber\\
&&=-c_{h}(e_u^{n-1};\mathbf{u}_h,e_u^n)
-\tilde{c}_{h}(e_u^n;e_B^{n-1},\mathbf{B}_h)
+\tilde{c}_{h}(\mathbf{u}_h;e_B^{n-1},e_B^{n})
\\
&&
\leq M_{h}|||\mathbf{u}_h|||_V|||e_u^{n-1}|||_V|||e_u^n|||_V
+\tilde{M}_{h}|||e_B^{n-1}|||_W|||e_u^n|||_V|||\mathbf{B}_h|||_W
+\tilde{M}_{h}|||e_B^{n-1}|||_W|||\mathbf{u}_h|||_V|||e_B^{n}|||_W,
\end{eqnarray*}
where in the second '=' we have used the relation $c_{h}(\mathbf{u}_h^{n-1};e_u^n,e_u^n)=0$. Similar to the proof of \eqref{ineq1},
the above estimate plus (\ref{3c41*}) yields
\begin{eqnarray*}
&&
|||e_u^n|||_V+|||e_B^n|||_W\\
&\leq& 2\zeta\left(H_aM_{h}|||\mathbf{u}_h|||_V|||e_u^{n-1}|||_V
+H_a\tilde{M}_{h}|||e_B^{n-1}|||_W|||\mathbf{B}_h|||_W
+R_m\tilde{M}_{h}|||e_B^{n-1}|||_W|||\mathbf{u}_h|||_V\right)\\
&\leq&4\zeta^2(H_aM_{h}+H_a\tilde{M}_{h}+R_m\tilde{M}_{h})
(H_a\|\mathbf{f}\|_{h}+ \|\mathbf{g}\|_{\tilde{h}})(|||e_u^{n-1}|||_V+|||e_B^{n-1}|||_W)\\
&\leq& 12\zeta^3
\max\{M_{h}, \tilde{M}_{h}\}(H_a\|\mathbf{f}\|_{h}+\|\mathbf{g}\|_{\tilde{h}})
(|||e_u^{n-1}|||_V+|||e_B^{n-1}|||_W)\\
&\leq& \delta (H_a\|\mathbf{f}\|_{h}+\|\mathbf{g}\|_{\tilde{h}})
(|||e_u^{n-1}|||_V+|||e_B^{n-1}|||_W)\\
&\leq& \cdots\\
 &\leq& (\delta (H_a\|\mathbf{f}\|_{h}+\|\mathbf{g}\|_{\tilde{h}}))^{n}(|||e_u^0|||_V+|||e_B^0|||_W),
\end{eqnarray*}
which, together with  the smallness condition  (\ref{3c81*}), i.e. $0\leq \delta (H_a\|\mathbf{f}\|_{h}+\|\mathbf{g}\|_{\tilde{h}})<1$, gives
\begin{eqnarray}\label{iteration scheme6*}
\lim_{n\rightarrow \infty}\left(|||\mathbf{u}_h^n-\mathbf{u}_h|||_V+|||\mathbf{B}_h^n-\mathbf{B}_h|||_W\right)
=0.
\end{eqnarray}
From (\ref{iteration scheme3*})  we can get for any $\mathbf{v}_h\in\mathbf{V}_h^0$,
\begin{align*}
b_{h}(\mathbf{v}_h,e_p^n)&=
-a_{h}(e_u^n,\mathbf{v}_h)
+c_{h}(\mathbf{u}_h;\mathbf{u}_h,\mathbf{v}_h)
-c_{h}(\mathbf{u}_h^{n-1};\mathbf{u}_h^n,\mathbf{v}_h)
+\tilde{c}_{h}(\mathbf{v}_h;\mathbf{B}_h,\mathbf{B}_h)
-\tilde{c}_{h}(\mathbf{v}_h;\mathbf{B}_h^{n-1},\mathbf{B}_h^{n})\\
&=
-a_{h}(e_u^n,\mathbf{v}_h)
-c_{h}(e_u^{n-1};\mathbf{u}_h,\mathbf{v}_h)
-c_{h}(e_u^{n-1};e_u^n,\mathbf{v}_h)
-\tilde{c}_{h}(\mathbf{v}_h;e_B^{n-1},\mathbf{B}_h)
-\tilde{c}_{h}(\mathbf{v}_h;e_B^{n-1},e_B^{n}).
\end{align*}
Using the inf-sup condition (\ref{inf-sup-bh}) and Lemma \ref{lemma13*},
we have
\begin{eqnarray*}
&&
|||e_p^n|||_Q\leq\sup_{\mathbf{v}_h\in\mathbf{V}_h^0}\frac{b_{h}(\mathbf{v}_h,e_p^n)}{|||\mathbf{v}_h|||_V}\\
&&\leq \frac1{H_a^2}|||e_u^n|||_V+M_{h}(|||e_u^{n-1}|||_V|||\mathbf{u}_h|||_V
+|||e_u^{n-1}|||_V|||e_u^n|||_V)
+\tilde{M}_{h}(|||e_B^{n-1}|||_W|||\mathbf{B}_h|||_W
+|||e_B^{n-1}|||_W|||e_B^n|||_W),
\end{eqnarray*}
which together with (\ref{iteration scheme6*}), yields
$$\lim_{n\rightarrow \infty}|||p_h^n-p_h|||_Q
=\lim_{n\rightarrow \infty}|||e_p^n|||_Q=0.$$
Similarly, by using the inf-sup condition (\ref{inf-sup-bh}), Lemma \ref{lemma13*} and \eqref{iteration scheme6*}, we can obtain
\begin{eqnarray*}
\lim_{n\rightarrow \infty}|||r_h^n-r_h|||_R=0.
\end{eqnarray*}
This completes the proof.
\end{proof}

\section{Numerical examples}

In this section, we give two 2D numerical examples  to verify the performance
of the WG scheme  (\ref{scheme01*}) for the steady incompressible MHD flow \eqref{mhd1*}.
We apply the  Oseen's iterative scheme with an initial guess
$(\mathbf{u}_{ho}^0,\mathbf{B}_{ho}^0)=(0,0)$ and     the stop criterion
$$\|\mathbf{u}_h^n-\mathbf{u}_h^{n-1}\|_0< 1e-8 $$ in all numerical experiments.

In the examples of the model   \eqref{mhd1*}, we set
$$\Omega=[0,1]^2, H_a= N= R_m=1,$$
and we use regular triangular meshes for the computation (cf. Figure 1).

\begin{figure}[htbp]
\centering
\begin{minipage}[t]{0.7\linewidth}
\centering
\includegraphics[height=3.5cm,width=5cm]{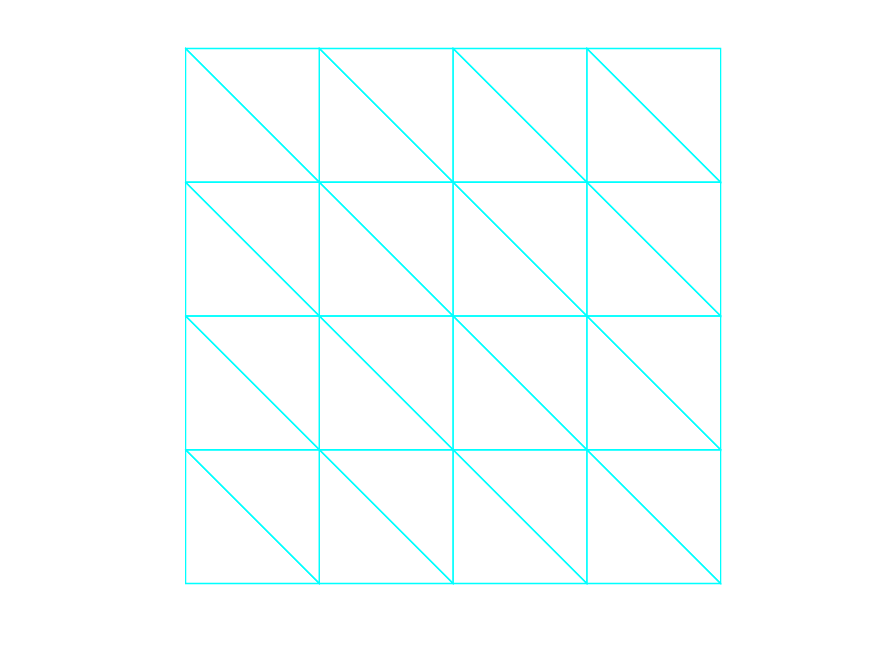}
\caption{The meshes:    $4\times4$ mesh .}
\end{minipage}
\centering
\label{fig1:mesh}
\end{figure}

\begin{exam} \label{exam1}
The exact solution $(\mathbf{u}, \mathbf{B}, p, r)$ is of the form
\begin{eqnarray*}
\left\{\begin {array}{lll}
u_1=-x^2(x-1)^2y(y-1)(2y-1),\ \
u_2=y^2(y-1)^2x(x-1)(2x-1),\\
B_1=-x^2(x-1)^2y(y-1)(2y-1),\ \
B_2=y^2(y-1)^2x(x-1)(2x-1),\\
p=x(x-1)(x-1/2)y(y-1)(y-1/2),\\
r=x(x-1)(x-1/2)y(y-1)(y-1/2).\\
\end{array}\right.
\end{eqnarray*}
We compute  the WG scheme (\ref{scheme01*}) with $k=1,2$. Numerical results are listed in Tables \ref{linetablek=2} and \ref{linetablek=4}.
\end{exam}

\begin{exam} \label{exam2}
The exact solution $(\mathbf{u}, \mathbf{B}, p, r)$ is of the form
\begin{eqnarray*}
\left\{\begin {array}{lll}
u_1=\sin(\pi x)\cos(\pi y),\ \
u_2=-\sin(\pi y)\cos(\pi x),\\
B_1=-x^2(x-1)^2y(y-1)(2y-1),\ \
B_2=y^2(y-1)^2x(x-1)(2x-1),\\
p=x^6-y^6,\ \
r=x(x-1)(x-1/2)y(y-1)(y-1/2).\\
\end{array}\right.
\end{eqnarray*}
We compute  the WG scheme (\ref{scheme01*}) with $k=1,2$. Numerical results are listed in Tables \ref{linetablek=5} and \ref{linetablek=6}.
\end{exam}

Table \ref{linetablek=2} - Table \ref{linetablek=6} show the histories of convergence for the velocity $\mathbf{u}_{ho}$, the magnetic field $\mathbf{B}_{ho}$, the pressure $p_{ho}$, and the magnetic pseudo-pressure $r_{ho}$. Results of $$\text{div} U_h:=\text{Max}_{K\in \mathcal{T}_h}h_K^{-1}||\nabla\cdot\mathbf{u}_{ho}||_{0,K}$$
and $$\text{div} B_h:=\text{Max}_{K\in \mathcal{T}_h}h_K^{-1}||\nabla\cdot\mathbf{B}_{ho}||_{0,K}$$
are also listed  to verify the divergence-free property. From the numerical results of the two examples, we have the following conclusions:
\begin{itemize}
\item The convergence rates of $\|\nabla\mathbf{u}-\nabla_h\mathbf{u}_{ho}\|_0, \quad \|\nabla\mathbf{u}-\nabla_{w,k-1}\mathbf{u}_{h}\|_0, $
$\|\nabla\times\mathbf{B}-\nabla_h\times\mathbf{B}_{ho}\|_0,  \quad \|\nabla\mathbf{B}-\nabla_{w,k-1}\mathbf{B}_{h}\|_0,$
$ \|p-p_{ho}\|_0, \ h\|\nabla p-\nabla_{w,k}p_{h}\|_0$ and  $ h\|\nabla r-\nabla_{w,k}r_{h}\|_0$
  for the  WG scheme with $k=1,2$ are of
   $k^{th}$ orders, which are   consistent with the established theoretical
results in   Theorem \ref{estimates2*}.

\item The convergence rate of $\|r-r_{ho}\|_0$ is also of
$k^{th}$ order.

\item The convergence rates of $\|\mathbf{u}-\mathbf{u}_{ho}\|_0$
and $\|\mathbf{B}-\mathbf{B}_{ho}\|_0$ are of $(k+1)^{th}$ orders.

\item  Based on the facts that
$$\|\nabla_h\cdot\mathbf{u}_{ho} \|_{0,\infty}\lesssim \max\limits_{K\in \mathcal{T}_h} h^{-1}_K\|\nabla\cdot\mathbf{u}_{ho}  \|_{0,K}=\text{div}U_{h},$$
and $$\|\nabla_h\cdot\mathbf{B}_{ho} \|_{0,\infty}\lesssim \max\limits_{K\in \mathcal{T}_h} h^{-1}_K\|\nabla\cdot\mathbf{B}_{ho}  \|_{0,K}=\text{div}B_{h},$$
 we can see that the      the discrete velocity and the discrete magnetic field  are globally divergence-free.
\end{itemize}

\begin{table}[H]
	\normalsize
	\caption{History of convergence results with  $k = 1$ for Example \ref{exam1}}\label{linetablek=2}
	\centering
	\footnotesize
{
		\begin{tabular}{p{1.2cm}<{\centering}|p{1.5cm}<{\centering}|p{1.4cm}<{\centering}|p{1.5cm}<{\centering}
|p{1.4cm}<{\centering}|p{1.5cm}<{\centering}|p{1.4cm}<{\centering}|p{1.5cm}<{\centering}}
			\hline
			\multirow{2}{*}{mesh}&
			\multicolumn{2}{c|}{$\frac{\|\mathbf{u}-\mathbf{u}_{ho}\|_0}{\|\mathbf{u}\|_0}$}
&\multicolumn{2}{c|}{$\frac{\|\nabla\mathbf{u}-\nabla_{w,k-1}\mathbf{u}_{h}\|_0}{\|\nabla\mathbf{u}\|_0}$}
&\multicolumn{2}{c|}{$\frac{\|\nabla\mathbf{u}-\nabla_{h}\mathbf{u}_{ho}\|_0}{\|\nabla\mathbf{u}\|_0}$}
&\multirow{2}{*}{$div Uh$}
\cr\cline{2-7}
			&error&order&error&order&error&order\cr
			\cline{1-8}
$ 2\times 2 $  & 2.3094e+00   &       -  & 9.0044e-01   &       - &  1.2207e+00   &       - &3.4001e-16
\\
\hline
$ 4\times 4 $ &  5.6715e-01  & 1.87 &  5.1165e-01 &  0.77&   6.4277e-01  & 0.92 &3.0878e-16
\\
\hline
$ 8\times 8 $ &  1.5224e-01  & 1.90 &  2.7237e-01  & 0.91&   3.2258e-01  & 0.99 &3.4348e-16
\\
\hline
$ 16\times 16 $ &  3.9918e-02  & 1.93  & 1.3841e-01  & 1.00 &  1.6098e-01 &  1.00 &1.3878e-17
\\
\hline
$ 32\times 32 $ &  1.0236e-02  & 1.96 &  6.9404e-02 &  1.00 &  8.0401e-02 &  1.00 &3.8511e-16
\\
\hline
$ 64\times 64 $ &  2.5908e-03 &  1.98 &  3.4720e-02  & 1.00 &  4.0183e-02 &  1.00 &3.4153e-14
\\
\hline
$ 128\times 128 $  & 6.5160e-04  & 1.99  & 1.7363e-02  & 1.00 &  2.0089e-02 &  1.00 &2.9676e-13
\\
\hline

		\end{tabular}
	}
	
{
		
		\begin{tabular}{p{1.2cm}<{\centering}|p{1.5cm}<{\centering}|p{1.4cm}<{\centering}|p{1.5cm}<{\centering}
|p{1.4cm}<{\centering}|p{1.5cm}<{\centering}|p{1.4cm}<{\centering}|p{1.5cm}<{\centering}}
            \hline
			\multirow{2}{*}{mesh}&
			\multicolumn{2}{c|}{$\frac{\|\mathbf{B}-\mathbf{B}_{ho}\|_0}{\|\mathbf{B}\|_0}$}
&\multicolumn{2}{c|}{$\frac{\|\nabla\times\mathbf{B}-\nabla_{w,k-1}\times\mathbf{B}_{h}\|_0}{\|\nabla\mathbf{B}\|_0}$}
&\multicolumn{2}{c|}{$\frac{\|\nabla\times\mathbf{B}-\nabla_{h}\times\mathbf{B}_{ho}\|_0}{\|\nabla\mathbf{B}\|_0}$}
&\multirow{2}{*}{$div Bh$}
\cr\cline{2-7}
			&error&order&error&order&error&order\cr
			\cline{1-8}
$ 2\times 2 $ &  4.7569e+00  &        -   &9.8373e-01  &        - &  1.5806e+00   &  - &4.8572e-17
\\
\hline
$ 4\times 4 $  & 1.1606e+00 &  2.04 &  5.6441e-01 &  0.82 &  8.2061e-01  & 0.94 &7.9797e-16
\\
\hline
$ 8\times 8 $ &  2.9482e-01  & 1.98  & 2.7675e-01  & 1.03&  3.7929e-01  & 1.11 &4.0939e-16
\\
\hline
$ 16\times 16 $ &  7.4064e-02  & 1.99   &1.3198e-01 &  1.07 &  1.6468e-01 &  1.20 &4.1633e-16
\\
\hline
$ 32\times 32 $  & 1.8525e-02  & 2.00 &  6.4624e-02 &  1.03   &7.5900e-02 &  1.11 &5.6899e-15
\\
\hline
$ 64\times 64 $ &  4.6309e-03 &  2.00  & 3.2115e-02 &  1.01 &  3.6869e-02  & 1.04 &2.0886e-14
\\
\hline
$ 128\times 128 $  & 1.1577e-03 &  2.00 &  1.6032e-02  & 1.00 &  1.8283e-02 &  1.01 &2.9830e-13
\\
\hline
		\end{tabular}
	}

{

		\begin{tabular}{p{1.2cm}<{\centering}|p{1.5cm}<{\centering}|p{0.9cm}<{\centering}|p{1.5cm}<{\centering}
|p{0.9cm}<{\centering}|p{1.5cm}<{\centering}|p{1.0cm}<{\centering}|p{1.5cm}<{\centering}|p{1.0cm}<{\centering}}
     \hline
			\multirow{2}{*}{mesh}&
			\multicolumn{2}{c|}{$\frac{\| p-p_{ho}\|_0}{\| p\|_0}$}
&\multicolumn{2}{c|}{$\frac{h\|\nabla p-\nabla_{w,k}p_{h}\|_0}{\|\nabla r\|_0}$}
&\multicolumn{2}{c|}{$\frac{\|r-r_{ho}\|_0}{\| r\|_0}$}
&\multicolumn{2}{c}{$\frac{h\|\nabla r-\nabla_{w,k}r_{h}\|_0}{\|\nabla r\|_0}$}
\cr\cline{2-9}
			&error&order&error&order&error&order&error&order\cr
			\cline{1-9}
$ 2\times 2 $&  7.0956e+00   & - & 6.2931e-01    &      - &  4.7146e+00  &    - & 7.1008e-01   &       - 
\\
\hline
$ 4\times 4 $ &   4.4806e-01  & 0.73 &2.5901e-01 &  1.20 &  2.8776e+00  & -0.05 & 4.4508e-01 &  0.70 
\\
\hline
$ 8\times 8 $ &   2.3643e-01  & 0.92&  1.2598e-01  & 1.03&   1.4532e-01 &   0.71 &  2.0693e-01 &  1.10 
\\
\hline
$ 16\times 16 $ &  1.1985e-01  & 0.98&6.2587e-02  & 1.00 &  7.2538e-01   & 0.98 & 9.2089e-02 &  1.16 
\\
\hline
$ 32\times 32 $ &  6.0129e-02 &  1.00& 3.1264e-02 &  1.00  &3.6249e-01   &  1.00 & 4.5880e-02 &  1.00 
\\
\hline
$ 64\times 64 $ &  3.0089e-02  & 1.00&  1.5628e-02  & 1.00&  1.8122e-02  &  1.00 &  2.5771e-02 &  0.89
\\
\hline
$ 128\times 128 $ &  1.5048e-02   &1.00 & 7.8132e-03 &  1.00 &  5.4339e-03  &   1.00&  1.5907e-02 &  0.90 
\\
\hline
		\end{tabular}
	}	
\end{table}

\begin{table}[H]
	\normalsize
	\caption{History of convergence results with  $k = 2$ for Example \ref{exam1}}\label{linetablek=4}
	\centering
	\footnotesize
{
		\begin{tabular}{p{1.2cm}<{\centering}|p{1.5cm}<{\centering}|p{1.4cm}<{\centering}|p{1.5cm}<{\centering}
|p{1.4cm}<{\centering}|p{1.5cm}<{\centering}|p{1.4cm}<{\centering}|p{1.5cm}<{\centering}}
			\hline
			\multirow{2}{*}{mesh}&
			\multicolumn{2}{c|}{$\frac{\|\mathbf{u}-\mathbf{u}_{ho}\|_0}{\|\mathbf{u}\|_0}$}
&\multicolumn{2}{c|}{$\frac{\|\nabla\mathbf{u}-\nabla_{w,k-1}\mathbf{u}_{h}\|_0}{\|\nabla\mathbf{u}\|_0}$}
&\multicolumn{2}{c|}{$\frac{\|\nabla\mathbf{u}-\nabla_{h}\mathbf{u}_{ho}\|_0}{\|\nabla\mathbf{u}\|_0}$}
&\multirow{2}{*}{$div Uh$}
\cr\cline{2-7}
			&error&order&error&order&error&order\cr
			\cline{1-8}
			
$2\times 2$  & 4.0948e-01  &        -   &4.5399e-01 &     - & 8.7670e-01  &        - &1.9385e-17
\\
\hline
$4\times 4$  & 5.6888e-02 &  2.85 &  1.3014e-01   &1.80 &  2.4354e-01  & 1.84    &8.4007e-15
\\
\hline
$8\times 8$  & 7.3768e-03 &  2.95 &  3.4894e-02  & 1.89  & 6.2798e-02 &  1.95 &4.6346e-15
\\
\hline
$16\times 16$  & 9.3177e-04 &  2.99 &  8.9496e-03  & 1.96 & 1.5637e-02  & 2.00 &1.7873e-15
\\
\hline
$32\times 32$  & 1.1713e-04 &  2.99  & 2.2591e-03 &  1.99 &  3.8806e-03 &  2.01 &9.7438e-15
\\
\hline
$64\times 64$  & 1.4693e-05&   3.00 &  5.6704e-04  & 2.00 & 9.6538e-04  & 2.00  &5.2180e-14
\\
\hline
$128\times 128$  & 1.8403e-06 &  2.99 &  1.4201e-04 &  2.00 &  2.4068e-04  & 2.00 &6.9921e-13
\\
\hline

		\end{tabular}
	}
	
{
		\begin{tabular}{p{1.2cm}<{\centering}|p{1.5cm}<{\centering}|p{1.4cm}<{\centering}|p{1.5cm}<{\centering}
|p{1.4cm}<{\centering}|p{1.5cm}<{\centering}|p{1.4cm}<{\centering}|p{1.5cm}<{\centering}}
            \hline
			\multirow{2}{*}{mesh}&
			\multicolumn{2}{c|}{$\frac{\|\mathbf{B}-\mathbf{B}_{ho}\|_0}{\|\mathbf{B}\|_0}$}
&\multicolumn{2}{c|}{$\frac{\|\nabla\times\mathbf{B}-\nabla_{w,k-1}\times\mathbf{B}_{h}\|_0}{\|\nabla\mathbf{B}\|_0}$}
&\multicolumn{2}{c|}{$\frac{\|\nabla\times\mathbf{B}-\nabla_{h}\times\mathbf{B}_{ho}\|_0}{\|\nabla\mathbf{B}\|_0}$}
&\multirow{2}{*}{$div Bh$}
\cr\cline{2-7}
			&error&order&error&order&error&order\cr
			\cline{1-8}
$2\times 2$  & 7.3927e-01   &       -  &5.9051e-01  &        -  & 6.4000e+00  &       - &3.8289e-16
\\
\hline
$4\times 4$  & 1.9191e-01  & 1.95  & 1.2959e-01 &  2.19  & 1.3377e+00 &  2.26 &1.8881e-16
\\
\hline
$8\times 8$ &  3.2161e-02  & 2.58  & 2.8450e-02 &  2.19   &3.3459e-01  & 2.00 &2.0293e-15
\\
\hline
$16\times 16$ &  4.5933e-03 &  2.81 &  7.0130e-03   &2.02 &  8.5809e-02&   1.96 &2.9827e-15
\\
\hline
$32\times 32$  & 6.0891e-04  & 2.92 &  1.7591e-03  & 2.00 &  2.1581e-02 &  1.99&5.4936e-15
\\
\hline
$64\times 64$  & 7.8185e-05 &  2.96  & 4.4098e-04  & 2.00 &  5.3937e-03   &2.00 &6.0488e-14
\\
\hline
$128\times 128$ &  9.8991e-06 &  2.98 &  1.1041e-04 &  2.00 &  1.3468e-03 &  2.00 &6.7897e-13
\\
\hline
\end{tabular}
}
{
		\begin{tabular}{p{1.2cm}<{\centering}|p{1.5cm}<{\centering}|p{0.9cm}<{\centering}|p{1.5cm}<{\centering}
|p{0.9cm}<{\centering}|p{1.5cm}<{\centering}|p{1.0cm}<{\centering}|p{1.5cm}<{\centering}|p{1.0cm}<{\centering}}
     \hline
			\multirow{2}{*}{mesh}&
			\multicolumn{2}{c|}{$\frac{\| p-p_{ho}\|_0}{\| p\|_0}$}
&\multicolumn{2}{c|}{$\frac{h\|\nabla p-\nabla_{w,k}p_{h}\|_0}{\|\nabla r\|_0}$}
&\multicolumn{2}{c|}{$\frac{\|r-r_{ho}\|_0}{\| r\|_0}$}
&\multicolumn{2}{c}{$\frac{h\|\nabla r-\nabla_{w,k}r_{h}\|_0}{\|\nabla r\|_0}$}
\cr\cline{2-9}
			&error&order&error&order&error&order&error&order\cr
			\cline{1-9}
$ 2\times 2 $ & 6.2098e+00       &   -&  1.6225e+01    &      - & 6.4000e+00  &       - & 1.6227e+01   &       - 
\\
\hline
$ 4\times 4 $ &  6.6197e-02  & 1.69 & 3.0672e+00  & 1.99 &1.3377e+00 &  2.26  & 4.0583e+00 &  1.99 
\\
\hline
$ 8\times 8 $  & 1.7485e-02   &1.92 &  1.8913e+00 &  1.93  &3.3459e-01  & 2.00 &1.0203e+00  & 1.99 
\\
\hline
$ 16\times 16 $  & 4.4316e-03  & 1.98&  3.8976e-01 &  1.96 &  5.3937e-02   &2.00&  2.6077e-01 &  1.96 
\\
\hline
$ 32\times 32 $ &  1.1117e-03 &  1.99 &6.2237e-02 &  1.90  &  8.5809e-02&   1.96 & 7.0588e-02 &  1.88 
\\
\hline
$ 64\times 64 $ &  2.7815e-04 &  1.99&  3.0583e-02  & 1.91&  2.1581e-02 &  1.99 &  2.2239e-02 &  1.89
\\
\hline
$ 128\times 128 $ &  6.9553e-05 &  1.99  &  8.7568e-03 &  1.89  &  5.3937e-03   &2.00& 8.7574e-03 &  1.94 
\\
\hline
		\end{tabular}
	}
\end{table}

\begin{table}[H]
	\normalsize
	\caption{History of convergence results with  $k = 1$ for Example \ref{exam2}}\label{linetablek=5}
	\centering
	\footnotesize
{
		\begin{tabular}{p{1.2cm}<{\centering}|p{1.5cm}<{\centering}|p{1.4cm}<{\centering}|p{1.5cm}<{\centering}
|p{1.4cm}<{\centering}|p{1.5cm}<{\centering}|p{1.4cm}<{\centering}|p{1.5cm}<{\centering}}
			\hline
			\multirow{2}{*}{mesh}&
			\multicolumn{2}{c|}{$\frac{\|\mathbf{u}-\mathbf{u}_{ho}\|_0}{\|\mathbf{u}\|_0}$}
&\multicolumn{2}{c|}{$\frac{\|\nabla\mathbf{u}-\nabla_{w,k-1}\mathbf{u}_{h}\|_0}{\|\nabla\mathbf{u}\|_0}$}
&\multicolumn{2}{c|}{$\frac{\|\nabla\mathbf{u}-\nabla_{h}\mathbf{u}_{ho}\|_0}{\|\nabla\mathbf{u}\|_0}$}
&\multirow{2}{*}{$div Uh$}
\cr\cline{2-7}
			&error&order&error&order&error&order\cr
			\cline{1-8}
$ 2\times 2 $ & 5.8720e-01   &       - &  5.0386e-01   &       - & 7.5274e-01     &     -& 4.4409e-16
\\
\hline
$ 4\times 4 $ &  1.5498e-01  & 1.92 &  2.6999e-01 &  0.90 &  3.5470e-01 &  1.08 &1.7764e-15
\\
\hline
$ 8\times 8 $ &  3.9245e-02  & 1.98 &  1.3795e-01 &  0.96 &  1.7228e-01 &  1.04 &2.1316e-14
\\
\hline
$ 16\times 16 $ &  9.8383e-03 &  1.99  & 6.9412e-02  & 0.99 & 8.5323e-02  & 1.01 &2.8422e-14
\\
\hline
$ 32\times 32 $ &  2.4610e-03 &  1.99  & 3.4774e-02 &  0.99   &4.2550e-02 &  1.00 &1.1369e-13
\\
\hline
$ 64\times 64 $ &  6.1533e-04  & 1.99  & 1.7399e-02  & 0.99  &  2.1260e-02  & 1.00 &9.5497e-12
\\
\hline
$ 128\times 128 $  & 1.5384e-04  & 2.00 &  8.7017e-03 &  0.99 &  1.0628e-02 &  1.00 &1.0368e-10
\\
\hline

		\end{tabular}
	}
	
{
		\begin{tabular}{p{1.2cm}<{\centering}|p{1.5cm}<{\centering}|p{1.4cm}<{\centering}|p{1.5cm}<{\centering}
|p{1.4cm}<{\centering}|p{1.5cm}<{\centering}|p{1.4cm}<{\centering}|p{1.5cm}<{\centering}}
            \hline
			\multirow{2}{*}{mesh}&
			\multicolumn{2}{c|}{$\frac{\|\mathbf{B}-\mathbf{B}_{ho}\|_0}{\|\mathbf{B}\|_0}$}
&\multicolumn{2}{c|}{$\frac{\|\nabla\times\mathbf{B}-\nabla_{w,k-1}\times\mathbf{B}_{h}\|_0}{\|\nabla\mathbf{B}\|_0}$}
&\multicolumn{2}{c|}{$\frac{\|\nabla\times\mathbf{B}-\nabla_{h}\times\mathbf{B}_{ho}\|_0}{\|\nabla\mathbf{B}\|_0}$}
&\multirow{2}{*}{$div Bh$}
\cr\cline{2-7}
			&error&order&error&order&error&order\cr
			\cline{1-8}
$ 2\times 2 $ &  4.7546e+00    &      -  & 9.9565e-01    &      - &  1.5813e+00 &   - &  1.7347e-17
\\
\hline
$ 4\times 4 $  & 1.1650e+00 &  2.02 &  5.6617e-01 &  0.81 & 8.2285e-01 &  0.94 & 3.4694e-18
\\
\hline
$ 8\times 8 $ &  2.9537e-01  & 1.97  & 2.7683e-01 &  1.03  & 3.7935e-01 &  1.11 & 3.1919e-16
\\
\hline
$ 16\times 16 $ &  7.4159e-02  & 1.99  & 1.3199e-01 &  1.06 &  1.6468e-01 &  1.20 &  9.7145e-17
\\
\hline
$ 32\times 32 $  & 1.8545e-02  & 1.99  & 6.4625e-02  & 1.03  & 7.5899e-02 &  1.11 & 1.4218e-14
\\
\hline
$ 64\times 64 $ &  4.6359e-03 &  2.00 &  3.2115e-02  & 1.00 &  3.6869e-02  & 1.04  &  5.5865e-14
\\
\hline
$ 128\times 128 $  & 1.1589e-03 &  2.00 &  1.6032e-02 &  1.00&   1.8283e-02 &  1.01 &  5.3192e-13
\\
\hline
		\end{tabular}
	}
	{
		\begin{tabular}{p{1.2cm}<{\centering}|p{1.5cm}<{\centering}|p{0.9cm}<{\centering}|p{1.5cm}<{\centering}
|p{0.9cm}<{\centering}|p{1.5cm}<{\centering}|p{1.0cm}<{\centering}|p{1.5cm}<{\centering}|p{1.0cm}<{\centering}}
     \hline
			\multirow{2}{*}{mesh}&
			\multicolumn{2}{c|}{$\frac{\| p-p_{ho}\|_0}{\| p\|_0}$}
&\multicolumn{2}{c|}{$\frac{h\|\nabla p-\nabla_{w,k}p_{h}\|_0}{\|\nabla r\|_0}$}
&\multicolumn{2}{c|}{$\frac{\|r-r_{ho}\|_0}{\| r\|_0}$}
&\multicolumn{2}{c}{$\frac{h\|\nabla r-\nabla_{w,k}r_{h}\|_0}{\|\nabla r\|_0}$}
\cr\cline{2-9}
			&error&order&error&order&error&order&error&order\cr
			\cline{1-9}
$ 2\times 2 $ & 1.5904e+00  &        -&  5.0377e-01   &       - &  4.8444e+00  & -&  7.4476e-01 &       -
\\
\hline
$ 4\times 4 $  &  8.2249e-01 &  0.95&2.4127e-01 &  1.06 &  3.5434e-01  &     -0.84&  4.8998e-01 &  0.61 \\
\hline
$ 8\times 8 $ &  4.0154e-01 &  1.03&  1.2071e-01  & 0.99&  5.1369e-01 &  0.82 &  2.1646e-01 &  1.17 \\
\hline
$ 16\times 16 $ & 1.9969e-01 &  1.00&6.0868e-02 &  0.98 & 2.9593e-01  &  0.80&  9.4338e-02 &  1.19 \\
\hline
$ 32\times 32 $ & 9.9854e-02  & 0.99& 3.0661e-02 &  0.98 &  1.8125e-01 &   1.01&  4.6403e-02  & 1.02 \\
\hline
$ 64\times 64 $ &  4.9969e-02  & 0.99&  1.5415e-02 &  0.99  & 7.2736e-02  &  1.01&  2.5885e-02 &  0.90\\
\hline
$ 128\times 128 $ &  2.5000e-02 &  0.99 & 7.7379e-03   &0.99 &  3.6275e-02  &  1.00 & 1.5930e-02&   0.8 \\
\hline
		\end{tabular}
	}	
\end{table}

\begin{table}[H]
	\normalsize
	\caption{History of convergence results with  $k = 2$ for Example \ref{exam2}}\label{linetablek=6}
	\centering
	\footnotesize
{
		\begin{tabular}{p{1.2cm}<{\centering}|p{1.5cm}<{\centering}|p{1.4cm}<{\centering}|p{1.5cm}<{\centering}
|p{1.4cm}<{\centering}|p{1.5cm}<{\centering}|p{1.4cm}<{\centering}|p{1.5cm}<{\centering}}
			\hline
			\multirow{2}{*}{mesh}&
			\multicolumn{2}{c|}{$\frac{\|\mathbf{u}-\mathbf{u}_{ho}\|_0}{\|\mathbf{u}\|_0}$}
&\multicolumn{2}{c|}{$\frac{\|\nabla\mathbf{u}-\nabla_{w,k-1}\mathbf{u}_{h}\|_0}{\|\nabla\mathbf{u}\|_0}$}
&\multicolumn{2}{c|}{$\frac{\|\nabla\mathbf{u}-\nabla_{h}\mathbf{u}_{ho}\|_0}{\|\nabla\mathbf{u}\|_0}$}
&\multirow{2}{*}{$div Uh$}
\cr\cline{2-7}
			&error&order&error&order&error&order\cr
			\cline{1-8}
			
$2\times 2$  & 1.1135e-01    &      - &  1.6761e-01    &      -  & 4.3095e-01    &     - &4.2931e-15
\\
\hline
$4\times 4$  & 1.4860e-02 &  2.91 &  4.3804e-02 &  1.94 &  1.0261e-01 &  2.07 &1.3417e-14
\\
\hline
$8\times 8$  & 1.8798e-03 &  2.99  & 1.1123e-02  & 1.98 &  2.4719e-02  & 2.05 &1.1394e-13
\\
\hline
$16\times 16$ &  2.3542e-04  & 3.00 &  2.7981e-03 &  1.99 & 6.1043e-03 &  2.01 &1.8455e-13
\\
\hline
$32\times 32$ &  2.9435e-05 &  3.00 &  7.0135e-04  & 2.00 &  1.5210e-03 &  2.00 &1.1427e-12
\\
\hline
$64\times 64$ &  3.6793e-06  & 3.00  & 1.7554e-04 &  2.00 & 3.7993e-04 &  2.00  &6.1542e-12
\\
\hline
$128\times 128$ &  4.5990e-07  & 3.00&  4.3909e-05 &  2.00  & 9.4962e-05 &  2.00 &1.8475e-10
\\
\hline

		\end{tabular}
	}
	
{
		\begin{tabular}{p{1.2cm}<{\centering}|p{1.5cm}<{\centering}|p{1.4cm}<{\centering}|p{1.5cm}<{\centering}
|p{1.4cm}<{\centering}|p{1.5cm}<{\centering}|p{1.4cm}<{\centering}|p{1.5cm}<{\centering}}
            \hline
			\multirow{2}{*}{mesh}&
			\multicolumn{2}{c|}{$\frac{\|\mathbf{B}-\mathbf{B}_{ho}\|_0}{\|\mathbf{B}\|_0}$}
&\multicolumn{2}{c|}{$\frac{\|\nabla\times\mathbf{B}-\nabla_{w,k-1}\times\mathbf{B}_{h}\|_0}{\|\nabla\mathbf{B}\|_0}$}
&\multicolumn{2}{c|}{$\frac{\|\nabla\times\mathbf{B}-\nabla_{h}\times\mathbf{B}_{ho}\|_0}{\|\nabla\mathbf{B}\|_0}$}
&\multirow{2}{*}{$div Bh$}
\cr\cline{2-7}
			&error&order&error&order&error&order\cr
			\cline{1-8}
$2\times 2$  & 7.3926e-01   &       -  & 5.9072e-01  &        - &  1.1218e+00  &  -& 5.0090e-17
\\
\hline
$4\times 4$  & 1.9145e-01  & 1.95 &  1.3030e-01 &  2.18 &  4.3639e-01  & 1.36 &2.8295e-16
\\
\hline
$8\times 8$ &  3.2158e-02 &  2.57  & 2.8597e-02 &  2.18  & 1.3865e-01 &  1.65 &4.8643e-16
\\
\hline
$16\times 16$ &  4.6052e-03  & 2.80 &  7.0452e-03  & 2.02 &  3.8347e-02  & 1.85 &1.3153e-15
\\
\hline
$32\times 32$  & 6.1557e-04 &  2.90 &  1.7668e-03  & 2.00 &  9.9908e-03 &  1.94 &2.0630e-14
\\
\hline
$64\times 64$  & 8.1469e-05 &  2.92 &  4.4291e-04 &  2.00  & 2.5434e-03  & 1.97 &1.1610e-13
\\
\hline
$128\times 128$ &  1.1444e-05 &  2.83 &  1.1089e-04  & 2.00 &  6.4124e-04 &  1.98 &7.3211e-13
\\
\hline

		\end{tabular}
	}
{
		\begin{tabular}{p{1.2cm}<{\centering}|p{1.5cm}<{\centering}|p{0.9cm}<{\centering}|p{1.5cm}<{\centering}
|p{0.9cm}<{\centering}|p{1.5cm}<{\centering}|p{1.0cm}<{\centering}|p{1.5cm}<{\centering}|p{1.0cm}<{\centering}}
     \hline
			\multirow{2}{*}{mesh}&
			\multicolumn{2}{c|}{$\frac{\| p-p_{ho}\|_0}{\| p\|_0}$}
&\multicolumn{2}{c|}{$\frac{h\|\nabla p-\nabla_{w,k}p_{h}\|_0}{\|\nabla r\|_0}$}
&\multicolumn{2}{c|}{$\frac{\|r-r_{ho}\|_0}{\| r\|_0}$}
&\multicolumn{2}{c}{$\frac{h\|\nabla r-\nabla_{w,k}r_{h}\|_0}{\|\nabla r\|_0}$}
\cr\cline{2-9}
			&error&order&error&order&error&order&error&order\cr
			\cline{1-9}
$ 2\times 2 $ & 7.9002e-01     &     -&  4.8151e-01    &   -&6.4267e+00  &   -  & 3.5362e-01     &     - \\
\hline
$ 4\times 4 $ &  1.4304e-01 &  2.47 & 2.4399e-01 &  1.98 &  1.6916e+00  & 1.93 &  8.8431e-02  & 1.99 \\
\hline
$ 8\times 8 $ &  2.9207e-02&   2.30&  1.2339e-01 &  1.98 & 4.8829e-01  & 1.79&  2.2139e-02 &  1.99 \\
\hline
$ 16\times 16 $ &  6.6988e-03 &  2.12&6.2086e-02  & 1.99&  1.2901e-01 &  1.92 &  5.5663e-03   & 1.97 \\
\hline
$ 32\times 32 $ &  1.6187e-03 &  2.05& 3.1145e-02 &  1.99&  3.2709e-02 &  1.98 &  7.0588e-02 &  1.90 \\
\hline
$ 64\times 64 $ &  3.9905e-04 &  2.02&  1.5599e-02 &  1.99& 8.1999e-03  & 2.00 &  2.2239e-02 &  1.89\\
\hline
$ 128\times 128 $ & 9.9147e-05  & 2.01 & 7.8059e-03 &  1.99&  2.0504e-03  & 2.00 &  8.7574e-03 &  1.89 \\
\hline
		\end{tabular}
	}	
\end{table}

\section{Conclusions}

In this paper, we have developed a weak Galerkin     method of arbitrary order
 for the steady incompressible Magnetohydrodynamics flow.  The well-posedness of the discrete scheme has been established.
 The method yields globally divergence-free  approximations of velocity and magnetic field, and is of optimal order convergence for the velocity, the magnetic field, the pressure, and the magnetic pseudo-pressure approximations.
 The proposed Oseen's iteration algorithm is unconditionally convergent. Numerical experiments have verified the theoretical results.


\begin{thebibliography}{99}


\bibitem{ABLR2010}
F. Auricchio, L. Beir\~{a}o da Veiga, C. Lovadina, and A. Reali,
The importance of the exact satisfaction of the incompressibility constraint in nonlinear elasticity: mixed FEMs versus NURBS-based approximations.
Computer Methods in Applied Mechanics and Engineering,
199:314-323, 2010.

\bibitem{BD2015}
D.S. Balsara and M. Dumbser,
Divergence-free MHD on unstructured meshes using high order finite volume schemes based on multidimensional Riemann solvers.
Journal of Computational Physics,
299:687-615, 2015.



\bibitem{BBKS1999}
V. Bandaru, T. Boeck, D. Krasnov, and J. Schumacher,
Numerical computation of liquid metal MHD duct flows at finite magnetic Reynolds number.
pamir.sal.lv, 1999.


\bibitem{BB1980}
J.U. Brackbill and D.C. Barnes,
The effect of nonzero $\nabla \cdot \mathbf{B}$ on the numerical solution of the magnetohydrodynamic equations.
Journal of Computational Physics,
35:426-430, 1980.


\bibitem{BLS2007}
S.C. Brenner, F. Li, and L. Sung,
A locally divergence-free nonconforming finite element method for the time-harmonic Maxwell equations.
Mathematics of Computation,
76:573-595, 2007.



\bibitem{B2010}H. Brezis,
Functional analysis, Sobolev spaces and partial differential equations.
Springer Science and Business Media, 2010.


\bibitem{BBDDFF2008}F. Brezzi, D. Boffi, L. Demkowicz, R.G. Dur\'{a}n, R.S. Falk, and M. Fortin,
Mixed finite elements, compatibility conditions, and applications. Springer, 2008.


\bibitem{CFX2016}G. Chen, M. Feng, and X. Xie,
Robust globally divergrence-free weak Galerkin methods for Stokes equations.
Journal of Computational Mathematics,
34:549-572, 2016.


\bibitem{CFX2017} G. Chen, M. Feng, and X. Xie,
A robust WG finite element method for convection-diffusion-reaction equations.
Journal of Computational and Applied Mathematics,
315:107-125, 2017.


\bibitem{CX2016} G. Chen and X. Xie,
A robust weak Galerkin finite element method for linear elasticity with strong symmetric stresses.
Computational Methods in Applied Mathematics,
16:389-408, 2016.

\bibitem{Chen-Xie2023}
G. Chen and X. Xie, Analysis of a class of globally divergence-free HDG methods for stationary Navier-Stokes equations. Science China-Mathematics, 2023, 66, https://doi.org/10.1007/s11425-022-2077-7





\bibitem{CWWY2015} L. Chen, J. Wang, Y. Wang, and X. Ye,
An auxiliary space multigrid preconditioner for the weak Galerkin method.
Computers and Mathematics with Applications,
70:330-344, 2015.




\bibitem{CKS2007}
B. Cockburn, G. Kanschat, and D. Sch\"{o}tzau,
A note on discontinuous Galerkin divergence-free solutions of the Navier-Stokes Equations.
Journal of Scientific Computing,
31:61-73, 2007.




\bibitem{CLS2004}
B. Cockburn, F. Li, and C. Shu,
Locally divergence-free discontinuous Galerkin methods for the Maxwell equations.
Journal of Computational Physics,
194:588-610, 2004.



\bibitem{C1996}
H. Conraths,
Eddy current and temperature simulation in thin moving metal strips.
International Journal for Numerical Methods in Engineering,
39:141-163, 1996.



\bibitem{Davidson2001} P.A. Davidson
An Introduction to Magnetohydrodynamics
Cambridge University Press, 2001.


\bibitem{D2018} B. Deka,
A weak galerkin finite element method for elliptic interface problems with polynomial reduction. Numerical Mathematics-Theory Methods and Applications,
11:655-672, 2018.



\bibitem{DLM2020} Q. Ding, X. Long, and S. Mao,
Convergence analysis of Crank-Nicolson extrapolated fully discrete scheme
for thermally coupled incompressible magnetohydrodynamic system.
Applied Numerical Mathematics,
157:522-543, 2020.






\bibitem{DH2018} X. Dong, and Y. He,
Optimal convergence analysis of Crank-Nicolson extrapolation scheme for the three-dimensional incompressible magnetohydrodynamics.
Computers and Mathematics with Applications,
76:2678-2700, 2018.

\bibitem{FS2018}K.G. Felker and J.M. Stone,
A fourth-order accurate finite volume method for ideal MHD via upwind constrained transport.
Journal of Computational Physics,
375:1365-1400, 2018.


\bibitem{GQ2019} H. Gao and  W. Qiu,
A semi-implicit energy conserving finite element method for the dynamical incompressible magnetohydrodynamics equations.
Computer Methods in Applied Mechanics and Engineering,
346:982-1001, 2019.



\bibitem{G2000} J.F. Gerbeau,
A stabilized finite element method for the incompressible magnetohydrodynamic equations.
Numerische Mathematik,
87:83-111, 2000.


\bibitem{GLL2006} J.F. Gerbeau, C. Le Bris, T. Leli$\grave{e}$vre
Mathematical Methods for the Magnetohydrodynamics of Liquid Metals. Numerical Mathematics and Scientific Computation, Oxford University Press, New York, 2006.

\bibitem{GLSW2010} C. Greif, D. Li, D. Schozau, and X. Wei,
A mixed finite element method with exactly divergence-free velocities for incompressible magnetohydrodynamics.
Computer Methods in Applied Mechanics and Engineering,
199:2840-2855, 2010.



\bibitem{GM2003} J.L. Guermond and P.D. Minev,
Mixed finite element approximation of an MHD problem involving conducting and insulating regions: the 3D case.
Numerical Methods for Partial Differential Equations,
19:709-731, 2003.


\bibitem{GMP1991} M.D. Gunzburger, A.J. Meir, and J.S. Peterson,
On the existence, uniquess and finite element approximation of solutions of the equations of sationary, incompressible magnetohydrodynamic.
Mathematics of Computation,
56:523-563, 1991.


\bibitem{HLX2019}
Y. Han, H. Li, and X. Xie,
Robust globally divergence-free weak Galerkin finite element methods for unsteady natural
convection problems.
Numerical Mathematics: Theory, Methods and Applications,
12:1266-1308, 2019.



\bibitem{HX2019}Y. Han and X. Xie,
Robust globally divergence-free weak Galerkin finite element methods for natural convection problems. Communications in Computational Physics,
26:1039-1070, 2019.


\bibitem{H2015} Y. He,
Unconditional convergence of the Euler semi-implicit scheme for the three-dimensional incompressible MHD equations.
IMA Journal of Numerical Analysis,
35:767-801, 2015.





\bibitem{HLMZ2018} R. Hiptmair, L. Li, S. Mao, and W. Zheng,
A fully divergence-free finite element method for magnetohydrodynamic equations.
Mathematical Models and Methods in Applied Sciences,
28:659-695, 2018.

\bibitem{HMX2017} K. Hu, Y. Ma, and J. Xu,
Stable finite element methods preserving $\nabla\cdot \mathbf{B} = 0$ exactly for MHD models.
Numerische Mathematik,
135:371-396, 2017.

\bibitem{HZ2012}
J. Huang and S. Zhang,
A divergence-free finite element method for a type of 3D Maxwell equations.
Applied Numerical Mathematics,
62:802-813, 2012.

\bibitem{J2010}
S. Jardin,
Computational Methods in Plasma Physics. CRC Press, Boca Raton, 2010.


\bibitem{JWP1996}
B.-N. Jiang, J. Wu and L.A. Povinelli,
The origin of spurious solutions in computational electromagnetics.
Journal of Computational Physics,
125:104-123, 1996.



\bibitem{JLMNR2017} V. John, A. Linke, C. Merdon, M. Neilan, and  L.G. Rebholz,
On the divergence constraint in mixed finite element methods for incompressible flows.
SIAM Review,
59:492-544, 2017.

\bibitem{KP2007}O.A. Karakashian and F. Pascal,
Convergence of adaptive discontinuous Galerkin approximations of second-order elliptic problems.
SIAM Journal on Numerical Analysis,
45:641-665, 2007.


\bibitem{LX2015} B. Li and X. Xie,
A two-level algorithm for the weak Galerkin discretization of diffusion problems.
Journal of Computational and Applied Mathematics,
287:179-195, 2015.

\bibitem{LX2016} B. Li and X. Xie,
BPX pre-conditioner for nonstandard finite element methods for diffusion problems.
SIAM Journal on Numerical Analysis,
54:1147-1168, 2016.




\bibitem{LS2005}
F. Li and C. Shu,
Locally divergence-free discontinuous Galerkin methods for the MHD equations.
Journal  Scientific Computing,
22-23:413-442, 2005.


\bibitem{LX2012}
F. Li and L. Xu,
Arbitrary order exactly divergence-free central discontinuous Galerkin methods for ideal MHD equations.
Journal of Computational Physics,
231:2655-2675, 2012.

\bibitem{LXY2011}
F. Li, L. Xu, and S. Yakovlev,
Central discontinuous Galerkin methods for ideal MHD equations with the exactly divergence-free magnetic field.
Journal of Computational Physics,
230:4828-4847, 2011.


\bibitem{LZZ2021}
L. Li, D. Zhang, and W. Zheng,
A constrained transport divergence-free finite element method for incompressible MHD equations.
Journal of Computational Physics,
428:109980, 2021.


\bibitem{L2009}
A. Linke,
Collision in a cross-shaped domain-a steady 2d Navier-Stokes example demonstrating the importance of mass conservation in CFD.
 Computer Methods in Applied Mechanics and Engineering,
198:3278-3286, 2009.



\bibitem{LM2016}
A. Linke and C. Merdon, Pressure-robustness and discrete
Helmholtz projectors in mixed finite element methods for the incompressible Navier-Stokes equations. Computer Methods in Applied Mechanics and Engineering, 311:
304-326, 2016.






\bibitem{MHHX2016}Y. Ma, K. Hu, X. Hu, and J. Xu,
Robust preconditioners for incompressible MHD models.
Journal of Computational Physics,
316:721-746, 2016.


\bibitem{MWYZ2015}L. Mu, J. Wang, X. Ye, and S. Zhang,
A weak Galerkin finite element method for the Maxwell equations.
Journal of Scientific Computing, 65:363-386, 2015.


\bibitem{Muller2001} U. M$\ddot{u}$ller, L. B$\ddot{u}$hler
Magnetofluiddynamics in Channels and Containers,
Springer Berlin, 2001.


\bibitem{OR2004}
M.A. Olshanskii and A. Reusken,
Grad-div stabilization for Stokes equations.
Mathematics of Computation,
73:1699-1718, 2004.



\bibitem{PT2013}
B. Pekmen and  M. Tezer-Sezgin,
DRBEM solution of incompressible MHD flow with magnetic potential.
Computer Modeling Engineering and Sciences.
96:275-292, 2013.





\bibitem{Peterson1988} J.S. Peterson,
On the finite element approximation of incompressible flows of an electrically
conducting fluid.
Numerical Methods for Partial Differential Equations, 4:57-68, 1988.


\bibitem{P2008} A. Prohl,
Convergent finite element discretizations of the nonstationary incompressible magnetohydrodynamics system.
ESAIM Mathematical Modelling and Numerical Analysis,
42:1065-1087, 2008.

\bibitem{QS2020}W. Qiu and K. Shi,
A mixed DG method and an HGD method for incompressible magnetohydrodynamics.
IMA Journal of Numerical Analysis,
40:1356-1389, 2020.

\bibitem{R2018} S.S. Ravindran,
A decoupled Crank-Nicolson time-stepping scheme for thermally coupled magneto-hydrodynamic system.
An International Journal of Optimization and Control: Theories and Applications,
8:43-62, 2018.

\bibitem{S2004}D. Schozau,
Mixed finite element methods for stationary incompressible magnetohydrodynamics.
Numerische Mathematik,
96:771-800, 2004.




\bibitem{SPBCLT2010}
J. Shadid, R. Pawlowski, J. Banks, L. Chacon, P. Lin, and R. Tuminaro,
Towards a scalable fullyimplicit fully-coupled resistive MHD formulation with stabilized FE methods.
Journal Of Computational Physics,
229:7649-7671, 2010.


\bibitem{Shercliff1965} J.A. Shercliff, A textbook of magnetohydrodynamics. Pergamon Press, Oxford, 1965.


\bibitem{SY2013}  D. Shi and Z. Yu,
Nonconforming mixed finite element methods for stationary incompressible magnetohydrodynamics. International Journal of Numerical Analysis and Modeling,
    10:904-919, 2013.



\bibitem{Shi;Wang2013}
 Z. Shi and M. Wang. Finite element methods. Science Press, 2013.


\bibitem{SMF2019} H. Su, S. Mao, and X. Feng,
Optimal error estimates of penalty based iterative methods for steady incompressible magnetohydrodynamics equations with different viscosities.
Journal of Scientific Computing,
79:1078-1110, 2019.


\bibitem{T2000}
G. T\'{o}th,
The $\nabla\cdot\mathbf{B} = 0$ constraint in shock-capturing Magnetohydrodynamics codes.
Journal of Computational Physics,
161:605-652, 2000.

\bibitem{Walker1980} J.S. Walker,
Large interaction parameter magnetohydrodynamics and applications in fusion reactor technology, Fluid Mechanics in Energy Conversion (J. Buckmaster, ed.), SIAM,
Philadelphia, 1980.







\bibitem{Wang;Ye2013} J. Wang and X. Ye, A weak Galerkin finite element method for second-order elliptic problems.
     Journal of Computational and Applied Mathematics,
     241:103-115, 2013.


 \bibitem{Wang;Ye2014}  J. Wang and X. Ye, A weak Galerkin mixed finite element method for second-order elliptic problems.
     Mathematics of Computation,
     83:2101-2126, 2014.


\bibitem{WY2016} J. Wang and X. Ye,
A weak Galerkin finite element method for the Stokes equations.
Advances in Computational Mathematics,
42:155-174, 2016.



\bibitem{WWZZ2016} R. Wang, X. Wang, Q. Zhai, and R. Zhang,
A weak Galerkin finite element scheme for solving the stationary Stokes equations.
Journal of Computational and Applied Mathematics,
302:171-185, 2016.

\bibitem{WWZ2018}R. Wang, X. Wang, and R. Zhang,
A weak galerkin finite element method for elliptic interface problems with polynomial reduction. Numerical Mathematics-Theory Methods and Applications,
11:518-539, 2018.






\bibitem{W2000}  M. Wiedmer,
Finite element approximation for equations of magnetohydrodynamics.
Mathematics of Computation, 69:83-101, 2000.




\bibitem{Winowich1983}
N.S. Winowich and W. Hughes, A finite element analysis of two dimensional MHD flow,
Liquid-Metal Flows and Magnetohydrodynamics (H. Branover, P.S. Lykoudis, and A.
Yakhot, eds.), AIAA, New York, 1983.


\bibitem{WLFH2017}J. Wu, D. Liu, X. Feng, and P. Huang,
An efficient two-step algorithm for the stationary incompressible magnetohydrodynamic equations.
Applied Mathematics and Computation,
302:21-33, 2017.





\bibitem{XZ2010}
X. Xu and S. Zhang,
A new divergence-free interpolation operator with applications to the Darcy-Stokes-Brinkman equations.
SIAM Journal On Scientific Computing,
32:855-874, 2010.


\bibitem{YH2018} J. Yang and Y. He,
Stability and error analysis for the first-order Euler implicit/explicit scheme for the 3D MHD equations. International Journal of Computational Methods,
15:1750077, 2018.

\bibitem{ZZMH2018} Q. Zhai, R. Zhang, N. Malluwawadu, and S. Hussain,
The weak Galerkin method for linear hyperbolic equation.
Communications in Computational Physics,
24:152-166, 2018.



\bibitem{ZHZ2014} G. Zhang, Y. He, and Y. Zhang,
Streamline diffusion finite element method for stationary incompressible magnetohydrodynamics.
Numerical Methods For Partial Differential Equations,
30:1877-1901, 2014.


\bibitem{ZYB2018}G.-D. Zhang, J. Yang, and C. Bi,
Second order unconditionally convergent and energy stable linearized scheme for MHD equations.
Advances in Computational Mathematics,
44:505-540, 2018.



\bibitem{ZZLW2018} J. Zhang, K. Zhang, J. Li, and X. Wang,
A weak Galerkin finite element method for the Navier-Stokes equations.
Communications in Computational Physics,
23:706-746, 2018.



\bibitem{ZMZ2016} J. Zhao, S. Mao, and W. Zheng,
Anisotropic adaptive finite element method for magnetohydrodynamic flow at high Hartmann numbers.
Applied Mathematics and Mechanics,
37:1479-1500, 2016.


\bibitem{ZCX2017} X. Zheng, G. Chen, and X. Xie,
A divergence-free weak Galerkin method for quasi-Newtonian Stokes flows.
Science China Mathematics,
60:1515-1528, 2017.


\bibitem{ZX2017} X. Zheng and X. Xie,
A posteriori error estimator for a weak Galerkin finite element solution of the Stokes problem.
East Asian Journal on Applied Mathematics,
7:508-529, 2017.






































%
%
%
%
%
%
%
%
%
%
%
%
%
%
%
%
%
%
%
%
%
%
%
%
%
%
%
%
%
%
%
%
%
%
%
%
%
%
%
%
%
%
%
%
%
%
%
%
%
%
%
%
%
%
%
%
%
%
%
%
%
%
%
%
%
%
%
%
%
%
%
%
%
%
%
%
%
%
%
%
%
%
%
%
%
%
%
%
%
%
%
%
%
%
%
%
%
%
%
%
%
%
%
%
%
%
%
%
%
%
%
%
%
%
%
%
%
%
%
%
%
%
%
%
%
%
%
%
%
%
%
%
%
%
%
%
%
%
%
%
%
%
%
%
%
%
%
%
%
%
%
%
%
%
%
%
%
%
%
%
%
%
%
%
%
%
%
%
%
%
%
%
%
%
%
%
%
%
%
%
%
%
%
%
%
%
%
%
%
%
%
%
%
%
%
%
%
%
%
%
%
%
%
%
%
%
%
%
%
%
%
%
%
%
%
%
%


\end{thebibliography}
\end{document}